%% file: GP121020.tex
\documentclass[a4paper,12pt]{article}
\usepackage{amsmath,amssymb,amsthm}
\usepackage{amscd}
\usepackage{color}

\makeatletter

\@addtoreset{equation}{subsection}
\makeatother

\newtheorem{dfn}{Definition}[subsection]
\newtheorem{prop}[dfn]{Proposition}
\newtheorem{thm}[dfn]{Theorem}
\newtheorem{lem}[dfn]{Lemma}
\newtheorem{cor}[dfn]{Corollary}

\newcommand\abel{{{\rm abel}}}
\newcommand\ASE{{A(S,E)}}
\newcommand\Aut{{\rm Aut}}
\newcommand\CC{{\mathcal{C}}}
\newcommand\CSE{{\mathcal{C}(S,E)}}
\newcommand\Der{{\rm Der}}
\newcommand\GG{{\mathcal{G}}}
\newcommand\Hom{{\rm Hom}}
\newcommand\ISE{{\mathcal{I}(S,E)}}
\newcommand\Ker{{\rm Ker}}
\newcommand\Ob{{\rm Ob}}
\newcommand\RR{{\mathcal{R}}}

\begin{document}

\title
{Groupoid-theoretical methods in the mapping class groups of surfaces}
\author{Nariya Kawazumi and Yusuke Kuno}
\maketitle

\begin{abstract} 
We provide some language for algebraic study of the mapping class groups 
for surfaces with non-connected boundary, where the Goldman Lie algebra
plays a central role. As applications, we generalize
our previous results on Dehn twists in \cite{KK1} and \cite{Ku2} to any
compact connected oriented surfaces with non-empty boundary. Moreover we
embed the `smallest' Torelli group in the sense of Putman \cite{P} into a
pro-nilpotent group coming from the Goldman Lie algebra. The graded
quotients of the embedding equal  the Johnson homomorphisms of all
degrees if the boundary is connected. 
\end{abstract}

\maketitle

\begin{center}
{\large Introduction}
\end{center}

In algebraic study of the mapping class group of a surface,
its action on the fundamental group of the surface plays an 
essential role. Let $\Sigma_{g,r}$ be a compact connected 
oriented surface of genus $g$ with $r$ boundary components,
$r \geq 1$, $\mathcal{M}_{g,r}$ the mapping class group of 
$\Sigma_{g,r}$. Choose a basepoint $* \in \partial\Sigma_{g,r}$.
Then the action induces a group homomorphism
$$
{\sf DN}\colon \mathcal{M}_{g,r} \to \Aut\pi_1(\Sigma_{g,r}, *).
$$
The Dehn-Nielsen theorem says ${\sf DN}$ is injective if $r=1$. 
If $r \geq 2$, it is {\it not} injective. In fact, the right handed 
Dehn twist along a boundary component without the basepoint $*$ 
is a non-trivial element in the kernel of the homomorphism ${\sf DN}$. 
To study the case $r \geq 2$, we have to consider the action of 
$\mathcal{M}_{g,r}$ on a groupoid $\CC$, more precisely, a full
subcategory of  the fundamental groupoid of the surface $\Sigma_{g,r}$, 
$\Pi\Sigma_{g,r}$, whose object set has at least one point in 
each boundary component. Then, as will be shown in \S \ref{sec:3-1},
the natural homomorphism ${\sf DN}\colon \mathcal{M}_{g,r} \to \Aut\CC$ is injective.

The purpose of this paper is to provide some language for studying 
the action of the mapping class group on such groupoids. 
In our previous paper \cite{KK1}, we found that the Goldman Lie algebra of a
surface acts on the group ring of the fundamental group of the surface by
derivations.  This action plays a central role in our language. As a
consequence, we come to the notion of the completed
Goldman Lie algebra for any oriented surface. From the results in
\cite{KK1}, in the case for $\Sigma_{g,1}$, this Lie algebra includes
Kontsevich's `associative' and `Lie' as Lie subalgebras.

The first application of our method is about Dehn twists and their generalization.
In \cite{KK1} the authors gave some description 
of Dehn twists on the surface $\Sigma_{g,1}$, which led us to 
the definition of a {\it generalized Dehn twist} along a non-simple closed 
curve as an automorphism of the completed group ring of the fundamental 
group of $\Sigma_{g,1}$. In \cite{Ku2}, the second-named author proved 
that the generalized Dehn twist along a figure eight is {\it not} realized by any element 
of $\mathcal{M}_{g,1}$. We generalize all these
results to the case $r \geq 2$. In fact, Theorem 5.2.1 is a generalization
of our description of Dehn twists (\cite{KK1} Theorem 1.1.1) to any oriented surfaces,
and Theorem 5.4.2 is a generalization of the non-realizability as a diffeomorphism
of the generalized Dehn twist along a figure eight (\cite{Ku2} Theorem 5.1.1)
to any oriented surfaces of finite type with non-empty boundary.

As was shown in \cite{Ku2} Theorem 3.3.2, 
the generalized Dehn twist along a closed curve $C$ is localized inside 
a regular neighborhood of the curve $C$. In almost all cases, the regular 
neighborhood has a non-connected boundary. This also leads us to 
studying groupoids. Our groupoid-theoretical methods make all 
the arguments on (generalized) Dehn twists much shorter than
those in \cite{KK1} and \cite{Ku2}.
In our paper \cite{KK4}, we prove that the generalized Dehn twists
along non-simple closed curves in wider classes are not realized by any diffeomorphisms.
In our previous papers \cite{KK1} \cite{Ku2}, the notions as
Magnus expansions or symplectic expansions played a crucial role.
The main theorems and constructions of this paper are
basically free from them, although we have used them in several arguments.
Recently Massuyeau and Turaev \cite{MT} developed
a theory on generalized Dehn twists without use of Magnus expansions.
Comparing their approach and ours seems interesting. 

The second application of our method is about the Johnson homomorphisms on the Torelli groups.
The higher Johnson homomorphism of the Torelli group is an important
tool to study the algebraic structure of the mapping class group
$\mathcal{M}_{g,1}$. If $r \geq 2$, theory of higher Johnson
homomorphisms  for $\mathcal{M}_{g,r}$ has not been established,  
since the map ${\sf DN}$ is not injective. 
Moreover it has not been clarified what Lie algebra should be an
appropriate target of the higher Johnson homomorphism in these cases.
In \S \ref{sec:6}, we discuss the Johnson homomorphisms via our groupoid-theoretical methods.
The completed Goldman Lie algebra we introduce here gives a geometric
interpretation of a completion of a Lie algebra introduced by Morita
\cite{MoPJA} \cite{MoICM} as an appropriate target of the higher
Johnson homomorphism for $r=1$, which is also Kontsevich's `Lie' \cite{Kon}.
By generalizing this construction and using Putman's result
on generators of the Torelli groups \cite{P}, we 
embed the `smallest' Torelli group in the sense of Putman \cite{P} into a
pro-nilpotent group coming from the Goldman Lie algebra in the case
$r$ is positive. If $r=1$, the graded quotients of the embedding
equal the Johnson homomorphisms of all degrees.
Moreover, our construction has a compatibility with respect
to an inclusion of surfaces (Proposition \ref{tau-compat}).
Recently Church \cite{Chu} introduced the first Johnson homomorphism for all kinds of
Putman's partitioned Torelli groups. It would be interesting to
describe an explicit relation between Church's homomorphisms and ours. 

The results in this paper will be fundamental in our subsequent paper \cite{KK4},
where we discuss further applications of our method.

\medskip
\noindent \textbf{Acknowledgments.}
The authors would like to thank Gw\'ena\"el Massuyeau for informing
us about the paper \cite{MT}.
The first-named author is grateful to
Atsushi Matsuo for helpful discussions on giving SAC's a name, and 
Masatoshi Sato for valuable discussions. 
The first-named author is partially supported by the Grant-in-Aid for
Scientific Research (A) (No.18204002) and (B) (No.24340010) from the
Japan Society for Promotion of Sciences. The second-named
author is supported by JSPS Research Fellowships
for Young Scientists (22$\cdot$4810).

\tableofcontents

\section{Special additive categories (SAC's)}
\label{sec:1}

In this and the next sections we develop some algebraic
machineries which do not need any surface topology.
What we want to consider is {\it a groupoid version} of several
constructions attached to a group, such as the group ring
and derivations on the group ring (\S \ref{11def}),
the completed group ring with respect to
the augumentation ideal and its powers
(\S \ref{sec:1-2} and \S \ref{21filtration}),
the Hopf algebra structure on the group ring (\S \ref{21filtration}),
and the abelianization (\S \ref{sec:2-2}).

Let $\RR$ be a small category. 
Throughout this paper we denote by $\Ob\RR$ the set of
objects in $\RR$ and write $\RR(p_0, p_1)=\Hom_\RR(p_0, p_1)$
for any objects $p_0$ and $p_1 \in \Ob\RR$. 
The composition $\circ$ in $\RR$ gives a multiplication $\cdot$ in $\RR$ by 
$$
\cdot\colon \RR(p_0, p_1) \times \RR(p_1, p_2) \to \RR(p_0, p_2), \quad
\gamma_1\gamma_2 = \gamma_1\cdot\gamma_2 := \gamma_2\circ\gamma_1.
$$

\subsection{Definition of special additive categories}
\label{11def}

Let $K$ be a commutative ring with unit.
\begin{dfn} A small additive category $\RR$ is called a {\it $K$-special 
additive category} {\rm (}$K$-SAC{\rm )} if it satisfies
\begin{enumerate}
\item[{\rm (i)}] For any $p_0, p_1$ and $p_2 \in \Ob\RR$, the additive group 
$\RR(p_0, p_1)$ is a $K$-module, and the multiplication $\cdot\colon \RR(p_0,
p_1)\times \RR(p_1, p_2) \to \RR(p_0, p_2)$ is $K$-bilinear. 
\item[{\rm (ii)}] For any $p_0$ and $p_1 \in \Ob\RR$ with $\RR(p_0, p_1) 
\neq 0$, there exists an isomorphism in $\RR(p_0, p_1)$.
\end{enumerate}
\end{dfn}

We denote by $\pi_0\RR$ the set of isomorphism classes in $\Ob\RR$.
For any $q \in \Ob\RR$ the additive group $\RR(q, q)$
is an associative $K$-algebra with unit by the condition (i).
Further, $\RR(p_0, p_1)$ is a left $\RR(p_0, p_0)$- and right
$\RR(p_1, p_1)$- module.

A typical example of a $K$-SAC is the free $K$-module over a groupoid 
$\GG$. For any $p_0$ and $p_1 \in \Ob\GG$, we define $(K\GG)(p_0, p_1)$ 
to be the free $K$-module over the set $\GG(p_0, p_1) = \Hom_\GG(p_0, p_1)$. 
Clearly $K\GG$ is a $K$-SAC, and $\pi_0K\GG = \pi_0\GG$.

A family of $K$-endomorphisms $D = D^{(p_0, p_1)}\colon \RR(p_0, p_1) \to
\RR(p_0, p_1)$, $p_0, p_1 \in \Ob\RR$, is called a {\it derivation} of $\RR$, 
if it satisfies Leibniz' rule
$$
D(uv) = (Du)v + u(Dv)
$$
for any $p_0, p_1, p_2 \in \Ob\RR$, $u \in \RR(p_0, p_1)$ and 
$v \in \RR(p_1, p_2)$. Then, for any $q \in \Ob\RR$, $D = D^{(q,q)}$ 
is a derivation of the associative $K$-algebra $\RR(q,q)$ in a usual
sense. It should be remarked that a derivation is not a covariant functor. 
We denote by $\Der\RR$ the $K$-Lie algebra consisting of all derivations of $\RR$.

Let $\RR$ and $\RR'$ be $K$-SAC's with the same object set. 
We denote 
$(\RR\otimes\RR')(p_0,p_1) := \RR(p_0,p_1)\otimes_K\RR'(p_0,p_1)$
for $p_0, p_1 \in \Ob\RR = \Ob\RR'$. The tensor product of 
multiplications in $\RR$ and $\RR'$ makes $\RR\otimes\RR'$ a $K$-SAC
whose object set is $\Ob\RR = \Ob\RR'$, which we call the {\it tensor
product} of $K$-SAC's $\RR$ and $\RR'$.

We call a covariant functor $\mathcal{F}\colon \RR \to \RR'$ a {\it
homomorphism} of $K$-SAC's, if $\mathcal{F}(p) = p$ for any 
$p \in \Ob\RR = \Ob\RR'$, and 
$\mathcal{F}\colon \RR(p_0, p_1) \to \RR'(p_0, p_1)$ is $K$-linear
for any $p_0$ and $p_1 \in \Ob\RR$. 
We denote by $\Hom(\RR, \RR')$ the $K$-vector space consisting of all
homomorphisms of $K$-SAC's from $\RR$ to $\RR'$. 
A homomorphism $U\in \Hom(\RR, \RR)$ is called an {\it
automorphism} of $\RR$, if $U\colon \RR(p_0, p_1) \to \RR(p_0, p_1)$ is a $K$-linear
isomorphism for any $p_0$ and $p_1 \in \Ob\RR$. 
We denote by $\Aut\RR$ the group consisting of all automorphisms of $\RR$. 

\subsection{Filtered SAC's}
\label{sec:1-2}

\begin{dfn} \label{12def}
A $K$-SAC $\RR$ is called {\it filtered}, if each $\RR(p_0, p_1)$, $p_0,p_1 \in \Ob\RR$,
has a sequence of $K$-submodules $\{F_n\RR(p_0, p_1)\}_{n\geq 0}$ such that
\begin{enumerate}
\item[{\rm (i)}] $F_0\RR(p_0, p_1) = \RR(p_0, p_1)$ and 
$F_n\RR(p_0, p_1) \supset F_{n+1}\RR(p_0, p_1)$ for any $n \geq 0$.
\item[{\rm (ii)}] $F_{n_1}\RR(p_0, p_1)\cdot F_{n_2}\RR(p_1, p_2) \subset
F_{n_1+n_2}\RR(p_0, p_2)$ for any $p_0, p_1, p_2 \in \Ob\RR$ and 
$n_1, n_2 \geq 0$. 
\end{enumerate}
\end{dfn}

Suppose $p_0, p_1, q_0$ and $q_1 \in \Ob\RR$ satisfy $[p_0] = [q_0]$ and 
$[p_1] = [q_1] \in \pi_0\RR$. Let $\gamma \in \RR(p_0, q_0)$ and 
$\delta \in \RR(q_1, p_1)$ be isomorphisms. Then, from the condition (ii) 
with $n_1$ or $n_2 = 0$, we have 
\begin{equation}
F_n\RR(p_0, p_1) = \gamma(F_n\RR(q_0, q_1))\delta
\label{12filter}
\end{equation}
for any $n \geq 0$. 
If $n < 0$, we define $F_n\RR(p_0,p_1) = \RR(p_0,p_1)$. 

If we define
$$
{\rm gr}\RR(p_0, p_1) := \bigoplus^\infty_{n=0}{\rm gr}_n\RR(p_0, p_1), 
\quad {\rm gr}_n\RR(p_0, p_1) = F_n\RR(p_0, p_1)/ F_{n+1}\RR(p_0, p_1)
$$
for $p_0$ and $p_1 \in \Ob\RR$, then ${\rm gr}\RR$ is a $K$-SAC 
with $\Ob{\rm gr}\RR = \Ob\RR$.

Now, in view of \cite{Qui} p.265, we consider the following four
conditions about a filtered $K$-SAC $\RR$.
\begin{enumerate}
\item[(C1)] $\RR(q, q)/F_1\RR(q,q) \cong K$ for any $q \in \Ob\RR$. 
\item[(C2)] The algebra ${\rm gr}\RR(q,q)$ is generated by ${\rm
gr}_1\RR(q,q)$ for any $q \in \Ob\RR$. In other words, the sum of the
multiplication and the inclusion
$F_1\RR(q,q)^{\otimes n}\oplus F_{n+1}\RR(q,q) \to F_{n}\RR(q,q)$
is surjective for any $q \in \Ob\RR$ and $n\geq 1$. 
\item[(C3)] $\RR(q,q) \cong \varprojlim_{n\to\infty}\RR(q,q)/F_n\RR(q,q)$ 
for any $q \in \Ob\RR$. 
\item[(C4)] The multiplication 
$F_1\RR(q,q)^{\otimes n} \to F_{n}\RR(q,q)$
is surjective for any $q \in \Ob\RR$ and $n\geq 1$. 
\end{enumerate}

\begin{lem}
These conditions are equivalent to the followings, respectively.
\begin{enumerate}
\item[{\rm (C'1)}] $\RR(p_0, p_1)/F_1\RR(p_0,p_1) \cong K$ for any $p_0, p_1 \in
\Ob\RR$ with $[p_0] = [p_1] \in \pi_0\RR$. 
\item[{\rm (C'2)}] If $[p_0] = [p_1] \in \pi_0\RR$, then the sum of the
multiplication and the inclusion
$\bigotimes^{n}_{i=1}F_1\RR(q_{i-1},q_i)\oplus F_{n+1}\RR(p_0,p_1) \to
F_{n}\RR(p_0,p_1)$
is surjective for any $q_1, \cdots, q_{n-1} \in
[p_0] \subset \Ob\RR$ and $n\geq 1$ with $q_0 = p_0$ and $q_n=p_1$. 
\item[{\rm (C'3)}] $\RR(p_0,p_1) \cong
\varprojlim_{n\to\infty}\RR(p_0,p_1)/F_n\RR(p_0,p_1)$ for any $p_0, p_1
\in\Ob\RR$. 
\item[{\rm (C'4)}] If $[p_0] = [p_1] \in \pi_0\RR$, then the 
multiplication 
$\bigotimes^{n}_{i=1}F_1\RR(q_{i-1},q_i)\to F_{n}\RR(p_0,p_1)$
is surjective for any $q_1, \cdots, q_{n-1} \in
[p_0] \subset \Ob\RR$ and $n\geq 1$ with $q_0 = p_0$ and $q_n=p_1$. 
\end{enumerate}
\end{lem}
\begin{proof} (C1) $\Leftrightarrow$ (C'1) and (C3) $\Leftrightarrow$
(C'3) follow immediately from (\ref{12filter}). 
(C'2) $\Rightarrow$ (C2) and (C'4) $\Rightarrow$ (C4) are clear. \par
Choose an object $q \in [p_0] \subset \Ob\RR$, and isomorphisms 
$\gamma_i \in \RR(q, q_i)$, $0 \leq i \leq n$. By (\ref{12filter}), 
we have $F_n\RR(p_0, p_1) = {\gamma_0}^{-1}F_n\RR(q,q)\gamma_n$. On the
other hand, we have 
\begin{eqnarray*}
&& {\gamma_0}^{-1}\overbrace{F_1\RR(q,q)\cdot\cdots\cdot
F_1\RR(q,q)}^{n}\gamma_n \\
&=&
{\gamma_{0}}^{-1}F_1\RR(q,q)\gamma_{1}\cdot
{\gamma_{1}}^{-1}F_1\RR(q,q)\gamma_{2}\cdot \cdots \cdot
{\gamma_{n-1}}^{-1}F_1\RR(q,q)\gamma_{n}\\
&=& 
F_1\RR(q_{0},q_{1})\cdot
F_1\RR(q_{1},q_{2})\cdot \cdots \cdot
F_1\RR(q_{n-1},q_{n}).
\end{eqnarray*}
Hence we obtain (C2) $\Rightarrow$ (C'2) and (C4) $\Rightarrow$ (C'4). 
\end{proof}

Let $\RR$ be a filtered $K$-SAC. We define 
\begin{eqnarray*}
&& \widehat{\RR}(p_0, p_1) := \varprojlim_{m\to\infty}\RR(p_0,
p_1)/F_m\RR(p_0, p_1), \quad\mbox{and}\\
&& F_n\widehat{\RR}(p_0, p_1) := \varprojlim_{m\to\infty}F_n\RR(p_0,
p_1)/F_m\RR(p_0, p_1) \subset \widehat{\RR}(p_0, p_1)
\end{eqnarray*}
for $p_0, p_1 \in \Ob\RR$ and $n \geq 0$. Then $\widehat{\RR}$ is a filtered
$K$-SAC  satisfying the condition (C3). We call it the {\it completion} 
of $\RR$. There is a natural homomorphism $\RR \to \widehat{\RR}$
of $K$-SAC's.

Let $\RR$ and $\RR'$ be filtered $K$-SAC's with the same object set. 
We define 
$$
(\RR\widehat{\otimes}\RR')(p_0,p_1) := \varprojlim_{n\to\infty}
(\RR\otimes\RR')(p_0,p_1)/\textstyle\sum_{n_1+n_2=n}F_{n_1}\RR(p_0,p_1)\otimes 
F_{n_2}\RR'(p_0,p_1)
$$
for $p_0, p_1\in \Ob\RR = \Ob\RR'$. Then $\RR\widehat{\otimes}\RR'$ is a
filtered $K$-SAC whose object set is $\Ob\RR = \Ob\RR'$ in an obvious
way, which we call the {\it completed tensor product} of filtered
$K$-SAC's $\RR$ and $\RR'$. If $a\otimes b \in (\RR \otimes \RR^{\prime})(p_0,p_1)$,
then its image under the homomorphism $\RR \otimes \RR^{\prime} \to
\RR\widehat{\otimes}\RR^{\prime}$ is denoted by $a \widehat{\otimes} b$.

\subsection{Derivations of a filtered SAC}
\label{sec:1-3}

Let $\RR$ be a filtered $K$-SAC. We define 
$$
F_n\Der\RR := \{D \in \Der\RR; 
\forall p_0, \forall p_1 \in \Ob\RR, \forall l \geq 0, 
D(F_l\RR(p_0, p_1)) \subset F_{l+n}\RR(p_0, p_1)\}
$$
for any $n \in \mathbb{Z}$. It is clear that 
\begin{equation}
[F_{n_1}\Der\RR, F_{n_2}\Der\RR] \subset F_{n_1+n_2}\Der\RR
\label{13filter}
\end{equation}
for any $n_1$ and $n_2 \in \mathbb{Z}$. 

\begin{lem}\label{13der} 
Let $\RR$ be a filtered $K$-SAC with the condition {\rm (C4)}, and $n \in
\mathbb{Z}$.  If a derivation $D \in \Der\RR$ satisfies
$$
D(\RR(p_0, p_1)) \subset F_n\RR(p_0, p_1)
$$
for any $p_0$ and $p_1 \in \Ob\RR$, then we have $D \in F_{n-1}\Der\RR$. 
In particular, we obtain
$$
\Der\RR = F_{-1}\Der\RR.
$$
\end{lem}
\begin{proof} Let $p_0$ and $p_1$ be objects in $\RR$. 
We may assume $[p_0] = [p_1] \in\pi_0\RR$, since 
$\RR(p_0, p_1) = 0$ if $[p_0] \neq [p_1]$. 
Let $l$ be a positive integer. Choose $q_1, \dots, q_{l-1}
\in [p_0] = [p_1] \subset \Ob\RR$, and denote $q_0 = p_0$
and $q_l=p_1$. From the condition (C4), the multiplication 
$\bigotimes^{l}_{i=1}F_1\RR(q_{i-1},q_i)\to F_{l}\RR(p_0,p_1)$
is surjective. Hence it suffices to show $D(u_1u_2\dots u_l) 
\in F_{n+l-1}\RR(p_0, p_1)$ for any $u_i \in F_1\RR(q_{i-1},q_i)$. 
Now we have 
$$
D(u_1u_2\cdots u_l) = \sum^l_{i=1} u_1\cdots u_{i-1}(Du_i) u_{i+1}\cdots
u_l,
$$
and $u_1\cdots u_{i-1}(Du_i) u_{i+1}\cdots u_l \in F_{n+l-1}\RR(p_0,p_1)$
from the assumption. This proves the lemma.
\end{proof}

Now we study an analytic function of a derivation $D \in \Der\RR$. 

\begin{lem}\label{13conv}
Let $\RR$ be a filtered $K$-SAC with the conditions {\rm (C2)} and {\rm (C3)}, 
and $f(t) = \sum^\infty_{k=0}a_kt^k \in K[[t]]$ a formal power series. 
Suppose a derivation $D\in\Der\RR$ satisfies the following three
conditions. 
\begin{enumerate}
\item[{\rm (i)}] $D \in F_0\Der\RR$.
\item[{\rm (ii)}] For any $p_0$ and $p_1 \in \Ob\RR$, there exists a positive
integer $\nu$ such that $D^\nu = 0$ on ${\rm gr}_1\RR(p_0, p_1)$. 
\item[{\rm (iii)}] For any $p_0$ and $p_1 \in \Ob\RR$, $D(\RR(p_0,p_1))
\subset F_1\RR(p_0, p_1)$.
\end{enumerate}
Then the series
$
f(D) = \sum^\infty_{k=0}a_kD^k \in {\rm End}(\RR(p_0, p_1))
$
converges for any $p_0$ and $p_1 \in \Ob\RR$. 
\end{lem}
\begin{proof}
We have 
\begin{equation}
D^{n(\nu-1)+1}(F_n\RR(p_0,p_1)) \subset F_{n+1}\RR(p_0,p_1)
\label{13increase}
\end{equation}
for any $n \geq 1$. In fact, we may assume $[p_0] = [p_1] \in \pi_0\RR$,
since it is trivial in the case $[p_0] \neq [p_1]$. 
Then choose $q_1, \dots, q_{n-1}
\in [p_0] = [p_1] \subset \Ob\RR$, and denote $q_0 = p_0$
and $q_n=p_1$. From the condition (C2), the multiplication induces 
a surjection $\bigotimes^{n}_{i=1}F_1\RR(q_{i-1},q_i) \oplus
F_{n+1}\RR(p_0,p_1)\to F_{n}\RR(p_0,p_1)$. By the condition (ii) 
we have $D^\nu u_i \in F_2\RR(q_{i-1}, q_i)$ for any $u_i \in
F_1\RR(q_{i-1}, q_i)$. Hence $D^{n(\nu-1)+1}(u_1u_2\cdots u_n) \in 
F_{n+1}\RR(p_0,p_1)$, while $DF_{n+1}\RR(p_0,p_1) \subset
F_{n+1}\RR(p_0,p_1)$ from the condition (i). This proves
(\ref{13increase}).\par
For any $u \in \RR(p_0,p_1)$, we have $Du \in F_1\RR(p_0,p_1)$ by the
condition (iii). Hence, by (\ref{13increase}), we have $D^mu \in
F_n\RR(p_0,p_1)$ if $m \geq 1+\sum^{n-1}_{k=1}k(\nu-1)+1 =
n+\frac12n(n-1)(\nu-1)$. By the condition (C3) the series $f(D)$
converges as an element of ${\rm End}(\RR(p_0,p_1))$. This completes the
proof.
\end{proof}

If $K$ includes the rationals $\mathbb{Q}$,
then we may consider $\exp(t)$ and $\frac{1}{t}(\exp(t)-1) \in K[[t]]$. 

\begin{prop}\label{13exp}
Suppose $K$ includes $\mathbb{Q}$ and
let $\RR$ be a filtered $K$-SAC with the conditions {\rm (C2)} and {\rm (C3)}, 
and $D$ and $D'$ derivations of $\RR$ satisfying the three conditions 
{\rm (i)-(iii)} in {\rm Lemma \ref{13conv}}. Then
\begin{enumerate}
\item[{\rm (1)}] If $[D, D'] = 0$, then the sum $D+D'$ satisfies the three
conditions {\rm (i)-(iii)}, and $\exp(D+D') = \exp (D)\exp (D')$.
\item[{\rm (2)}] $\exp (D) \in {\rm Aut}\RR$.
\item[{\rm (3)}] If $\exp(D) = \exp(D') \in {\rm Aut}\RR$, then we have $D = D' \in
\Der\RR$. 
\end{enumerate}
\end{prop}
\begin{proof}
(1) The conditions (i) and (iii) are clear. By $[D, D'] = 0$, we have 
\begin{equation}
(D+D')^m = \sum^m_{k=0}\begin{pmatrix} m\\k\end{pmatrix} D^k{D'}^{m-k}.
\label{13bino} 
\end{equation}
Hence, if $D^\nu = {D'}^{\nu'}=0$ on ${\rm gr}_1\RR(p_0,p_1)$, 
then $(D+D')^{\nu+\nu'} = 0$ on ${\rm gr}_1\RR(p_0,p_1)$. 
This implies $D+D'$ satisfies the condition (ii). 
By (\ref{13bino}) we compute $\exp(D+D') = \exp (D)\exp (D')$. \par
(2) It is clear that $(\exp D)(1_q) = 1_q$ for any $q \in \Ob\RR$.
Here $1_q$ is the unit of the object $q$.
Leibniz' formula implies $D^k(uv) = \sum^k_{j=0}\begin{pmatrix} k\\j\end{pmatrix}
(D^ju)(D^{k-j}v)$ for any $p_0, p_1, p_2 \in \Ob\RR$, $u \in
\RR(p_0,p_1)$ and $v \in \RR(p_1,p_2)$. Hence $(\exp D)(uv) = (\exp
D)(u)(\exp D)(v)$, which means that $\exp (D)$ is a covariant functor from $\RR$
to $\RR$ itself. By (1), $\exp(-D)$ is the inverse of $\exp (D)$. 
Hence we obtain $\exp (D) \in {\rm Aut}\RR$.\par
(3) We denote $f(t) := \frac{1}{t}(\exp t -1) =
\sum^\infty_{k=0}\frac{1}{(k+1)!}t^k \in K[[t]]$. We have $(\exp(D) - 1)^n
= D^n f(D)^n$. Hence, as was proved in the proof of Lemma \ref{13conv},
we have $(\exp(D) - 1)^mu \in F_n\RR(p_0, p_1)$ for any $p_0, p_1 \in
\Ob\RR$ and $u \in \RR(p_0,p_1)$, if $m \geq n+\frac{1}{2}n(n-1)(\nu-1)$. 
This implies $\log(\exp (D)) =
\sum^\infty_{n=1}\frac{(-1)^{n-1}}{n}(\exp(D) - 1)^n$ converges as an
element of ${\rm End}\RR(p_0,p_1)$. On the other hand, $\log(\exp (D))
\equiv D \pmod {D^m}$ for any $m \geq 1$. Hence, from the condition (C3),
we have $\log(\exp (D)) = D$. In particular, if $\exp (D)= \exp (D')$, then 
we have $D = D'$.
\end{proof}

\section{Groupoids}
\label{sec:2}

Let $K$ be a commutative ring with unit, and $\GG$ a groupoid. 
As was stated in \S\ref{11def}, $K\GG$, the free $K$-module over $\GG$, 
is a $K$-SAC with $\Ob K\GG = \Ob\GG$. 
A homomorphism of $K$-SAC's $\Delta\colon K\GG \to K\GG\otimes K\GG$ is
defined by $\Delta\gamma := \gamma\otimes\gamma$ for any $\gamma \in
\GG(p_0,p_1)$, $p_0, p_1 \in \Ob\GG$, which we call the {\it coproduct} of
$K\GG$. If $\Ob\GG$ is a singleton, namely, $\GG$ is a group, then 
$\Delta$ defines the standard Hopf algebra structure on the group ring
$K\GG$. 

\subsection{Filtration on $K\GG$}
\label{21filtration}

We have the augmentation $\varepsilon\colon K\GG(q,q) \to K$ and the
augmentation ideal $I\GG(q,q) := \Ker\varepsilon \subset K\GG(q,q)$ for
any $q \in \Ob\GG$. We remark that the power $I\GG(q,q)^n$ is a two-sided
ideal of the group ring $K\GG(q,q)$ for any $n \geq 0$. 

\begin{prop} \label{21filter}
The $K$-SAC $K\GG$ has a filration
$\{F_nK\GG(p_0,p_1)\}_{n\geq 0}$, $p_0, p_1 \in \Ob\GG$, such that 
$F_nK\GG(q,q) = I\GG(q,q)^n$ for any $n \geq 0$. The filtered $K$-SAC
$K\GG$ satisfies the conditions {\rm (C1)} and {\rm (C4)}. 
\end{prop}
\begin{proof} Let $p_0$ and $p_1$ be objects in $\GG$ with $[p_0] = [p_1]
\in \pi_0\GG = \pi_0K\GG$. Then, for any $q, q_1 \in [p_0]$, $\gamma \in
\GG(p_0, q)$, $\gamma_1 \in\GG(p_0, q_1)$, $\delta \in\GG(q, p_1)$,  
$\delta_1 \in\GG(q_1, p_1)$ and $n \geq 0$, we have 
\begin{equation}
\gamma I\GG(q,q)^n \delta = \gamma_1 I\GG(q_1,q_1)^n \delta_1
\subset K\GG(p_0,p_1)
\label{21welldef}
\end{equation}
In fact, since the map $\GG(q,q) \to \GG(q_1, q_1)$, $x \mapsto
{\gamma_1}^{-1}\gamma x \gamma^{-1}\gamma_1$, is an isomorphism of
groups, we have ${\gamma_1}^{-1}\gamma I\GG(q,q)^n \gamma^{-1}\gamma_1
= I\GG(q_1,q_1)^n$. Since $I\GG(q_1,q_1)^n$ is a two-sided ideal of
$K\GG(q_1,q_1)$ and $\delta_1\delta^{-1}\gamma^{-1}\gamma_1 \in
\GG(q_1,q_1)$ is an invertible element of $K\GG(q_1,q_1)$, we have 
$I\GG(q_1,q_1)^n =
I\GG(q_1,q_1)^n\delta_1\delta^{-1}\gamma^{-1}\gamma_1$. 
Hence ${\gamma_1}^{-1}\gamma I\GG(q,q)^n \gamma^{-1}\gamma_1
= I\GG(q_1,q_1)^n\delta_1\delta^{-1}\gamma^{-1}\gamma_1$, and so 
$\gamma I\GG(q,q)^n \delta = \gamma_1 I\GG(q_1,q_1)^n \delta_1$. \par
Hence we may define $$
F_nK\GG(p_0, p_1) := \gamma I\GG(q,q)^n \delta, \quad n \geq 0,
$$ 
if $[p_0] = [p_1]$. In the case $[p_0] \neq [p_1]$, we define
$F_nK\GG(p_0, p_1) := 0 \subset K\GG(p_0, p_1) = 0$. \par
Next we prove the condition (ii) in Definition \ref{12def}
$$
F_{n_1}K\GG(p_0, p_1)\cdot F_{n_2}K\GG(p_1, p_2) \subset
F_{n_1+n_2}K\GG(p_0, p_2)
$$ 
for any $p_0, p_1, p_2 \in \Ob\GG$ and $n_1, n_2 \geq 0$. 
Choose $\gamma \in \GG(p_0,p_1)$ and $\delta \in \GG(p_1,p_2)$. 
Then $\gamma^{-1}F_{n_1}K\GG(p_0,p_1) = I\GG(p_1,p_1)^{n_1}$ and 
$F_{n_2}K\GG(p_1,p_2)\delta^{-1}= I\GG(p_1,p_1)^{n_2}$, so that 
$\gamma^{-1}F_{n_1}K\GG(p_0,p_1) \cdot F_{n_2}K\GG(p_1,p_2)\delta^{-1}
\subset I\GG(p_1,p_1)^{n_1+n_2}$. This proves that $K\GG$ is a filtered $K$-SAC.

The conditions (C1) and (C4) are clear from the definition of the power
of the augmentation ideal $I\GG(q,q)^n$, $n\geq 0$, $q \in \Ob\GG$.
\end{proof}

Following \S \ref{sec:1-2} we can define the completion of the filtered $K$-SAC $K\GG$, 
$\widehat{K\GG}$, which satisfies the conditions (C1), (C2) and (C3).
Since $K\GG$ satisfies the condition (C4), we have $\Der K\GG =
F_{-1}\Der K\GG$ from Lemma \ref{13der}. In particular, any derivation 
$D \in \Der K\GG$ induces a derivation of the completion $\widehat{K\GG}$
in a natural way. In other words, we have a natural homomorphism of
$K$-Lie algebras, $\Der K\GG \to \Der\widehat{K\GG}$. \par

The coproduct $\Delta$ on $K\GG$ satisfies
\begin{equation}
\Delta F_nK\GG(p_0,p_1) \subset\sum_{n_1+n_2=n}F_{n_1}K\GG(p_0,p_1)
\otimes F_{n_2}K\GG(p_0,p_1)
\label{21coprod}
\end{equation}
for any $p_0, p_1 \in \Ob\GG$ and $n \geq 0$. 
In fact, it is easy to see
$\Delta(I\GG(q,q)^n) \subset\sum_{n_1+n_2=n}
I\GG(q,q)^{n_1}\otimes I\GG(q,q)^{n_2}$ for any $q \in \Ob\GG$. 
For $\gamma \in \GG(p_0,q)$ and $\delta \in \GG(q, p_1)$, $p_0, p_1 \in
[q]$, we have $\Delta F_nK\GG(p_0,p_1) = \Delta(\gamma I\GG(q,q)^n\delta)
= (\gamma\otimes\gamma)\Delta(I\GG(q,q)^n)(\delta\otimes\delta) 
\subset \sum_{n_1+n_2=n}
\gamma I\GG(q,q)^{n_1}\delta\otimes\gamma I\GG(q,q)^{n_2}\delta
= \sum_{n_1+n_2=n}F_{n_1}K\GG(p_0,p_1)
\otimes F_{n_2}K\GG(p_0,p_1)$, as was to be shown. \par
Hence $\Delta$ induces a homomorphism of $K$-SAC's
$$
\Delta\colon \widehat{K\GG} \to \widehat{K\GG}\widehat{\otimes}\widehat{K\GG},
$$
which we call the coproduct of $\widehat{K\GG}$. We denote by 
$\Der_\Delta\widehat{K\GG}$ the Lie subalgebra of $\Der\widehat{K\GG}$ 
consisting of all continuous derivations $D$ stabilizing the coproduct $\Delta$, 
namely, satisfying 
$$
\Delta D = (D\widehat{\otimes}1 + 1\widehat{\otimes}D)\Delta\colon 
\widehat{K\GG}(p_0,p_1) \to
(\widehat{K\GG}\widehat{\otimes}\widehat{K\GG})(p_0,p_1)
$$
for any $p_0, p_1 \in \Ob\GG$.

\subsection{Abelianization $\GG^\abel$}
\label{sec:2-2}

In this subsection we introduce the abelianization of a groupoid $\GG$, 
$\GG^\abel$. Before defining the abelianization, we remark that, for any
filered $K$-SAC $\RR$, the quotient $\RR/F_n\RR$ given by
$\Ob\RR/F_n\RR := \Ob\RR$ and $(\RR/F_n\RR)(p_0,p_1) :=
\RR(p_0,p_1)/F_n\RR(p_0,p_1)$, $p_0, p_1 \in \Ob\RR$, is also a filtered 
$K$-SAC. Any derivation $D \in F_0\Der\RR$ defines a derivation of the
quotient $\RR/F_n\RR$ in a natural way. In other words, one can define a
natural homomorphism of $K$-Lie algebras $F_0\Der\RR \to
F_0\Der(\RR/F_n\RR)$. Moreover we remark that, for any group $G$, we have
a natural isomorphism
$$
IG/(IG)^2 \cong G^\abel, \quad x-1 \mapsto x\bmod [G,G].
$$
Here $IG$ is the augmentation ideal of the integral group ring
$\mathbb{Z}G$. The unit $1 \in \mathbb{Z}G$ gives a decomposition 
$\mathbb{Z}G/(IG)^2 = \mathbb{Z}\oplus IG/(IG)^2 = \mathbb{Z}\oplus
G^\abel$. The multiplication of any two elements in $G^\abel$ vanishes 
in the ring $\mathbb{Z}G/(IG)^2$. In particular, $\mathbb{Z}G/(IG)^2$ is
a commutative ring. The conjugate action of $G$ on the ring
$\mathbb{Z}G/(IG)^2$ is trivial. \par 
For any groupoid $\GG$, we call the
$\mathbb{Z}$-SAC
$$
\GG^\abel := \mathbb{Z}\GG/F_2\mathbb{Z}\GG
$$
the {\it abelianization} of the groupoid $\GG$. The conjugate
action of $\GG(q,q)$ on the ring $\GG^\abel(q,q)$ is trivial. Hence, 
for any $\lambda \in \pi_0\GG$ and $q_1, q_2 \in \lambda$, the
isomorphism 
$$
\GG^\abel(q_1,q_1) \overset\cong\to \GG^\abel(q_2,q_2), \quad
x \mapsto \gamma x\gamma^{-1}
$$
does not depend on the choice of $\gamma \in \GG(q_2,q_1)$. 
Under this identification we define
$$
H\GG(\lambda) := \GG^\abel(q,q), \quad q \in \lambda.
$$
$H\GG(\lambda)$ is a commutative ring. If $p_0, p_1 \in \lambda$, then 
$\GG^\abel(p_0,p_1)$ is a left and right $H\GG(\lambda)$-module. 
We remark $\gamma x = x\gamma \in \GG^\abel(p_0,p_1)$ for any $\gamma \in
\GG(p_0, p_1)$ and $x \in H\GG(\lambda)$. 

\section{Oriented surfaces}
\label{sec:3}

Let $S$ be an oriented surface, or equivalently, an oriented $2$-dimensional $C^\infty$-manifold.
Throughout this paper, a {\it simple closed curve} (SCC) on $S$ means 
a smooth submanifold of the interior $S \setminus \partial S$ which 
is diffeomorphic to the circle $S^1$. It can be regarded as an unoriented free loop in $S$. 

Let $E$ be a non-empty closed subset of the surface $S$.
What we actually have in mind is the case $E$ is a disjoint union of
finitely many simple closed curves and finitely many points.
Then we introduce a groupoid $\CC = \CSE$ by setting $\Ob\CC := E$ and 
$$
\CC(p_0, p_1) = \Hom_\CC(p_0,p_1) := \Pi S(p_0, p_1) = [([0,1], 0, 1),
(S, p_0, p_1)],  
$$
the homotopy set of continuous paths on $S$ from $p_0$ to $p_1$,
for $p_0, p_1 \in E = \Ob\CC$. As usual a path and its homotopy class
will be denoted by the same letter, if there is no fear of confusion.
The multiplication
$\gamma_1\gamma_2=\gamma_1 \cdot \gamma_2$, where $\gamma_1\in \CC(p_0,p_1)$,
$\gamma_2 \in \CC(p_1,p_2)$ and $p_0,p_1,p_2 \in E$, means the
(homotopy class of) conjunction, which traverses $\gamma_1$ first.
The natural map $\pi_0\CSE \to \pi_0S$ is injective. 

Let $K$ be a commutative ring including the rationals $\mathbb{Q}$. 
Then,  if $G$ is a finitely generated free group or a surface group, the
completion map $KG \to \widehat{KG} = \varprojlim_{n\to\infty}KG/(IG)^n$
is injective \cite{Bou} \cite{Ch}. Hence, if the fundamental group of any
connected component of $S$ is finitely generated, then the completion map
$K\CSE \to
\widehat{K\CSE}$ is injective. 

\subsection{Dehn-Nielsen homomorphism}
\label{sec:3-1}

We define {\it the mapping class group of the pair $(S,E)$},
which is denoted by $\mathcal{M}(S,E)$, to
be the set of orientation preserving diffeomorphisms of
$S$ that fix $E \cup \partial S$ pointwise, modulo
isotopies relative to $E \cup \partial S$.
A diffeomorphism and its class in $\mathcal{M}(S,E)$
will be denoted by the same letter, if there is no fear of confusion.
Adopting the functional notation, the product $\varphi_1\varphi_2$
for $\varphi_1, \varphi_2 \in \mathcal{M}(S,E)$ means the (class of)
composition $\varphi_1 \circ \varphi_2$.
Then we can define a group homomorphism of Dehn-Nielsen type
$$
{\sf DN}\colon \mathcal{M}(S, E) \to {\rm Aut}(K\CSE),
$$
by ${\sf DN}(\varphi)(\gamma):=\varphi(\gamma)$, for
$\gamma \in \mathcal{C}(p_0,p_1)$, $p_0,p_1\in E$.
For any groupoid $\GG$ we denote by $\Aut\GG$ the group consisting of all
covariant functors $\mathcal{F}\colon \GG\to\GG$ satisfying the conditions
$\mathcal{F}(p_0) = p_0$ and $\mathcal{F}\colon \GG(p_0,p_1) \to \GG(p_0,p_1)$
is an isomorphism for any $p_0$ and $p_1 \in \Ob\GG$. By definition the Dehn-Nielsen
homomorphism ${\sf DN}\colon \mathcal{M}(S, E) \to {\rm Aut}(K\CSE)$
factors through the group $\Aut\CSE$.

We say $S$ is {\it of finite type}, if $S$ is a compact
connected oriented surface, or a surface obtained from
a compact connected oriented surface by removing
finitely many points in the interior.

\begin{thm} \label{31DN}
Suppose $S$ is of finite type with non-empty boundary,
$E\subset \partial S$, and
any connected component of $\partial S$ has an element of $E$.
Then the homomorphism ${\sf DN}\colon \mathcal{M}(S, E) \to {\rm Aut}(K\CSE)$
is injective.
\end{thm}
\begin{proof}
Let $\varphi \in \mathcal{M}(S,E)$ and suppose ${\sf DN}(\varphi)=1$.
Since $\varphi$ is identity on $\partial S$, for any $p,q \in \partial S$
and $\gamma \in \Pi S(p,q)$, we have $\varphi(\gamma)=\gamma$. Moreover,
by \cite{EPS} Theorem 3.1, $\varphi(\gamma)$ is isotopic to $\gamma$.
Now we take a system of proper arcs $\alpha_1,\beta_1,\ldots,\alpha_g,
\beta_g$, $\gamma_1,\ldots,\gamma_{r-1}$, and
$\delta_1,\ldots, \delta_n$ in $S$ ($g$ is the genus of $S$,
$r$ is the number of components of $\partial S$, and $n$ is
the number of punctures), such that the surface
obtained from $S$ by cutting along these arcs is the union of a disk
and $n$ punctured disks. See Figure 1.
\begin{figure}
\begin{center}
\input{DNtheorem.tex}

\caption{the case $g=1$, $r=2$, and $n=1$}
\end{center}
\end{figure}
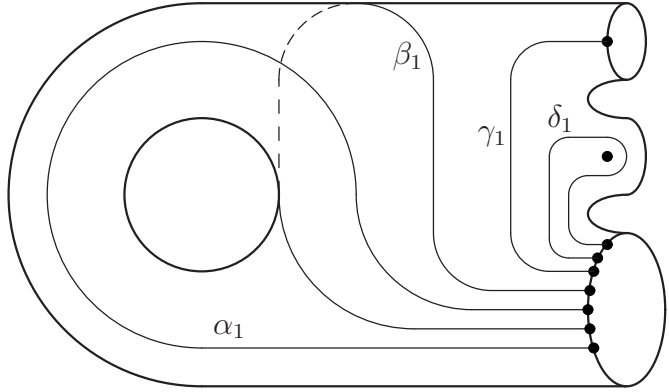

Applying \cite{FM} Proposition 2.8, we may assume $\varphi$ is identity
on these arcs. Finally, by the fact that the group ${\rm Diff}(D^2,\partial D^2)$
is contractible \cite{Smale}, we conclude that $\varphi$ is isotopic to the identity. 
\end{proof}

If $K$ includes $\mathbb{Q}$ and $S$ is of finite type, 
the completion map $K\CSE \to\widehat{K\CSE}$ is injective, as was stated
in the beginning of this section. Hence the natural homomorphism
$\widehat{{\sf DN}}\colon \mathcal{M}(S, E) \to {\rm Aut}(\widehat{K\CSE})$ is
also injective under the assumption of Theorem \ref{31DN}.\par
It would be very interesting if one could find a characterization of the
image of the homomorphism ${\sf DN}$ for any $(S, E)$. \par

Next we consider the case $E\not\subset \partial S$ and $E\setminus\partial S$ is
finite. We number the elements of the set $E\setminus\partial S$, as
$E\setminus\partial S = \{q^0_1, q^0_2,\dots, q^0_s\}$, where $s =
\sharp(E\setminus\partial S) \geq 1$. 
\begin{lem}\label{31ker} Assume $S$ is connected and $E\cap\partial S\neq\emptyset$. Then the
kernel of the forgetful homomorphism $\Aut\CSE \to
\Aut\CC(S,E\cap\partial S)$ is isomorphic to the fundamental group 
$\pi_1(S^s, (q^0_1,\dots, q^0_s)) = \prod^s_{i=1}\pi_1(S,q^0_i)$. 
\end{lem}
\begin{proof}
For any $x_i \in \pi_1(S,q_i^0)$, $1 \leq i\leq s$, we define
$\mathcal{F}=\mathcal{F}(x_1,\dots,x_s) \in \Aut\CSE$ by 
$$
\mathcal{F}\gamma :=
\begin{cases}
\gamma, &\mbox{if $p_0, p_1 \in E\cap\partial S$,}\\
x_{i_0}\gamma, &\mbox{if $p_0 = q^0_{i_0}$ and $p_1 \in E\cap\partial
S$,}\\
\gamma{x_{i_1}}^{-1}, &\mbox{if $p_0 \in E\cap\partial
S$ and $p_1 = q^0_{i_1}$,}\\
x_{i_0}\gamma{x_{i_1}}^{-1}, &\mbox{if $p_0 = q^0_{i_0}$ and $p_1 =q^0_{i_1}$,}
\end{cases}
$$
for $\gamma \in \mathcal{C}(p_0,p_1)$. It is clear that the map 
$$
\prod^s_{i=1}\pi_1(S,q^0_i) \to \Aut\CSE, \quad
(x_1,\dots, x_s) \mapsto \mathcal{F}(x_1,\dots,x_s)
$$
is an injective group homomorphism, and its image is in the kernel of the
forgetful homomorphism. Hence it suffices to show that the kernel is included in the image.

Let $U$ be an element of the kernel of the forgetful homomorphism. 
Choose a point $* \in E\cap\partial S$ and paths $\gamma_i \in \Pi
S(p_i,*)$, $1\leq i \leq s$. Define $x_i := (U\gamma_i){\gamma_i}^{-1}
\in \Pi S(p_i,p_i) = \pi_1(S,p_i)$, $1 \leq i\leq s$. Then $x_i$ does not
depend on the choice of $*$ and $\gamma_i$.
In fact, for another $*^{\prime} \in E\cap \partial S$ and
$\gamma'_i \in \Pi S(p_i,*^{\prime})$, take some $\delta \in \Pi S(*^{\prime},*)$.
Then we have $(U\gamma_i)^{-1}(U\gamma'_i)\delta = U({\gamma_i}^{-1}\gamma'_i\delta) =
{\gamma_i}^{-1}\gamma'_i\delta \in \pi_1(S,*)$ since $U$ is an element of the
kernel of the forgetful homomorphism. Hence we have
$(U\gamma'_i){\gamma'_i}^{-1} = x_i$. This means $U =
\mathcal{F}(x_1,\dots,x_s)$ and proves the lemma. 
\end{proof}

Let $F_s(S\setminus\partial S)$ be the configuration space of ordered
distinct $s$ points
$$
F_s(S\setminus\partial S) := \{
(q_1,q_2,\dots,q_s) \in (S\setminus\partial S)^s; \forall i\neq\forall
j,\,  q_i\neq q_j\}.
$$
Then we have a natural exact sequence
\begin{equation*}
1 \to \pi_1(F_s(S\setminus\partial S), (q^0_1,\dots, q^0_s)) \to 
\mathcal{M}(S,E) \to \mathcal{M}(S,E\cap\partial S)\to 1.
\end{equation*}
See \cite{Bir} Theorem 4.3, or \cite{FM} Theorem 4.6.

\begin{thm} Assume $S$ is of finite type with non-empty boundary,
any component of $\partial S$ has an element of $E$, and
$E\setminus\partial S$ is a non-empty finite set. 
Then the kernel
of the Dehn-Nielsen homomorphism ${\sf DN}\colon \mathcal{M}(S,E) \to 
\Aut K\CSE$ is isomorphic to the kernel of the inclusion homomorphism 
$\pi_1(F_s(S\setminus\partial S), (q^0_1,\dots, q^0_s)) \to \pi_1(S^s,
(q^0_1,\dots, q^0_s))$. 
In particular, the homomorphism ${\sf DN}$ is injective if and only if
$\sharp (E\setminus\partial S) = 1$.
\end{thm}
\begin{proof} Consider the morphism of exact sequences
$$
\begin{CD}
1 @>>> \pi_1(F_s(S\setminus\partial S)) @>>> \mathcal{M}(S,E) @>>> 
\mathcal{M}(S,E\cap\partial S) @>>> 1\\
@. @VVV @VVV @VVV @.\\
1 @>>> \pi_1(S^s) @>>> \Aut\CSE @>>> \Aut\CC(S,E\cap\partial S).
\end{CD}
$$
The right vertical arrow is injective from Theorem \ref{31DN}. 
The theorem follows from chasing the diagram. 
\end{proof}

\subsection{van Kampen theorem}
\label{sec:3-2}

In this subsection we prove the easier half of the van Kampen theorem 
for the groupoid $\CSE$. Let $S_1$ and $S_2$ be oriented surfaces, 
$\partial'S_1$ and $\partial'S_2$ sums of some connected components of
the boundary $\partial S_1$ and $\partial S_2$, respectively, and 
$\varphi\colon \partial'S_1 \overset\cong\to \partial'S_2$ an
orientation-reversing diffeomorphism. Moreover let $E_1 \subset S_1$ and
$E_2\subset S_2$ be non-empty closed subsets. We assume the
condition 
\begin{quote}
Any connected component of $\partial'S_i$ has some point
in $E_i$ for $i =
1,2$, and $\varphi$ maps the set $E_1\cap \partial'S_1$ onto the
set $E_2\cap
\partial'S_2$.
\end{quote}
Then we define $S_3 := S_1\cup_\varphi S_2$, $E_3 := E_1\cup_\varphi
E_2$ and $E_3^\partial \subset E_3$ the image of $E_1\cap \partial'S_1$ 
and $E_2\cap\partial'S_2$. We write simply $\CC_i := \CC(S_i, E_i)$ for
$i=1,2,3$. We have the inclusion map $\iota_i\colon \CC_i \to \CC_3$ for $i=
1,2$. \par
The van Kampen theorem says $\CC_3$ is ``generated" by $\CC_1$ and
$\CC_2$.  In order to formulate it in a rigorous way, we prepare some
notations.  Let $p_0$ and $p_1$ be points in $E_3$. Then we denote by
$\overline{\mathcal{E}}(p_0, p_1)$ the set of finite sequences of points
in $E_3$, $\lambda = (q_0, q_1,\dots, q_n) \in {E_3}^{n+1}$, $n \geq 0$, 
satisfying the conditions
\begin{enumerate}
\item[(i)] $q_0 = p_0$ and $q_n = p_1$. 
\item[(ii)] For $1 \leq j\leq n$, either 
$\{q_{j-1}, q_j\} \subset S_1$ or $\{q_{j-1}, q_j\} \subset S_2$. 
\end{enumerate}
Further we denote by $\mathcal{E}(p_0,p_1)$ the set of pairs $(\lambda,
\mu)$, $\lambda = (q_0, q_1,\dots, q_n) \in
\overline{\mathcal{E}}(p_0, p_1)$, $\mu = (\mu_1, \dots, \mu_n) \in
\{1,2\}^n$ such that $\{q_{j-1}, q_j\} \subset S_{\mu_j}$ for any $1 \leq
j\leq n$. For $(\lambda,\mu) \in \mathcal{E}(p_0,p_1)$, we denote 
$K\CC(\lambda, \mu) := \bigotimes^n_{j=1}K\CC_{\mu_j}(q_{j-1}, q_j)$. 
One can define the multiplication map $K\CC(\lambda, \mu)\to
K\CC_3(p_0,p_1)$ in an obvious way. 
\begin{prop}\label{32vKT}
The multiplication map
$$
\bigoplus_{(\lambda, \mu) \in \mathcal{E}(p_0, p_1)}K\CC(\lambda, \mu)\to
K\CC_3(p_0,p_1)
$$
is surjective for any $p_0$ and $p_1 \in E_3$. 
\end{prop}
\begin{proof} Take an open tubular neighborhood $U$ of
$\partial'S_1\cup_\varphi\partial'S_2$ in $S_3$. Then $\{S_1\cup U,
S_2\cup U\}$ is an open covering of $S_3$, and $S_i$ is a
deformation retract of $S_i\cup U$ for $i=1,2$. For any path 
$\ell\colon ([0,1], 0,1) \to (S, p_0, p_1)$, there exists a positive integer
$n \geq 1$ such that $\ell([\frac{j-1}{n}, \frac{j}{n}])$ is included in
$S_1\cup U$ or $S_2\cup U$ for any $1 \leq j \leq n$. Then we may
assume $\ell(\frac{j}{n})\in E_3$ for $1 \leq j \leq n-1$. In
fact, if $\ell(\frac{j}{n}) \in U$, the connected component of
$\ell(\frac{j}{n})$ in $U$ contains a point in $E_3^\partial$
by the assumption on $E_i$ and $\partial'S_i$. We insert a path
connecting $\ell(\frac{j}{n})$ to a point in $E_3^\partial$ inside $U$
into the path $\ell$, and deform it to obtain a new path homotopic to the
original path $\ell$ with $\ell(\frac{j}{n}) \in E_3^\partial$. On the
other hand, if $\ell(\frac{j}{n}) \in S_i \setminus U$, where $i=1$ or $2$, then we
can deform the path $\ell$ inside $S_i\cup U$ such that
$\ell(\frac{j}{n}) \in \{p_0, p_1\}\cup E_3^\partial \subset E_3$. 
Moreover, using the deformation retraction of $S_i\cup U$ onto $S_i$, 
we deform the path $\ell([\frac{j-1}{n}, \frac{j}{n}])$ inside $S_1$ or
$S_2$. Hence we obtain a new path $\ell$ such that 
\begin{itemize}
\item For $1\le j \le n$, either
$\ell([\frac{j-1}{n}, \frac{j}{n}])\subset S_1$
or $\ell([\frac{j-1}{n}, \frac{j}{n}])\subset S_2$,
\item $\ell(\frac{j}{n})$ is in $E_3$ for $1 \leq j \leq n-1$, and
\item $\ell$ is homotopic to the original path relative to $\{0,1\}$.
\end{itemize}
Then we have  
$\ell\vert_{[\frac{0}{n},\frac{1}{n}]}\otimes
\ell\vert_{[\frac{1}{n},\frac{2}{n}]}\otimes\cdots\otimes 
\ell\vert_{[\frac{n-1}{n},\frac{n}{n}]} \in K\CC(\lambda, \mu)$
for some $(\lambda, \mu) \in \mathcal{E}(p_0, p_1)$. 
This proves the proposition.
\end{proof}

\subsection{Automorphisms}
\label{sec:3-3}

From now until the end of \S \ref{sec:5} we suppose $K$ is a commutative
ring including the rationals $\mathbb{Q}$.

Let $S$ be an oriented surface, and $E$ and $E'$ closed subsets, as
before.  If $E \subset E'$, the inclusion $E \hookrightarrow E'$ induces a
homomorphism of filtered SAC's $\widehat{K\CC(S, E)} \to
\widehat{K\CC(S,E')}$ and the forgetful homomorphism 
$\phi\colon \Aut\widehat{K\CC(S,E')} \to \Aut\widehat{K\CSE}$. In this
subsection we study this forgetful homomorphism. 
For topological study of surfaces, the group $\Aut\widehat{K\CSE}$ is 
too large, so that we begin with introducing an appropriate subgroup 
of $\Aut\widehat{K\CSE}$.

\begin{dfn}\label{32ASE}
We define $\ASE$ to be the subgroup of $\Aut\widehat{K\CSE}$ consisting of all
automorphisms $U$ satisfying the following four conditions.
\begin{enumerate}
\item[{\rm (i)}] $U$ is a homeomorphism with respect to the filtration
$\{F_p\widehat{K\CSE}\}_{p\geq 0}$.
\item[{\rm (ii)}] If $\gamma \in \Pi S(p_0, p_1)$, $p_0, p_1 \in E$, is
represented by a path included in $E$, then $U\gamma = \gamma$. 
\item[{\rm (iii)}] $\varepsilon U = \varepsilon\colon \widehat{K\CC}(p_0, p_1) \to
K$ for any $p_0, p_1 \in E$. Here $\varepsilon$ is the augmentation,
which is induced by the $K$-linear map $K\CC(p_0,p_1) \to K$,
$\CC(p_0,p_1)\ni \gamma \mapsto 1$.
\item[{\rm (iv)}] $\Delta U = (U\widehat{\otimes}U)\Delta\colon
\widehat{K\CC}(p_0,p_1) \to
(\widehat{K\CC}\widehat{\otimes}\widehat{K\CC})(p_0,p_1)$ 
for any $p_0, p_1 \in E$. Here $\Delta$ is the coproduct of
$\widehat{K\CC}$ as in {\rm \S \ref{21filtration}}.
\end{enumerate}
\end{dfn}
Clearly the image of the Dehn-Nielsen homomorphism $\widehat{\sf DN}\colon 
\mathcal{M}(S,E) \to \Aut\widehat{K\CSE}$ is included in $\ASE$, 
so that we obtain $\widehat{\sf DN}\colon \mathcal{M}(S,E) \to \ASE$. 
We have the forgetful homomorphism $\phi\colon A(S, E') \to \ASE$. 
\par
For the rest of this subsection we suppose
each component of $S$ is a surface of finite type and not closed
(hence its fundamental group is a finitely generated free group).
Let $C_i\subset S\setminus (E\cup \partial S)$, $1 \leq i \leq n$,
be disjoint simple closed curves 
which are {\it not} null-homotopic in $S$. Choose a basepoint $*_i \in
C_i$ and a simple loop $\eta_i\colon ([0,1], \{0,1\}) \to (C_i, *_i)$ going
around $C_i$. We also denote its homotopy class by 
$\eta_i \in \pi_1(S, *_i)$. We can define ${\eta_i}^a :=
\exp(a\log\eta_i) \in \widehat{K\pi_1(S,*_i)}$ for $a \in K$. 
For any $p \in C_i$, we have some $t \in [0,1]$ such that $\eta_i(t) =
p$. We define $\eta^a_{i,p} :=
(\eta_i\vert_{[0,t]})^{-1}{\eta_i}^a(\eta_i\vert_{[0,t]}) \in
\widehat{K\pi_1(S,p)}$, which is independent of the choice of the path 
$\eta_i\vert_{[0,t]}$. 
We denote $E_1 :=
\bigcup^n_{i=1} C_i$. 
\begin{prop}\label{32forget}
Suppose $U\in A(S, E\cup E_1)$ is in the kernel of the forgetful homomorphism
$\phi\colon A(S, E\cup E_1) \to \ASE$. Then
there exist some $a_i = a^U_i \in K$, $1 \leq i \leq n$, such that 
$$
Uv = 
\begin{cases}
v, &\mbox{if $p_0, p_1 \in E$,}\\
\eta^{a_{i_0}}_{i_0,p_0}v, &\mbox{if $p_0 \in C_{i_0}$,
$p_1\in E$,}\\
v(\eta^{a_{i_1}}_{i_1,p_1})^{-1}, &\mbox{if $p_0\in E$,
$p_1 \in C_{i_1}$,}\\
\eta^{a_{i_0}}_{i_0,p_0}v(\eta^{a_{i_1}}_{i_1,p_1})^{-1}, &\mbox{if $p_0
\in C_{i_0}$,
$p_1 \in C_{i_1}$,}
\end{cases}
$$
for any $v \in \widehat{K\CC}(p_0,p_1)$, $p_0, p_1 \in E \cup E_1$. 
\end{prop} 
\begin{proof}
For each $1\le i \le n$, choose a point $* \in E$ which is in the connected
component containing $*_i$, and a path $\gamma_i \in \Pi S(*,*_i)$.
Consider $u_i := {\gamma_i}^{-1}(U\gamma_i) \in \widehat{K\pi_1(S,*_i)}$. We have
$(\gamma_i\widehat{\otimes}\gamma_i)(u_i\widehat{\otimes}u_i) 
= (U\gamma_i)\widehat{\otimes}(U\gamma_i) = (U\widehat{\otimes}U)
\Delta\gamma_i = \Delta(\gamma_iu_i) =
(\gamma_i\widehat{\otimes}\gamma_i) \Delta u_i$, and so $u_i$ is
group-like. Moreover $u_i$ does not depend on the choice of $*$ and $\gamma_i$.
In fact, for another $*^{\prime} \in E$ and $\gamma'_i \in \Pi S(*^{\prime},*_i)$,
take some $\delta \in \Pi S(*,*^{\prime})$. Then
$\gamma'_i{\gamma_i}^{-1}\delta = (U\gamma'_i){(U\gamma_i)}^{-1}\delta$ since 
$\gamma'_i{\gamma_i}^{-1}\delta \in \pi_1(S, *^{\prime})$ and $\phi(U) = 1$.
Hence we have ${\gamma'_i}^{-1}(U\gamma'_i) = {\gamma_i}^{-1}(U\gamma_i) =
u_i$. In particular, if $\gamma'_i = \gamma_i\eta_i$, we obtain
${\eta_i}^{-1}u_i\eta_i = u_i$, since $U\eta_i = \eta_i$ from the condition
(ii) in Definition \ref{32ASE}. Now we have the following.
\begin{prop}\label{32simple}
Let $S$ be a surface of finite type and not closed,
and $C$ a simple closed curve in $S$ which
is {\rm not} null-homotopic in $S$. Choose a basepoint $*\in C$ and a
simple loop $\eta\colon ([0,1], \{0,1\}) \to (C, *)$ going around $C$. We also 
denote its homotopy class by $\eta \in \pi_1(S, *)$. Then the
subalgebra 
$$
Z(\eta) := \{u \in \widehat{K\pi_1(S,*)}; \eta u = u\eta\}
$$
of $\widehat{K\pi_1(S,*)}$ equals the ring of formal power series in 
$\eta-1$, $K[[\eta-1]] = K[[\log\eta]]$.
\end{prop}
The proof will be given in the second half of this subsection. 
From this proposition we obtain $u_i \in K[[\eta_i-1]]$. 
Since the inclusion homomorphism $K[[\eta_i-1]] \to
\widehat{K\pi_1(S,*_i)}$ is injective, $u_i$ is group-like 
also in $K[[\eta_i-1]]$. Hence we have $u_i = {\eta_i}^{-a_i}$ 
for some $a_i \in K$. 
We have $U\gamma_i =
\gamma_iu_i$ and $U{\gamma_i}^{-1} =
(u_i)^{-1}{\gamma_i}^{-1}$. Since $u_i$ does not depend on the choice of
$\gamma_i$, we have 
\begin{equation}
\forall v \in \widehat{K\CC}(*,*_i), Uv = vu_i, \quad 
\forall v' \in \widehat{K\CC}(*_i,*), Uv' = (u_i)^{-1}v'.
\label{32vv}
\end{equation}
Now let $p_0, p_1 \in E \cup E_1$ and $v \in \widehat{K\CC}(p_0,p_1)$. 
\par
(i) If $p_0, p_1 \in E$, then $Uv = v$ since $\phi(U) = 1$. \par

(ii)
Suppose $p_0 = \eta_{i_0}(t_0) \in C_{i_0}$ and $p_1 \in E$. 
Choose $\delta_1 \in \Pi S(p_1, *)$. Then, since $U\delta_1 = \delta_1$
and $U(\eta_{i_0}\vert_{[0,t_0]}) = \eta_{i_0}\vert_{[0,t_0]}$, we have
$(\eta_{i_0}\vert_{[0,t_0]})(Uv)\delta_1 = U((\eta_{i_0}\vert_{[0,t_0]})
v\delta_1)
= (u_{i_0})^{-1}(\eta_{i_0}\vert_{[0,t_0]})v\delta_1$, and so 
$Uv = \eta^{a_{i_0}}_{i_0,p_0}v$.\par

(iii)
Suppose $p_0 \in E$ and $p_1 = \eta_{i_1}(t_1) \in C_{i_1}$. 
Choose $\delta_0 \in \Pi S(*, p_0)$. Then, since $U\delta_0 = \delta_0$
and $U(\eta_{i_1}\vert_{[0,t_1]})^{-1} = (\eta_{i_1}\vert_{[0,t_1]})^{-1}$, we
have
$\delta_0(Uv)(\eta_{i_1}\vert_{[0,t_1]})^{-1} =
U(\delta_0v(\eta_{i_1}\vert_{[0,t_1]})^{-1})
=\delta_0v(\eta_{i_1}\vert_{[0,t_1]})^{-1}u_{i_1}$, and so $Uv =
v(\eta^{a_{i_1}}_{i_1,p_1})^{-1}$.\par

(iv)
Suppose $p_0 = \eta_{i_0}(t_0) \in C_{i_0}$ and $p_1 = \eta_{i_1}(t_1)
\in C_{i_1}$. Then 
$\gamma_{i_0}(\eta_{i_0}\vert_{[0,t_0]})v(\eta_{i_1}\vert_{[0,t_1]})^{-1} \in
\widehat{K\CC}(*, *_{i_1})$. Hence we have 
$\gamma_{i_0}u_{i_0}
(\eta_{i_0}\vert_{[0,t_0]})(Uv)(\eta_{i_1}\vert_{[0,t_1]})^{-1}
=
U(\gamma_{i_0}(\eta_{i_0}\vert_{[0,t_0]})v
(\eta_{i_1}\vert_{[0,t_1]})^{-1})
=
\gamma_{i_0}(\eta_{i_0}\vert_{[0,t_0]})v
(\eta_{i_1}\vert_{[0,t_1]})^{-1}u_{i_1}$, 
and so $Uv = \eta^{a_{i_0}}_{i_0,p_0}v(\eta^{a_{i_1}}_{i_1,p_1})^{-1}$. 
\par
This completes the proof of Proposition \ref{32forget}.
\end{proof}

As a corollary, we have

\begin{prop}\label{32subs}
Let $N \subset S\setminus (E \cup \partial S)$ be a connected compact
subsurface with non-empty boundary,
which is {\rm not} diffeomorphic to the disk $D^2$. Assume
the inclusion homomorphism of fundamental groups $\pi_1(N) \to \pi_1(S)$ is injective. 
Let $i\colon \widehat{K\CC(N,\partial N)} \to \widehat{K\CC(S,E\cup\partial
N)}$ be the inclusion homomorphism.
Numbering the boundary components of $\partial N$ as
$\partial N=\coprod_{i=1}^n \partial_iN$, we choose $*_i
\in \partial_iN$ and $\eta_i \in \pi_1(N, *_i)$ as in {\rm Proposition \ref{32forget}}.

Suppose $U \in A(N,\partial N)$ and 
$\tilde U \in  A(S, E\cup\partial N)$ satisfy ${\tilde U}\circ i = i\circ
U\colon \widehat{K\CC(N,\partial N)} \to \widehat{K\CC(S,E\cup\partial
N)}$ and $\tilde U$ is in the kernel of the forgetful homomorphism $\phi\colon
A(S, E\cup\partial N) \to A(S, E)$. Then 
there exist some $a_i = a^U_i \in K$, $1 \leq i \leq n$, such that 
\begin{equation}
Uv = \eta^{a_{i_0}}_{i_0,p_0}v(\eta^{a_{i_1}}_{i_1,p_1})^{-1}
\label{33eq}\end{equation}
for any $v \in \widehat{K\CC(N, \partial N)}(p_0,p_1)$, $p_0 \in
\partial_{i_0} N$,  and $p_1 \in \partial_{i_1} N$. 
\end{prop}
\begin{proof}
Since $N\neq D^2$, the inclusion homomorphism $\pi_1(\partial_iN) \to
\pi_1(N)$ is injective. It follows from the assumption the simple closed
curve $\partial_iN$ is {\it not} null-homotopic in $S$. Hence we can
apply Proposition \ref{32forget} to $\partial N$, from which it follows
that there exist some $a_i = a^U_i \in K$, $1 \leq i \leq n$, such that 
$$
{\tilde U}v = \eta^{a_{i_0}}_{i_0,p_0}v(\eta^{a_{i_1}}_{i_1,p_1})^{-1}
$$
for any $v \in \widehat{K\CC(S, E\cup\partial N)}(p_0,p_1)$, $p_0 \in
\partial_{i_0} N$,  and $p_1 \in \partial_{i_1} N$. \par
From the assumption the inclusion homomorphism $K\Pi N(p_0, p_1) \to
K\Pi S(p_0,p_1)$ is injective for any $p_0, p_1 \in \partial N$. Since
$\pi_1(S)$ is a finitely generated free group, the completion map
${K\Pi S}(p_0,p_1) \to \widehat{K\Pi S}(p_0,p_1)$ is also
injective.  Hence the equation (\ref{33eq}) holds for any $p_0, p_1 \in
\partial N$ and $v \in K\Pi N(p_0,p_1)$, while $K\Pi N(p_0,p_1)$ is dense
in $\widehat{K\Pi N}(p_0,p_1)$ and $U$ is continuous. Hence the equation
(\ref{33eq}) holds for any for any $v \in \widehat{K\CC(N, \partial N)}(p_0,p_1)$, $p_0 \in
\partial_{i_0} N$,  and $p_1 \in \partial_{i_1} N$. This proves the
proposition.
\end{proof}

The rest of this subsection is devoted to the proof of Proposition
\ref{32simple}. \par
First of all we need some algebraic facts. 
Let $H_{\mathbb{Z}}$ be a $\mathbb{Z}$-free module of finite rank.
We denote $H:=H_{\mathbb{Z}} \otimes_{\mathbb{Z}} K$ and
$\widehat{T}:=\prod_{m=0}^{\infty} H^{\otimes m}$, the completed
tensor algebra generated by $H$. Throughout this paper we omit the symbol $\otimes$ as a
multiplication in the algebra $\widehat{T}$.

\begin{lem}\label{32center1}
Let $X^{\prime} \in H_{\mathbb{Z}}$ be a primitive element,
and $X:=X^{\prime}\otimes 1 \in H$. Then we have 
$$
\{u \in \widehat{T}; Xu = uX\} = K[[X]].
$$
\end{lem}
\begin{proof}
It suffices to prove $\{u \in H^{\otimes n}; 
Xu = uX\} \subset K[[X]]\cap H^{\otimes n}$ by induction on $n \geq 0$. 
It is clear for the case $n=0$. \par

Choose a $\mathbb{Z}$-free basis $\{ X^{\prime}_i \}_{i=1}^{{\rm rank}H_{\mathbb{Z}}}$
of $H_{\mathbb{Z}}$ with $X^{\prime}_1=X^{\prime}$. The set
$\{ X_i \}_{i=1}^{{\rm rank}H_{\mathbb{Z}}}$ defined by $X_i:=X^{\prime}_i \otimes 1 \in H$
is a $K$-free basis of $H$.
Assume $n \geq 1$. For any $u \in H^{\otimes n}$, there exist unique
elements $u_i \in H^{\otimes (n-1)}$, $1 \leq i \leq {\rm rank}H_{\mathbb{Z}}$, such that 
$u = Xu_1 + \sum_{i\geq 2}X_iu_i$. If $Xu = uX$, then $X^2u_1 +
\sum_{i\geq 2}XX_iu_i = Xu_1X + \sum_{i\geq 2}X_iu_iX$. Since 
$\{X_i\}^{{\rm rank}H_{\mathbb{Z}}}_{i=1}$ is linearly independent, we have $u_i = 0$ 
for $i \geq 2$, and $Xu_1 = u_1X$. Hence, by the inductive assumption, 
$u_1 \in K[[X]]$, and so $u = Xu_1 \in K[[X]]$. This completes the
induction. 
\end{proof}

We remark Lemma \ref{32center1} holds for any commutative ring with unit.
We identify $\Lambda^2 H_{\mathbb{Z}}$ with the $\mathbb{Z}$-submodule
of ${H_{\mathbb{Z}}}^{\otimes 2}$ generated by the set
$\{ X^{\prime}Y^{\prime}-Y^{\prime}X^{\prime};\ X^{\prime}, Y^{\prime}
\in H_{\mathbb{Z}} \}$.

\begin{lem}\label{32center2}
Let $v^{\prime}_0 \in \Lambda^2 H_{\mathbb{Z}}$ be primitive as an
element of ${H_{\mathbb{Z}}}^{\otimes 2}$, and $v_0:=v^{\prime}_0 \otimes 1
\in \Lambda^2 H=\Lambda^2 H_{\mathbb{Z}} \otimes_{\mathbb{Z}} K$.
Then we have 
$$
\{u \in \widehat{T}; v_0u = uv_0\} = K[[v_0]].
$$
\end{lem}
\begin{proof}
It is clear $K[[v_0]]$ is in the LHS. We prove $\{u \in H^{\otimes n}; 
v_0u = uv_0\} \subset K[[v_0]]\cap H^{\otimes n}$ by induction on $n \geq 1$.

To prove the case $n=1$, we consider the adjoint map 
${\rm ad}v'_0: H_{\mathbb{Z}} \to {H_{\mathbb{Z}}}^{\otimes 3}$, 
$Y' \mapsto v'_0Y' - Y'v'_0$. Since the image 
$({\rm ad}v'_0)(H_{\mathbb{Z}})$ is $\mathbb{Z}$-free, 
$\Ker ({\rm ad}v'_0)$ is a direct summand of $H_{\mathbb{Z}}$, 
and we have $\{Y \in H; v_0Y = Yv_0\} = \Ker ({\rm ad}v'_0)
\otimes_{\mathbb{Z}} K$. Assume $\{Y \in H; v_0Y = Yv_0\} \neq 0$. 
Then $\Ker ({\rm ad}v'_0) \neq 0$. In particular, we have 
some primitive element $Z' \in H_{\mathbb{Z}}$ such that 
$v'_0Z' = Z'v'_0$. From Lemma \ref{32center1} for 
$K = \mathbb{Z}$, we have $v'_0 = \lambda {Z'}^2$ 
for some $\lambda \in \mathbb{Z}$. Since 
$v'_0 \in \Lambda^2H_{\mathbb{Z}}$, this implies 
$v'_0 = 0$, which contradicts the assumption $v'_0$ 
is primitive, and proves the case $n=1$.

Choose a $\mathbb{Z}$-free basis 
$\{v'_i\}_{i=1}^{({\rm rank}H_{\mathbb{Z}})^2}$ of 
${H_{\mathbb{Z}}}^{\otimes 2}$ with $v'_1 = v'_0$. 
The subset $\{v_i\}_{i=1}^{({\rm rank}H_{\mathbb{Z}})^2}$ 
defined by $v_i := v'_i\otimes 1 \in \Lambda^2H$ 
is a $K$-free basis of $H^{\otimes 2}$.
Assume $n \geq 2$. For any $u \in H^{\otimes n}$, there exist
unique elements $u_i \in H^{\otimes (n-2)}$, $1 \leq i \leq ({\rm rank}H_{\mathbb{Z}})^2$,
such that 
$u = v_0u_1 + \sum_{i\geq 2}v_iu_i$. If $v_0u = uv_0$, then ${v_0}^2u_1 +
\sum_{i\geq 2}v_0v_iu_i = v_0u_1v_0 + \sum_{i\geq 2}v_iu_iv_0$. Since 
$\{v_i\}^{({\rm rank}H_{\mathbb{Z}})^2}_{i=1}$ is linearly independent, we have $u_i = 0$ 
for $i \geq 2$, and $v_0u_1 = u_1v_0$. Hence, by the inductive
assumption, 
$u_1 \in K[[v_0]]$, and so $u = v_0u_1 \in K[[v_0]]$. This completes the
induction. 
\end{proof}

Further we need some general result on a filtered $\mathbb{Q}$-vector
space. Let $M=F_0 M \supset F_1 M \supset \cdots$
be a filtered $\mathbb{Q}$-vector space. Assume the filtration
$\{F_pM\}_{p\geq0}$ is separated $\bigcap^\infty_{p=0}F_pM = 0$, and 
complete $M = \varprojlim_{p\to\infty}M/F_pM$. Let $\{p_q\}^\infty_{q=0}$
be a sequence of natural numbers with $\lim_{q\to\infty}p_q = +\infty$. 
If a sequence $\{a_q\}^\infty_{q=0} \subset M$ satisfies $a_q \in
F_{p_q}M$ for each $q \geq 0$, then the series $\sum^\infty_{q=0}a_qx^q$
converges as an element of $M =  \varprojlim_{p\to\infty}M/F_pM$ for any
$x \in \mathbb{Q}$. 
\begin{lem}\label{32series}
If there exists an infinite subset $X \subset \mathbb{Q}$ such that 
$$
\forall x\in X, \quad\sum^\infty_{q=0}a_qx^q = 0,
$$
then we have $a_q = 0$ for any $q \geq 0$. 
\end{lem}
\begin{proof}
It suffices to show $a_q \in F_rM$ for any $r \geq 0$. 
There exists some $n=n(r)$ such that $a_q \in F_rM$ for any $q > n$. 
This implies $\sum^n_{q=0}a_qx^q = 0 \in M/F_rM$ for any $x \in X$. 
Since $X$ is infinite, we have some distinct $n+1$ elements $x_0, x_1,
\dots, x_n$ in $X$. Because of the Vandermonde determinants, the
$(n+1)\times(n+1)$-matrix $({x_i}^j)_{0\leq i,j\leq n}$ has an inverse
matrix. Hence we obtain $a_q \in F_rM$, $0 \leq \forall q\leq n$. 
This proves the lemma.
\end{proof}

In order to deduce Proposition \ref{32simple} from these algebraic facts, 
we need the notion of a {\it group-like expansion} of a free group \cite{Mas}.
This notion will be also used in \S \ref{sec:4-3} and \S \ref{aglg}.
Let $\pi$ be a finitely generated free group, and $H$ the $K$-first
homology group of $\pi$,
$$
H := H_1(\pi; K) = \pi^\abel\otimes_{\mathbb{Z}}K.
$$
We denote by $[x] \in H$ the homology class of $x \in \pi$.
Let $\widehat{\mathcal{L}}$ be the space of all Lie-like elements in the
completed tensor algebra $\widehat{T} := \prod^\infty_{m=0}H^{\otimes
m}$. The image of the exponential 
$$
\exp\colon \widehat{\mathcal{L}} \to \widehat{T},\quad
u \mapsto \exp(u) := \sum^\infty_{k=0}\frac{1}{k!}u^k
$$
is a subgroup of the multiplicative group of the algebra $\widehat{T}$.  
We denote by $u*v \in \widehat{\mathcal{L}}$ the Hausdorff series of $u$
and $v \in \widehat{\mathcal{L}}$. 
By definition, we have $\exp(u*v) =
(\exp u)(\exp v)$. 

\begin{dfn}[Massuyeau \cite{Mas}]\label{33Mas}
A map $\theta\colon \pi \to \widehat{T}$ is called
a group-like expansion, if $\theta$ is a group homomorphism of $\pi$
into the multiplicative group $\exp\widehat{\mathcal{L}}$, and $\theta(x) 
\equiv 1+[x] \pmod{\prod^\infty_{m=2}H^{\otimes m}}$ for any $x \in \pi$. 
\end{dfn}

Any group-like expansion $\theta$ induces a filter-preserving isomorphism of Hopf algebras
\begin{equation}
\label{expansion}
\theta\colon \widehat{K\pi}\overset\cong\to \widehat{T}
\end{equation}
(see \cite{Ka} \cite{Mas}).
Here the algebra $\widehat{T}$ is filtered by the ideals
$\widehat{T}_p:=\prod_{m=p}^{\infty}H^{\otimes m}$, $p\ge 1$.

\begin{prop}\label{32expn}
Let $S$ be a surface of finite type and not closed,
and $C$ a simple closed curve in $S$.
Choose a basepoint $* \in C$ and let $\eta \in \pi_1(S,*)$ be a simple
loop around $C$. Then there exists a group-like expansion $\theta$ of
the free group $\pi_1(S, *)$ such that 
\begin{enumerate}
\item[{\rm (i)}] $\theta(\eta) = \exp([\eta])$, if $[\eta] \neq 0 \in H =
H_1(S; K)$,
\item[{\rm (ii)}]
$\theta(\eta) = \exp(\eta'_0\otimes 1)$ for some $\eta'_0 
\in \Lambda^2H_{\mathbb{Z}}$, if $[\eta] = 0 \in H$. 
Here $\eta'_0$ is primitive as an element of ${H_{\mathbb{Z}}}^{\otimes 2}$.
\end{enumerate}
\end{prop}
\begin{proof}
From the assumption on $S$, the fundamental group $\pi_1(S, *)$ has a
presentation 
\begin{equation}
\langle\alpha_i, \beta_i (1 \leq i\leq g), 
\gamma_j (1 \leq j\leq r); 
[\alpha_1, \beta_1]\cdots[\alpha_g, \beta_g]\gamma_1\cdots\gamma_r = 1
\rangle
\label{32pres}
\end{equation}
for some $g \geq 0$ and $r\geq 1$. It suffices to consider the following
three cases.
\begin{enumerate}
\item[(1)] $C$ is non-separating.
\item[(2)] $C$ is separating, and $[\eta] \neq 0 \in H$.
\item[(3)] $C$ is separating, and $[\eta] = 0 \in H$.
\end{enumerate}
\par
(1) By the classification theorem of surfaces, we may take $\eta =
\alpha_1$. Since $r \geq 1$, $\pi_1(S,*)$ is freely generated by 
$\{\alpha_i, \beta_i (1 \leq i\leq g), 
\gamma_j (2 \leq j\leq r)\}$. We can define $\theta$ by 
$\theta(\alpha_i) := \exp([\alpha_i])$, 
$\theta(\beta_i) := \exp([\beta_i])$, $1 \leq i \leq g$, and
$\theta(\gamma_j) := \exp([\gamma_j])$, $2 \leq j\leq r$. 
Clearly we have $\theta(\eta) = \exp([\eta])$. 
\par
(2) The complement $S\setminus C$ has two connected components, 
so that we have $r \geq 2$. By the classification theorem of surfaces, we
may take $\eta^{\pm1} = \gamma_{k+1}\cdots\gamma_r[\alpha_1,
\beta_1]\cdots[\alpha_h, \beta_h]$ for some $1\leq k \leq r-1$ and $0
\leq h \leq g$. The fundamental group $\pi_1(S,*)$ is freely generated by 
$\{\alpha_i, \beta_i (1 \leq i\leq g), 
\gamma_j (2 \leq j\leq r)\}$. We define 
$\theta(\alpha_i) := \exp([\alpha_i])$, 
$\theta(\beta_i) := \exp([\beta_i])$, $1 \leq i \leq g$, and
$\theta(\gamma_j) := \exp([\gamma_j])$, $2 \leq j\leq r-1$. 
We denote $\log\theta(x) = \sum^\infty_{p=1}\ell_p(x)$, $\ell_p(x) \in
\widehat{\mathcal{L}}\cap H^{\otimes p}$. We define $\ell_p(\gamma_r)$ by
induction on $p \geq 1$. $\ell_1(\gamma_r)$ must be $[\gamma_r]$. Assume
$p\geq2$ and $\ell_q(\gamma_r)$ is defined for $q \leq p-1$. Then we
define $\ell_p(\gamma_r)$ to be minus the $H^{\otimes p}$-component of 
$[\gamma_{k+1}]*\cdots*[\gamma_{r-1}]*
\left(\sum^{p-1}_{q=1}\ell_q(\gamma_r)\right)*
[\alpha_1]*[\beta_1]*(-[\alpha_1])*(-[\beta_1])*\cdots*
[\alpha_h]*[\beta_h]*(-[\alpha_h])*(-[\beta_h])$. From this definition we
have $\ell_p(\eta^{\pm1}) = \pm\ell_p(\eta) = 0$ for $p \geq 2$, namely
$\theta(\eta) = \exp([\eta])$.\par
(3) By the classification theorem of surfaces, we may take 
$\eta^{\pm1} = [\alpha_1,\beta_1]\cdots[\alpha_h, \beta_h]$ for some
$0\leq h \leq g$.  There exists a group-like expansion $\theta'$ of the
free group 
$\langle\alpha_i, \beta_i (1 \leq i\leq h)\rangle$ satisfying 
$
\theta'([\alpha_1, \beta_1]\cdots[\alpha_h, \beta_h])
= \exp(\sum^h_{i=1}[\alpha_i]\wedge[\beta_i])
$, c.f., Definition \ref{def:symp-exp} which is originally due to \cite{Mas}.
See also \cite{Ku1}. The fundamental group $\pi_1(S,*)$ is 
freely generated by $\{\alpha_i, \beta_i (1 \leq i\leq g), 
\gamma_j (2 \leq j\leq r)\}$. We can define $\theta$ by 
$\theta(\alpha_i) := \theta'(\alpha_i)$, 
$\theta(\beta_i) := \theta'(\beta_i)$, if $1 \leq i \leq h$, 
$\theta(\alpha_i) := \exp([\alpha_i])$, 
$\theta(\beta_i) := \exp([\beta_i])$, if $h+1 \leq i \leq g$, and
$\theta(\gamma_j) := \exp([\gamma_j])$, $2 \leq j\leq r$. 
Then we have $\theta(\eta) =
\exp(\pm\sum^h_{i=1}[\alpha_i]\wedge[\beta_i])$.
It is easy to show that $\eta'_0 := \pm\sum^h_{i=1}[\alpha_i]\wedge 
[\beta_i] \in \Lambda^2H_1(S; \mathbb{Z})$ is primitive 
as an element of $H_1(S; \mathbb{Z})^{\otimes 2}$.
\end{proof}

We remark $[\eta] \in H_1(S; \mathbb{Z})$ is primitive in the case (i), 
and $\eta'_0 \in \Lambda^2H_1(S; \mathbb{Z})$ is also primitive 
as an element of $H_1(S; \mathbb{Z})^{\otimes 2}$ in the case (ii).
Hence one can apply Lemmas \ref{32center1} and \ref{32center2} 
to these primitive elements, respectively.

\begin{proof}[Proof of Proposition \ref{32simple}] It is clear $K[[\eta-1]]
\subset Z(\eta)$. Let $u \in Z(\eta)$. For any $n \in
\mathbb{Z}_{\geq0}$ we have $0 = u\eta^n - \eta^nu = u\exp(n\log\eta) -
\exp(n\log\eta)u$. It follows from Lemma \ref{32series} that 
$u(\log\eta) = (\log\eta)u$. Apply the group-like expansion $\theta$ in
Proposition \ref{32expn} to this equation, we obtain 
$
\theta(u)\theta(\log\eta) = \theta(\log\eta)\theta(u) \in\widehat{T}.
$
Here we remark $\theta(\log\eta) \in H$ or $\theta(\log\eta) \in
\Lambda^2H$. Since $C$ is {\it not} null-homotopic in $S$, we have 
$\theta(\log\eta) \neq 0$. Hence, from Lemmas \ref{32center1} and
\ref{32center2}, $\theta(u) \in K[[\theta(\log\eta)]]$. 
Since $\theta\colon \widehat{K\pi_1(S,*)} \to \widehat{T}$ is an isomorphism
of algebras, we have $u \in K[[\log\eta]] = K[[\eta-1]]$. This completes
the proof.
\end{proof}

\section{Completion of the Goldman Lie algebra}
\label{sec:4}

Let $K$ be a commutative ring including the rationals $\mathbb{Q}$.
Let $S$ be an oriented surface, $E$ a non-empty closed subset
with the property that $E\setminus\partial S$ is closed in $S$.
In this section we recall the Goldman Lie algebra of an oriented surface,
and look at its action by derivations on the $K$-SAC $K\CC(S,E)$.
This action turns out to be compatible with the filtrations,
and we are naturally lead to the definition of the completed Goldman Lie algebra
and its action on the completion $\widehat{K\CC(S,E)}$.

We denote by $\hat{\pi}(S) = [S^1, S]$ the homotopy set of free loops on
$S$.  The free $K$-module over the set
$\hat{\pi}(S)$, $K\hat{\pi}(S)$, has a natural structure of a $K$-Lie
algebra, called {\it the Goldman Lie algebra of $S$}, as follows \cite{Go}. 
For any $q \in S$ we denote by $|\ |\colon \pi_1(S,q)\to
\hat{\pi}(S)$ the natural map forgetting the basepoint $q$. 
For a loop $\alpha\colon S^1\to S$ and a simple point $p\in \alpha$,
let $\alpha_p$ be the oriented loop $\alpha$ based at $p$.
Let $\alpha$ and $\beta$ be immersed loops in $S$ such that
$\alpha \cup \beta \colon S^1\cup S^1 \to S$
is an immersion with at worst transverse double points.
For each intersection $p\in \alpha \cap \beta$,
the conjunction $\alpha_p\beta_p\in \pi_1(S,p)$ is defined.
Let $\varepsilon(p;\alpha,\beta)\in \{ \pm 1 \}$ be
the local intersection number of $\alpha$ and $\beta$ at $p$ and set
$$
[\alpha,\beta]:=\sum_{p\in \alpha \cap \beta}
\varepsilon(p;\alpha,\beta)|\alpha_p\beta_p| 
\in K\hat{\pi}(S).
$$
This bracket makes the vector space $K\hat{\pi}(S)$ a $K$-Lie algebra.

We denote $S^* := S \setminus (E\setminus\partial S)$. Let $\alpha\colon S^1\to S^*$ be
an immersed loop and $\beta\colon ([0,1], 0,1) \to (S, *_0,*_1)$
an immersed path from $*_0\in E$ to $*_1 \in E$, and suppose 
$\alpha \cup \beta$ has at worst transverse double points.
For each intersection $p\in \alpha \cap \beta$, let $\alpha_p$
and $\varepsilon(p;\alpha,\beta)$ be the same as before and
let $\beta_{*_0p}$ (resp. $\beta_{p*_1}$)
be the path along $\beta$ from $*_0$ to $p$ (resp. $p$ to $*_1$).
Then the conjunction $\beta_{*_0p}\alpha_p\beta_{p*_1}\in 
\Pi S(*_0,*_1)$ is defined.
For such $\alpha$ and $\beta$, define
\begin{equation*}
\sigma(\alpha)\beta:=\sum_{p\in \alpha \cap \beta}
\varepsilon(p;\alpha,\beta) \beta_{*_0p}\alpha_p\beta_{p*_1}
\in K\Pi S(*_0,*_1) = K\CSE(*_0,*_1).
\label{40sigma*}
\end{equation*}
Then, by a similar way to Proposition 3.2.2 \cite{KK1},  
we obtain a well-defined homomorphism of $K$-Lie algebras
$$
\sigma\colon K\hat{\pi}(S^*) \to \Der K\CSE.
$$
\par
On the other hand, in \S\ref{21filtration} we introduced 
a natural homomorphism $\Der K\CSE \to \Der \widehat{K\CSE}$. 
Hence we have a natural homomorphism of $K$-Lie algebras
$$
\sigma\colon K\hat{\pi}(S^*) \to \Der \widehat{K\CSE}. 
$$
\par

\subsection{Filtration on the Goldman Lie algebra}
\label{sec:4-1}

We denote by $S_\lambda$ the connected
component corresponding to $\lambda \in\pi_0S$. Then we have a direct sum
decomposition of $K$-Lie algebras
\begin{equation}
K\hat{\pi}(S) = \bigoplus_{\lambda \in \pi_0S}K\hat{\pi}(S_\lambda). 
\label{33decomp}
\end{equation}
Now we introduce a $K$-submodule $K\hat{\pi}(S)(n) \subset 
K\hat{\pi}(S)$ for any $n \geq 1$. We begin by considering the case $S$ is
connected. Then $\hat{\pi}(S)$ is the set of all conjugacy
classes in the fundamental group $\pi_1(S, q)$ for any $q \in S$, namely, 
the forgetful map $\vert\ \vert\colon \pi_1(S, q) \to \hat{\pi}(S)$ is
surjective.  We denote by $1_q \in \pi_1(S, q)$ the constant loop based
at $q$. We have 
$$
\vert K1_{q_1} + (I\pi_1(S, q_1))^n\vert = 
\vert K1_{q} + (I\pi_1(S, q))^n\vert
$$
for any $n \geq 1$ and any other $q_1 \in S$. In fact, if we choose
$\gamma\in\Pi S(q_1,q)$, then $\gamma(K1_{q} + (I\pi_1(S,
q))^n)\gamma^{-1} = K1_{q_1} + (I\pi_1(S, q_1))^n$. Hence we may define 
$$
K\hat{\pi}(S)(n) := \vert K1_{q} + (I\pi_1(S, q))^n\vert. 
$$
In the general case, we define 
$$
K\hat{\pi}(S)(n) := \bigoplus_{\lambda \in \pi_0S}K\hat{\pi}(S_\lambda)(n)
$$
for any $n \geq 1$. In any cases, we have $K\hat{\pi}(S)(1) =K\hat{\pi}(S)$.

\begin{thm}\label{33filt}
In the situation stated at the beginning of this section, 
we have 
$$
\sigma(K\hat{\pi}(S^*)(n))\subset F_{n-2}\Der(K\CSE)
$$
for any $n \geq 1$. 
\end{thm}
\begin{proof} We may assume $S$ is connected. 
Choose $q \in S\setminus (E\cup\partial S)$. 
By Lemma \ref{13der}, it suffices to show
$\sigma(|u|)(\gamma) \in F_{n-1}K\CC(p_0,p_1)$ for any $n\geq
1$, $u \in I\pi_1(S, q)^n$, $p_0, p_1 \in E$, and $\gamma \in \CC(p_0,
p_1)$. The $K$-module $I\pi_1(S, q)^n$ is generated by the set
$\{(x_1-y_1)(x_2-y_2)\cdots(x_n-y_n); x_i, y_i \in \pi_1(S, q)\}$. Putting 
$x_1, x_2, \dots, x_n$ and $\gamma$ in general position, we may assume $(x_i\cap
\gamma) \cap (x_j\cap \gamma) = \emptyset$ if $i\neq j$. Then we have 
$$
\sigma(\vert x_1 x_2\cdots x_n\vert)(\gamma) 
= \sum^n_{i=1}\sum_{p \in x_i\cap\gamma}\varepsilon(p; x_i,
\gamma)\gamma_{p_0p}(x_i)_{pq}x_{i+1}\cdots x_nx_1\cdots
x_{i-1}(x_i)_{qp}\gamma_{pp_1}, 
$$
and so
\begin{eqnarray*}
&& \sigma(\vert (x_1-y_1)(x_2-y_2)\cdots(x_n-y_n)\vert)(\gamma)\\
&=& \sum^n_{i=1}\sum_{p \in x_i\cap\gamma}\varepsilon(p; x_i,
\gamma)\gamma_{p_0p}(x_i)_{pq}(x_{i+1}-y_{i+1})\cdots
(x_{i-1}-y_{i-1})(x_i)_{qp}\gamma_{pp_1}\\
&&- \sum^n_{i=1}\sum_{p \in y_i\cap\gamma}\varepsilon(p; y_i,
\gamma)\gamma_{p_0p}(y_i)_{pq}(x_{i+1}-y_{i+1})\cdots
(x_{i-1}-y_{i-1})(y_i)_{qp}\gamma_{pp_1},
\end{eqnarray*}
which is in $F_{n-1}K\CC(p_0, p_1)$. 
This completes the proof.
\end{proof}

As a by-product we see that the filtration of $K\hat{\pi}(S)$ is
compatible with the bracket.

\begin{thm}\label{33bracket} Let $S$ be an oriented surface. Then 
we have 
$$
[K\hat{\pi}(S)(n_1), K\hat{\pi}(S)(n_2)] \subset 
K\hat{\pi}(S)(n_1+n_2-2)
$$
for any $n_1, n_2 \geq 1$. 
\end{thm}
\begin{proof} We may assume $S$ is connected. Choose a point $p \in
S\setminus \partial S$, and set $E = \{p\}$. 
By Theorem \ref{33filt}, $\sigma(u)(v) \in F_{n_1+n_2-2}K\CC(p,p) =
(I\pi_1(S, p))^{n_1+n_2-2}$ for any $u \in K\hat{\pi}(S^*)(n_1)$ and 
$v \in (I\pi_1(S, p))^{n_2}$. Hence $[u, \vert v\vert] =
\vert\sigma(u)(v)\vert \in K\hat{\pi}(S)(n_1+n_2-2)$. On the other hand, 
the inclusion homomorphism
$\pi_1(S^*, q) \to \pi_1(S, q)$ is surjective for any $q \in S^* =
S\setminus \{p\}$. Hence the inclusion homomorphism $K\hat{\pi}(S^*)(n)
\to K\hat{\pi}(S)(n)$ is also surjective for any $n \geq 1$.
This proves the theorem. 
\end{proof}

In particular, if $n \geq 2$, $K\hat{\pi}(S)(n)$ is a Lie subalgebra of 
$K\hat{\pi}(S)$ and an ideal of $K\hat{\pi}(S)(2)$. \par
We define 
$$
\widehat{K\hat{\pi}}(S) :=
\varprojlim_{n\to\infty}K\hat{\pi}(S)/K\hat{\pi}(S)(n),
$$
and call it {\it the completed Goldman Lie algebra} of the surface $S$. 
It is a $K$-Lie algebra from Theorem \ref{33bracket}. We define
$$\widehat{K\hat{\pi}}(S)(n):=
\varprojlim_{m\to\infty}K\hat{\pi}(S)(n)/K\hat{\pi}(S)(m)$$
for $n\ge 1$. From Theorem \ref{33bracket} we have
$$[\widehat{K\hat{\pi}}(S)(n_1),\widehat{K\hat{\pi}}(S)(n_2)]
\subset \widehat{K\hat{\pi}}(S)(n_1+n_2-2)$$
for any $n_1,n_2 \ge 1$.
For any $q \in S$ the forgetful map $\vert\,\vert\colon \pi_1(S, q) \to \hat{\pi}(S)$ induces a natural map
$$
\vert\,\vert\colon \widehat{K\pi_1(S, q)} \to \widehat{K\hat{\pi}}(S).
$$
If $E \subset S$ is a non-empty closed subset with the property that
$E\setminus\partial S$ is closed in $S$, then the homomorphism
$\sigma$ induces a natural homomorphism of $K$-Lie algebras
$$
\sigma\colon \widehat{K\hat{\pi}}(S^*) \to \Der(\widehat{K\CSE})
$$
by Theorem \ref{33filt}. Since $\sigma(K\hat{\pi}(S^*)) \subset F_{-1}\Der
K\CC$, we have $\sigma(\widehat{K\hat{\pi}}(S^*)) \subset
F_{-1}\Der\widehat{K\CC}$. This implies that the action of any element of
$\widehat{K\hat{\pi}}(S^*)$ on $\widehat{K\CSE}$ is continuous with
respect to the topology induced by the filtration
$\{F_n\widehat{K\CSE} \}_{n\geq0}$. \par
As will be shown in \S\ref{aglg}, if $S = \Sigma_{g,1}$ and $E \subset 
\partial\Sigma_{g,1}$, then the Lie algebra
$\widehat{K{\hat\pi}}(\Sigma_{g,1})$ is isomorphic to the Lie algebra 
of symplectic derivations of the completed tensor algebra 
$\prod^\infty_{m=0}H_1(\Sigma_{g,1}; K)^{\otimes m}$ and 
$\widehat{K{\hat\pi}}(\Sigma_{g,1})(2)$ is isomorphic to
the degree completion of Kontsevich's `associative' $a_g$ \cite{Kon}.
These isomorphisms are essentially due to our previous work \cite{KK1}.

\subsection{Action on $\CSE^\abel$}
\label{sec:4-2}

Let $S$ be an oriented surface, and $E$ a non-empty closed
subset of $S$ with the property that
$E\setminus\partial S$ is closed in $S$. 
We consider the abelianization of the groupoid $\CC = \CSE$ introduced in
\S \ref{sec:2-2}. From Theorem
\ref{33filt} we have
$\sigma((x-1)^2) :=\sigma(\vert (x-1)^2\vert)
\in F_0\Der\mathbb{Z}\CSE$ for any $q \in S^*$ and $x \in \pi_1(S^*, q)$.
Hence it induces a derivation of $\CSE^\abel =
\mathbb{Z}\CSE/F_2\mathbb{Z}\CSE$, which we denote
$\overline{\sigma}((x-1)^2) \in \Der\CSE^\abel$. 

\begin{lem}\label{34square}
\begin{enumerate}
\item[{\rm (1)}] We have $\overline{\sigma}((x-1)^2)(\gamma) =
2(x\cdot\gamma)\gamma(x-1)
\in \CC^\abel(p_0, p_1)$ for any $p_0, p_1 \in [q] \subset E$ and $\gamma
\in
\CC(p_0,p_1)$. Here $(x\cdot \gamma)\in K$ means the algebraic
intersection number $H_1(S^*;K) \times H_1(S,E;K)\to K$ and
we regard $x-1$ as an element of $H\CC([q])$. 
In particular, ${\sigma}((x-1)^2)(\mathbb{Z}\CC(p_0,p_1)) \subset
F_1\mathbb{Z}\CC(p_0,p_1)$. 
\item[{\rm (2)}] The square $\overline{\sigma}((x-1)^2)^2 \in {\rm
End}(\CC^\abel(p_0, p_1))$ vanishes for any $p_0, p_1 \in [q]
\subset E$.
\end{enumerate}
\end{lem}
\begin{proof}
(1) We remark $0 = (x-1)^2 = (x^2-x) - (x-1) \in H\CC([q])$. 
Put $x$ and $\gamma$ in general position. Then we have 
\begin{eqnarray*}
&& \sigma((x-1)^2)(\gamma) = \sigma(x^2)(\gamma) - 2\sigma(x)(\gamma)\\
&=& 2 \sum_{p\in x\cap\gamma}
\varepsilon(p;x,\gamma)\gamma_{p_0p}{x_p}^2\gamma_{pp_1} 
- 2 \sum_{p\in x\cap\gamma}
\varepsilon(p;x,\gamma)\gamma_{p_0p}{x_p}\gamma_{pp_1} \\
&=& 2 \gamma\sum_{p\in x\cap\gamma}
\varepsilon(p;x,\gamma){\gamma_{pp_1}}^{-1}({x_p}^2 - x_p)\gamma_{pp_1},
\end{eqnarray*}
and so 
$$
\overline{\sigma}((x-1)^2)(\gamma) = 2(x\cdot\gamma)\gamma(x^2-x) =
2(x\cdot\gamma)\gamma(x-1).
$$
\par
(2) We write simply $D = \overline{\sigma}((x-1)^2)$. Then we have 
\begin{eqnarray*}
&& D^2\gamma = 2(x\cdot\gamma)D(\gamma x-\gamma)\\
&=& 2(x\cdot\gamma)\cdot 2(x\cdot\gamma x)\gamma x(x-1)
- 2(x\cdot\gamma)\cdot 2(x\cdot\gamma)\gamma (x-1)\\
&=& 4(x\cdot\gamma)^2\gamma(x-1)^2 = 0.
\end{eqnarray*}
The last equality follows from $(x-1)^2 = 0 \in H\CC([q])$. 
This completes the proof.
\end{proof}

\subsection{Homological interpretation of the completed Goldman Lie algebra}
\label{sec:4-3}

Let $S$ be a surface of finite type and not closed. 
In this subsection we give a natural identification of the completed
Goldman Lie algebra of $S$ with the first homology group of the surface
with some twisted coefficients. As a corollary, we prove that
the completion map of the Goldman Lie algebra $\rho\colon K{\hat\pi}(S)/K1
\to \widehat{K{\hat\pi}}(S)$ is injective. Here $1 \in \hat\pi(S)$ is the
constant loop. For the proof of this identification we use a group-like
expansion of the fundamental group $\pi_1(S)$ \cite{Mas}. 
We adopt the same notation as in \S \ref{sec:3-3}.

First of all, we introduce two local systems $\mathcal{S}^c(S)$ and
$\widehat{\mathcal{S}}^c(S)$ on $S$. The stalks at $p \in S$ are given by 
$$
\mathcal{S}^c(S)_p := K\pi_1(S,p),\quad\mbox{and}\quad
\widehat{\mathcal{S}}^c(S)_p := \widehat{K\pi_1(S,p)},
$$
respectively. Since $\pi_1(S)$ is a free group, 
the completion map $\mathcal{S}^c(S)\to \widehat{\mathcal{S}}^c(S)$ 
is injective \cite{Bou}, and $H_2(S;
\widehat{\mathcal{S}}^c(S)/\mathcal{S}^c(S)) = 0$. 
Hence the induced homomorphism 
$H_1(S; \mathcal{S}^c(S)) \to H_1(S; \widehat{\mathcal{S}}^c(S))$ 
is injective. In the sequel we regard the former as a submodule of
the latter by this injection. \par
In \cite{KK1} \S3.4, we introduce a $K$-linear map
$$
\lambda\colon K\hat\pi(S) \to H_1(S; \mathcal{S}^c(S)).
$$
This maps $\alpha \in \hat\pi(S)$ to the homology class induced by the
section $s_\lambda(\alpha) \in \Gamma(\alpha^*\mathcal{S}^c(S))$ 
given by $s_\lambda(\alpha)(t) := \alpha_{\alpha(t)} \in
K\pi_1(S,\alpha(t))$, $t \in S^1$. The kernel of the map $\lambda$ is
spanned by the constant loop $1$, $\Ker\lambda = K1$. See \cite{KK1} Proposition 3.4.3 (1).

The $K$-bilinear map $\mathcal{B}_p\colon K\pi_1(S, p)\otimes K\pi_1(S, p)
\to K\hat\pi(S)$, $u \otimes v\mapsto \vert uv\vert$, with the
intersection form on the surface $S$ defines the pairing $\mathcal{B}(\
\cdot \ )\colon H_1(S;
\mathcal{S}^c(S))^{\otimes 2} \to K\hat\pi(S)$. As was shown in \cite{KK1}
Proposition 3.4.3 (2), we have 
\begin{equation}
[u,v] =
\mathcal{B}(\lambda(u)\cdot\lambda(v))
\label{43B}\end{equation}
for any $u$ and $v \in
K\hat\pi(S)$. Similarly we have the pairing $\widehat{\mathcal{B}}(\
\cdot \ )\colon H_1(S;\widehat{\mathcal{S}}^c(S))^{\otimes 2} \to
\widehat{K\hat\pi}(S)$ induced by the $K$-bilinear map
$\widehat{\mathcal{B}}_p\colon
\widehat{K\pi_1(S,p)}\otimes \widehat{K\pi_1(S,p)}\to
\widehat{K\hat\pi}(S)$, $u \otimes v
\mapsto \vert uv\vert$. 

\begin{thm}\label{43isom} Let $S$ be a surface of finite type
and not closed, and $K$ a commutative ring including $\mathbb{Q}$.
Then the map $\lambda$ extends to an isomorphism 
$$
\widehat{\lambda}\colon \widehat{K\hat\pi}(S) \overset\cong\to
H_1(S;\widehat{\mathcal{S}}^c(S)).
$$
This satisfies 
$$
\widehat{\lambda}([u,v]) =
\widehat{\mathcal{B}}(\widehat{\lambda}(u)\cdot\widehat{\lambda}(u))
$$
for any $u$ and $v \in \widehat{K\hat\pi}(S)$. 
\end{thm}
As a corollary we have

\begin{cor}\label{43inj} Let $S$ be a surface of finite type
and not closed, and $K$ a commutative ring including $\mathbb{Q}$.
Then the kernel of the completion map of the Goldman Lie
algebra $\rho\colon K{\hat\pi}(S) \to \widehat{K{\hat\pi}}(S)$ is
spanned by the constant loop $1 \in \hat\pi(S)$. 
\end{cor}
\begin{proof} Consider the commutative diagram
$$
\begin{CD}
K\hat\pi(S) @>{\lambda}>> H_1(S;\mathcal{S}^c(S))\\
@VVV @VVV\\
\widehat{K\hat\pi}(S) @>{\widehat{\lambda}}>>
H_1(S;\widehat{\mathcal{S}}^c(S)).
\end{CD}
$$
From Theorem \ref{43isom} the map $\widehat{\lambda}$ is an
isomorphism, while the right vertial arrow is injective. 
Hence the kernel of the completion map equals $\Ker\lambda = K1$. 
\end{proof}

To prove Theorem \ref{43isom}, we use a group-like
expansion of the fundamental group $\pi_1(S)$. 
Choose a basepoint $* \in S$ and denote $\pi:=\pi_1(S, *)$. 
Let $\theta\colon
\pi \to \widehat{T}$ be a group-like expansion. See Definition
\ref{33Mas}. 
\par
We define a $K$-linear map $N\colon \widehat{T} \to \widehat{T}$ 
by $N\vert_{H^{\otimes 0}} := 0$ and 
\begin{equation}
N(X_1\cdots X_n) :=
\sum^n_{i=1}X_i\cdots X_nX_1\cdots X_{i-1},
\label{43N}\end{equation}
for $X_j \in H$, $n \geq 1$. 
Then $\theta$ induces an isomorphism 
$\theta_*\colon H_1(S;\widehat{\mathcal{S}}^c(S)) \overset\cong\to
N(\widehat{T}_1)$ (\cite{KK1} (5.3.1), Lemma 6.1.1). 
Moreover the composite $\theta_*\circ\lambda\colon K\hat\pi(S) \to
N(\widehat{T}_1)$ equals the map $\lambda_\theta\colon K{\hat\pi}(S)
\to N(\widehat{T}_1)$ defined by $\lambda_\theta(\vert x\vert) :=
N\theta(x)$, $x \in \pi$ (\cite{KK1} Lemma 6.3.2).
Here we should remark that the proofs of the lemma and the proposition in
\cite{KK1} work well for group-like expansions over a commutative ring
including $\mathbb{Q}$ as well as for symplectic
expansions over the rationals $\mathbb{Q}$.
The key to proving the injectivity of the map $\widehat{\lambda}$ 
is the following lemma. 
\begin{lem}\label{43filter} For any $n \geq 1$ we have 
$$
{\lambda_\theta}^{-1}(N(\widehat{T}_n)) = \vert K1+I\pi^n\vert 
\ (= K{\hat\pi}(S)(n)).
$$
\end{lem}

To prove this lemma, we need the following.
\begin{lem}\label{43exact}
$0 \to K1\oplus [\widehat{T}, \widehat{T}] \hookrightarrow \widehat{T}
\overset{N}\to N(\widehat{T}_1) \to 0$ {\rm (exact)}.
\end{lem}

\begin{proof} Since $N(1) = 0$ and $N([u,v]) = N(uv-vu) = 0$ for $u, v \in
\widehat{T}$, we have $K1\oplus [\widehat{T}, \widehat{T}] \subset \Ker
N$. Since $N$ is homogeneous it suffices to show 
$$
(\Ker N) \cap H^{\otimes n} \subset (K1\oplus [\widehat{T},
\widehat{T}]) \cap H^{\otimes n},
$$
for any $n\ge 0$.
It is clear in the case $n=0$. For any $X_j \in H$ and  $n\geq 1$, we
have 
\begin{eqnarray*}
X_1\cdots X_n - \frac1n N(X_1\cdots X_n) &=& 
\frac1n\sum^n_{i=1}X_1\cdots X_n - X_i\cdots X_nX_1\cdots X_{i-1}\\
&=& \frac1n\sum^n_{i=2}[X_1\cdots X_{i-1}, X_i\cdots X_n]
\in [\widehat{T},\widehat{T}].
\end{eqnarray*}
Hence $u - \frac1n Nu
\in [\widehat{T},\widehat{T}]$ for any $u \in H^{\otimes n}$.
This proves the lemma. 
\end{proof}

\begin{proof}[Proof of Lemma \ref{43filter}] Since $\theta(I\pi^n) \subset
\widehat{T}_n$,  we have $\lambda_\theta(\vert I\pi^n\vert) \subset
N(\widehat{T}_n)$. Clearly $\lambda_\theta(1) = N(1) = 0$. Hence
$\vert K1+I\pi^n\vert
\subset {\lambda_\theta}^{-1}(N(\widehat{T}_n))$. \par
Suppose $u \in K\pi$ satisfies $N\theta(u) \in N(\widehat{T}_n)$.  
From Lemma \ref{43exact}, $\theta(u) \in \widehat{T}_n + K1+
[\widehat{T},\widehat{T}]$. This means $\theta(u-\varepsilon(u)1) 
- \sum^m_{i=1}[v'_i,w'_i] \in \widehat{T}_n$ for some $v'_i,w'_i \in
\widehat{T}$. There exist $v_i$ and $w_i \in K\pi$ such that
$v'_i - \theta(v_i)$ and $w'_i-\theta(w_i) \in \widehat{T}_n$, since 
$\widehat{T}/\widehat{T}_n
\overset\theta\cong\widehat{K\pi}/\widehat{I\pi^n} = K\pi/I\pi^n$
by (\ref{expansion}).
This implies $\theta(u - \varepsilon(u)1 - \sum^m_{i=1}[v_i,w_i]) \in
\widehat{T}_n$. Since $\widehat{T}/\widehat{T}_n
\cong K\pi/I\pi^n$, $z := u - \varepsilon(u)1 - \sum^m_{i=1}[v_i,w_i]
\in I\pi^n$. Hence we have $\vert u\vert = \vert\varepsilon(u)1 + z\vert 
\in \vert K1 + I\pi^n\vert$. This completes the proof.
\end{proof}

\begin{proof}[Proof of Theorem \ref{43isom}] From Lemma \ref{43filter}
the map $\lambda_\theta$ induces an injective linear map
$
\lambda_\theta\colon K{\hat\pi}(S)/K{\hat\pi}(S)(n)
\hookrightarrow N(\widehat{T}_1)/N(\widehat{T}_n)
$
for any $n \geq 1$. This map is surjective. In fact, for any $u \in
\widehat{T}_1$, there exists some $z \in K\pi$ such that $u - \theta(z)
\in \widehat{T}_n$, and so $Nu - \lambda_\theta(z) \in N(\widehat{T}_n)$.
Hence we have an isomorphism 
$$
\lambda_\theta\colon K{\hat\pi}(S)/K{\hat\pi}(S)(n)
\overset\cong\to N(\widehat{T}_1)/N(\widehat{T}_n)
$$
for any $n \geq 1$. Taking the projective limits, we obtain an
isomorphism 
$\widehat{\lambda}_\theta\colon \widehat{K{\hat\pi}}(S) \overset\cong\to
\varprojlim_{n\to \infty}N(\widehat{T}_1)/N(\widehat{T}_n) =
N(\widehat{T}_1)$, which preserves the filtrations. 
Hence we obtain an isomorphism
$\widehat{\lambda} := {\theta_*}^{-1}\circ\widehat{\lambda}_\theta\colon 
\widehat{K{\hat\pi}}(S) \to H_1(S;\widehat{\mathcal{S}}^c(S))$, 
which is independent of the choice of a group-like expansion $\theta$.
The latter half of the theorem follows immediately from (\ref{43B}). 
This completes the proof. 
\end{proof}

In the proof of Theorem 4.3.1, we have obtained the following corollary.
\begin{cor}\label{43Ntheta} Let $S$ be a surface of finite type and not closed, 
$K$ a commutative ring including $\mathbb{Q}$, and $\theta$ a group-like expansion of $\pi_1(S)$,
the fundamental group of $S$. Then the map 
$$
\lambda_\theta: \widehat{K\hat\pi}(S) \to N(\widehat{T}_1),
$$
defined by $\lambda_\theta(\vert x\vert) := N\theta(x)$ for $x \in \pi_1(S)$, 
is an isomorphism.  
\end{cor}

\section{Dehn twists}
\label{sec:5}

In this section we suppose $K$ is a commutative ring including the rationals $\mathbb{Q}$.
We shall generalize results in \cite{KK1} to any
oriented surface $S$, and those in \cite{Ku2} to any surface of finite
type with non-empty boundary. Let $E$ be a non-empty closed subset of
$S$ with the property that $E\setminus\partial S$ is closed in $S$.  We
consider the completed Goldman Lie algebra
$\widehat{K\hat{\pi}}(S^*)$ of $S^* = S\setminus
(E\setminus\partial S)$ and the homomorphism of Lie algebras
$\sigma\colon \widehat{K\hat{\pi}}(S^*) \to \Der(\widehat{K\CSE})$.

\subsection{An invariant of unoriented free loops}
\label{sec:5-1}

We begin by defining an invariant of an unoriented free loop $C$ in 
$S\setminus (E\cup \partial S) = S^* \setminus \partial S$. 
\begin{lem}\label{51ft}
Let $f(t) \in K[[t-1]]$ be a formal power series in $t-1$ with $f(1) =
f'(1) = 0$.
\begin{enumerate}
\item[{\rm (1)}] For $\alpha \in {\hat\pi}(S^*)$, choose a point $q \in S^*$ and
a based loop $x \in \pi_1(S^*,q)$ representing the free loop
$\alpha$. Then 
$$
f(\alpha) := \vert f(x)\vert \in \widehat{K\hat{\pi}}(S^*)(2)
$$
is well-defined. In other words, $\vert f(x)\vert$ does not depend on the
choice of $q$ and $x$.
\item[{\rm (2)}] If $f(t) = f(t^{-1})$, then $f(\alpha) = f(\alpha^{-1})$. In
particular, we may define $f(C):= f(\alpha)\in
\widehat{K\hat{\pi}}(S^*)(2)$ if $C = \alpha^{\pm1}$, namely, 
the unoriented free loop $C$ is represented by an oriented free loop
$\alpha$. 
\item[{\rm (3)}] $\sigma(f(\alpha)) \in \Der\widehat{K\CC}$ satisfies the three
conditions {\rm (i)-(iii)} in {\rm Lemma \ref{13conv}}. In particular, we can define 
the exponential $\exp(\sigma(f(\alpha)))
\in \Aut\widehat{K\CC}$. It satisfies the conditions {\rm (i)-(iii)} in
{\rm Definition \ref{32ASE}}.  
\end{enumerate}
\end{lem}
\begin{proof}(1) Suppose $q_1 \in S^*$ and $x_1 \in \pi_1(S^*, q_1)$
satisfy $\vert x\vert = \vert x_1\vert$. Then we have $x_1 =
\gamma^{-1}x\gamma$ for some
$\gamma\in\Pi S^*(q,q_1)$, and so $f(x_1) =
\gamma^{-1}f(x)\gamma \in \widehat{K\Pi S^*}(q_1,q_1)$. 
This implies $\vert f(x_1)\vert = \vert  f(x)\vert \in
{\widehat{K\hat{\pi}}(S^*)}$.\par
(2) is clear. 
\par
(3) Since $f(\alpha) \in \widehat{K\hat{\pi}}(S^*)(2)$, we have 
$\sigma(f(\alpha)) \in F_0\Der\widehat{K\CC}$ by Theorem \ref{33filt}. 
On the other hand, $\sigma((x-1)^2)$ satisfies the conditions (ii) and
(iii) from Lemma \ref{34square}. 
Now we have
$$
f(x) \equiv c(x-1)^2 \bmod{(x-1)^3},
$$
for some constant $c \in K$, and any element in $(x-1)^3K[[x-1]]$ induces
an element of
$F_1\Der\widehat{K\CC}$ by Theorem \ref{33filt}. Hence 
$\sigma(f(\alpha))$ satisfies all the conditions (i)-(iii) in Lemma
\ref{13conv}. 
The condition (i) in Definition \ref{32ASE} for $\exp(\sigma(f(\alpha)))$
follows from the fact $\sigma(f(\alpha)) \in F_0\Der\widehat{K\CC}$, 
(ii) from $\alpha \cap E = \emptyset$, and (iii) from Lemma
\ref{34square}. 
\end{proof}
Now we define 
$$
L(t) := \frac12(\log t)^2 \in \mathbb{Q}[[t-1]].
$$
Here we remark $tL'(t) = \log t$. 
From Lemma \ref{51ft} we obtain $L(C) \in \widehat{K\hat{\pi}}(S^*)(2)$
and $\exp(\sigma(L(C))) \in \Aut\widehat{K\CC}$ 
for any unoriented free loop $C$ in $S^*$. Furthermore we have
\begin{lem}\label{51sLC}
The derivation $\sigma(L(C))$ stabilizes the coproduct $\Delta$,
$$\sigma(L(C)) \in {\rm Der}_{\Delta}\widehat{K\mathcal{C}}.$$
In particular, we have $\exp(\sigma(L(C))) \in \ASE$.
\end{lem}
\begin{proof} It suffices to prove $\sigma(L(C)) \in
\Der_\Delta\widehat{K\CC}$, see \S \ref{21filtration}.
Choose $\alpha \in {\hat\pi}(S^*)$ such that $C = \alpha^{\pm1}$. 
For any $\gamma \in \CC(*_0,*_1)$, $*_0,*_1 \in E$, and $n \geq 0$, 
we have $\sigma(\alpha^n)(\gamma) = \sum_{p\in\alpha\cap\gamma}
n\varepsilon(p;\alpha,\gamma)\gamma_{*_0p}{\alpha_p}^n\gamma_{p*_1}$, 
so that 
$$
\sigma(f(\alpha))(\gamma) =  \sum_{p\in\alpha\cap\gamma}
\varepsilon(p;\alpha,\gamma)\gamma_{*_0p}\alpha_pf'(\alpha_p)\gamma_{p*_1}
$$
for any $f(t) \in K[[t-1]]$. In particular, 
\begin{equation}
\sigma(L(C))(\gamma) =  \sum_{p\in\alpha\cap\gamma}
\varepsilon(p;\alpha,\gamma)\gamma_{*_0p}(\log{\alpha_p})\gamma_{p*_1}.
\label{51LC}
\end{equation}
On the other hand, we have $\Delta(\log\alpha_p) =
(\log\alpha_p)\widehat{\otimes}1 + 1\widehat{\otimes}(\log\alpha_p)
\in \widehat{K\pi_1(S,p)}^{\widehat{\otimes} 2}$. Hence 
\begin{eqnarray*}
&&\Delta\sigma(L(C))\gamma 
= \sum_{p\in\alpha\cap\gamma}
\varepsilon(p;\alpha,\gamma)(\gamma_{*_0p}\widehat{\otimes}\gamma_{*_0p})
(\log{\alpha_p}\widehat{\otimes}1 +
1\widehat{\otimes}\log{\alpha_p})
(\gamma_{p*_1}\widehat{\otimes}\gamma_{p*_1})\\
&=& \sigma(L(C))\gamma\widehat{\otimes}\gamma +
\gamma\widehat{\otimes}\sigma(L(C))\gamma
= (\sigma(L(C))\widehat{\otimes}1 +
1\widehat{\otimes}\sigma(L(C)))\Delta\gamma.
\end{eqnarray*}
This means $\sigma(L(C)) \in \Der_\Delta\widehat{K\CC}$, and proves the
lemma. 
\end{proof}

Let us go back to the situation of Proposition \ref{32subs}.
We take $\eta_i \in \pi_1(N, *_i)$ in the positive direction. 
Then we define $F^U \in \widehat{K\hat{\pi}}(N)$ by 
$$
F^U := \sum^n_{i=1} a^U_i L(\partial_iN).
$$
By Lemmas \ref{51ft} and \ref{51sLC}, 
$F^U \in \widehat{K\hat{\pi}}(N)(2)$ and $\exp\sigma(F^U)\in A(N,
\partial N)$. From the construction, we have 
\begin{equation}
U = \exp\sigma(F^U) \in A(N, \partial N).
\label{51subs}
\end{equation}

\subsection{The logarithm of Dehn twists}
\label{sec:5-2}

Recall the Dehn-Nielsen homomorphism $\widehat{\sf DN}\colon \mathcal{M}(S, E)
\to {\rm Aut}(\widehat{K\CSE})$ from \S \ref{sec:3}. The following theorem is a
generalization of a part of our previous result \cite{KK1} Theorem 1.1.1 to any oriented surfaces.
It does not involve a symplectic expansion and the total Johnson map.

\begin{thm}\label{42logDT} Let $S$ be an oriented surface and $E$ a
non-empty closed subset of $S$ with the property that $E\setminus \partial S$ is
closed in $S$. Then the Dehn-Nielsen homomorphism $\widehat{\sf
DN}$ maps the right handed Dehn twist $t_C$ along a simple closed curve
$C$ in $S\setminus (E\cup\partial S)$ to 
$$
\widehat{\sf DN}(t_C) = \exp(\sigma(L(C))) \in {\rm Aut}\widehat{K\CSE}. 
$$
\end{thm}
\begin{proof}
We begin by computing $\widehat{\sf DN}(t_C)$ in the case $S$
is an annulus $\mathfrak{a} = S^1\times [0,1]$. We regard $S^1 =
[0,1]/(0\sim1)$, and define $p_0 := (0\bmod \sim, 0)$, $p_1 := (0\bmod
\sim, 1)$, $E := \{p_0, p_1\}$ and $\CC = \CC(\mathfrak{a}, E)$. 
Consider a path $\gamma_0\colon [0,1] \to \mathfrak{a}$ given by 
$t \in [0,1] \mapsto (0\bmod \sim, t)$, a based loop $x \in
\pi_1(\mathfrak{a}, p_1)$ given by $t \in [0,1] \mapsto (t\bmod \sim,
1)$, and a simple closed curve $C = \vert x^{\pm1}\vert$. We have 
$\sigma(\vert x^n\vert)(\gamma_0) = n\gamma_0x^n$ for any $n \geq 0$, 
and $\sigma(\vert x^n\vert)$ acts trivially on $K\CC(p_0, p_0)$ and 
$K\CC(p_1,p_1)$. Hence, for any formal power series $f(x) \in
\widehat{K\pi_1(\mathfrak{a}, p_1)}$ in $x-1$, the derivation 
$\sigma(f(x))$ acts trivially on $\widehat{K\CC}(p_0,p_0)$ and 
$\widehat{K\CC}(p_1,p_1)$, and $\sigma(f(x))(\gamma_0) = \gamma_0xf'(x)
\in \widehat{K\CC}(p_0,p_1)$. Since $tL'(t) = \log(t)$, 
$\sigma(L(C))(\gamma_0) = \sigma(L(x))(\gamma_0) = \gamma_0\log x$. 
Clearly $\exp(\sigma(L(C)))(x) = x = {\sf DN}(t_C)(x)$. Hence we have 
$$
\exp(\sigma(L(C)))(\gamma_0) = \gamma_0x = {\sf DN}(t_C)(\gamma_0).
$$
This proves 
\begin{equation}
\exp(\sigma(L(C))) = \widehat{\sf DN}(t_C) \in {\rm
Aut}\widehat{K\CC(\mathfrak{a},E)},
\label{42annulus}
\end{equation}
namely, the theorem in the case $S$ is an annulus. \par
Next we consider the general case. Choose a closed tubular neighborhood 
$\mathfrak{a}$ of the simple closed curve $C$ in the surface $S\setminus
(E\cup\partial S)$. The boundary $\partial\mathfrak{a}$ has two connected
components $\partial_0\mathfrak{a}$ and $\partial_1\mathfrak{a}$. 
Choose a point $q_i$ on each $\partial_i\mathfrak{a}$, $i = 0,1$. 
We define $S_1 := S\setminus {\rm int}\mathfrak{a}$, $S_2 :=
\mathfrak{a}$, $E_1 := E \cup \{q_0,q_1\}$ and $E_2 := \{q_0, q_1\}$. 
Then, in the setting of
Proposition \ref{32vKT}, we have $S_3 = S_1\cup S_2 = S$, $E_3 = E_1$ 
and $K\CC_3 = K\CC(S, E_1)$ is generated by $K\CC_1$ and $K\CC_2$. 
We may regard $L(C) \in \widehat{K\hat{\pi}}({S_3}^*)(2)$ for ${S_3}^* =
S^* \setminus \{q_0,q_1\}$. Both of $\exp(\sigma(L(C)))$ and $\widehat{\sf
DN}(t_C)$ act trivially on $K\CC_1$, and coincide with each other on
$K\CC_2$ by (\ref{42annulus}). Hence, by Proposition \ref{32vKT}, they
coincide with each other on $K\CC_3= K\CC(S, E_1)$. Since both of them are
continuous and $K\CC_3$ is dense in $\widehat{K\CC_3}$, they coincide
with each other on $\widehat{K\CC_3}$. Since
$\widehat{K\CSE}$ is a full subcategory of $\widehat{K\CC_3}$, 
they coincide with each other on $\widehat{K\CSE}$. 
This completes the proof of the theorem. 
\end{proof}

\subsection{Generalized Dehn twists and their localization}
\label{sec:5-3}

Theorem \ref{42logDT} motivates us to define
a generalization of Dehn twists for not necessarily simple
loops. Let $C$ be an unoriented free loop in $S^* \setminus \partial S$.

\begin{dfn}
The generalized Dehn twist along $C$ is defined to be
$$t_{C}:=\exp (\sigma(L(C))) \in A(S,E) \subset
{\rm Aut}\widehat{K\mathcal{C}(S,E)}.$$
\end{dfn}

The case $S=\Sigma_{g,1}$ and $E=\{ *\}$, where
$* \in \partial S$, is treated in \cite{Ku2}.

It is natural to ask whether $t_C$ is realizable as
a diffeomorphism, i.e., is in the image
of $\widehat{{\sf DN}} \colon \mathcal{M}(S,E) \to
{\rm Aut}\widehat{K\mathcal{C}(S,E)}$.
We show that if $t_C$ is realizable as a diffeomorphism, then it is localized
inside a regular neighborhood of $C$.
To restrict ourself to the case $\widehat{{\sf DN}}$ is injective,
hereafter we assume $S$ is of finite type with non-empty boundary and work
under the assumption of Theorem \ref{31DN}.

Let ${\rm End}(K\hat{\pi}(S))$ be the space of
filter-preserving endomorphisms of $K\hat{\pi}(S)$.
Also let ${\rm Aut}(K\hat{\pi}(S))$
be the group of filter-preserving $K$-linear automorphisms
of $K\hat{\pi}(S)$.

\begin{lem}
\begin{enumerate}
\item[{\rm (1)}]
Let $D\in {\rm Der}K\mathcal{C}(S,E)$.
For $a \in K\mathcal{C}(p,p)$, where $p\in E$, set $|D|(|a|):=|D(a)|$.
Then this defines a well-defined $K$-linear map
$|\ |\colon {\rm Der}K\mathcal{C}(S,E) \to
{\rm End}(K\hat{\pi}(S))$.
\item[{\rm (2)}]
Let $U\in {\rm Aut}(K\mathcal{C}(S,E))$. For $a\in K\mathcal{C}(p,p)$,
where $p\in E$, set $|U|(|a|):=|U(a)|$. Then this defines a well-defined group
homomorphism $|\ |\colon {\rm Aut}K\mathcal{C}(S,E)
\to {\rm Aut}(K\hat{\pi}(S))$.
\end{enumerate}
\end{lem}

\begin{proof}
Let $p,q \in E$ and assume $[p]=[q]=\lambda\in \pi_0\mathcal{C}$.
Recall that $|\ |\colon K\mathcal{C}(p,p) \to K\hat{\pi}(S_{\lambda})$
is surjective for any $p\in E$. Take some $\gamma \in \Pi S(q,p)$
and let $a\in K\mathcal{C}(p,p)$.

(1) It is sufficient
to prove $|D(a)|=|D(\gamma a \gamma^{-1})|$. First of all,
since $0=D(1)=D(\gamma\gamma^{-1})=(D\gamma)\gamma^{-1}
+\gamma D(\gamma^{-1})$, we have $D(\gamma^{-1})
=-\gamma^{-1}(D\gamma)\gamma^{-1}$. We compute
$D(\gamma a \gamma^{-1})=(D\gamma)a \gamma^{-1}+
\gamma (Da)\gamma^{-1}+\gamma a D(\gamma^{-1})
=\gamma (Da)\gamma^{-1}+(D\gamma)a \gamma^{-1}
-\gamma a \gamma^{-1}(D\gamma)\gamma^{-1}$.
Notice that $|(D\gamma)a \gamma^{-1}|=
|\gamma a \gamma^{-1}(D\gamma)\gamma^{-1}|$.
Hence $|D(\gamma a \gamma^{-1})|=|\gamma (Da) \gamma^{-1}|
=|D(a)|$, as desired.

(2) This is clear from $|U(\gamma a \gamma^{-1})|=
|U(\gamma) Ua U(\gamma)^{-1}|=|U(a)|$.
\end{proof}

Any filter-preserving endomorphism (resp. automorphism)
of $K\hat{\pi}(S)$ naturally extends to an endomorphism
(resp. automorphism) of $\widehat{K\hat{\pi}}(S)$.
Consequently we have a $K$-linear map ${\rm End}(K\hat{\pi}(S))
\to {\rm End}(\widehat{K\hat{\pi}}(S))$ and a group homomorphsim
${\rm Aut}(K\hat{\pi}(S)) \to {\rm Aut}(\widehat{K\hat{\pi}}(S))$.
The diagrams
$$
\begin{CD}
{\rm Der}K\mathcal{C}(S,E) @>>> {\rm End}(K\hat{\pi}(S)) \\
@VVV @VVV \\
{\rm Der}\widehat{K\mathcal{C}(S,E)} @>>>
{\rm End}(\widehat{K\hat{\pi}}(S))
\end{CD}
{\rm \quad and\quad}
\begin{CD}
{\rm Aut}K\mathcal{C}(S,E) @>>> {\rm Aut}(K\hat{\pi}(S)) \\
@VVV @VVV \\
{\rm Aut}\widehat{K\mathcal{C}(S,E)} @>>>
{\rm Aut}(\widehat{K\hat{\pi}}(S))
\end{CD}
$$
commute.

The following theorem is a generalization of \cite{Ku2} \S 3.3.

\begin{thm}
\label{localize}
Suppose $S$ is of finite type with non-empty boundary, $E\subset \partial S$,
and any connected component of $\partial S$
has an element of $E$. Let $C$ be an unoriented immersed free loop in
$S \setminus \partial S$ and assume the generalized Dehn twist $t_C$
is in the image of $\widehat{{\sf DN}}$.
Then there is an orientation preserving diffeomorphism $\varphi$
of $S$ fixing $\partial S$ pointwise,
such that $\widehat{{\sf DN}}(\varphi)=t_{C}$ and
the support of $\varphi$ lies in a regular neighborhood of $C$.
\end{thm}

\begin{proof}
Take a diffeomorphism $\varphi$ such that
$\widehat{{\sf DN}}(\varphi)=t_C$.
We shall deform $\varphi$ by isotopies until it has the desired property.

We claim that if $\delta$ is a proper arc or an oriented loop in $S$ that is disjoint from $C$,
then $\varphi(\delta)$ is isotopic to $\delta$. The case $\delta$ is a proper
arc is clear from $\sigma(L(C))\delta=0$.
To prove the case $\delta$ is an oriented loop, consider the composite
$\mathcal{M}(S,E)\stackrel{\widehat{{\sf DN}}}{\to}
{\rm Aut}\widehat{K\mathcal{C}(S,E)} \stackrel{|\ |}{\to} {\rm Aut}(\widehat{K\hat{\pi}}(S))$.
As we have noted in the proof of Theorem \ref{33bracket},
$|\sigma(u)(v)|=[u,|v|]$ for $u\in K\hat{\pi}(S)$ and
$v\in K\mathcal{C}(p,p)$, where $p\in E$. This implies that
$|\sigma(\alpha)|={\rm ad}(\alpha)$ for $\alpha \in K\hat{\pi}(S)$.
Therefore $|\widehat{{\sf DN}}(\varphi)|=\exp({\rm ad}(L(C)))
\in {\rm Aut}(\widehat{K\hat{\pi}}(S))$. Since $\delta$ is disjoint
from $C$, ${\rm ad}(L(C))\delta=[L(C),\delta]=0$. Thus
$|\widehat{{\sf DN}}(\varphi)|\delta=\delta\in \widehat{K\hat{\pi}}(S)$.
By Corollary \ref{43inj}, this implies $|{\sf DN}(\varphi)|\delta-\delta
\in K1$. Since the action of $\mathcal{M}(S,E)$ on $K\hat{\pi}(S)$
preserves the augmentation $K\hat{\pi}(S) \to K, \hat{\pi}\ni x \mapsto 1$,
we conclude $|{\sf DN}(\varphi)|\delta=\delta$. The claim is proved.

Let $N=N(C)$ be a closed regular neighborhood of $C$.
The Euler characteristic of $N$ must be non-positive.
If $C$ is simple, then the assertion is clear from Theorem \ref{42logDT}.
Thus we may assume $N$ is neither diffeomorphic to a disk nor an annulus.
Let $S \setminus {\rm Int}(N) \cong \coprod_{\lambda}S_{\lambda}$
be the decomposition into connected components. Note that for any $\lambda$
we have $S_{\lambda} \cap \partial N \neq \emptyset$.
We shall take a system $\mathcal{B}_{\lambda}$ of simple closed curves
and proper arcs in $S_{\lambda}$ by the following way.
Let $\chi(S_{\lambda})$ be the Euler characteristic of $S_{\lambda}$.

{\it Case 1.} $\chi(S_{\lambda})\ge 0$. Then $S_{\lambda}$ is
one of the following: 
(a) a closed disk, (b) an annulus of which both the boundary components
are in $\partial N$, (c) an annulus of which one of the boundary component
is in $\partial N$, and the other component is in $\partial S$,
(d) a once punctured disk. In these cases we let $\mathcal{B}_{\lambda}$ to be empty.

{\it Case 2.} $\chi(S_{\lambda})=-1$. Then $S_{\lambda}$ is one
of the following: (e) a torus with one boundary component,
(f) a pair of pants of which the three boundary components
are in $\partial N$, (g) a pair of pants of which two boundary
components are in $\partial N$, and the other component is in $\partial S$,
(h) a pair of pants of which one boundary component is in $\partial N$,
and the other two components are in $\partial S$, (i) a once
punctured annulus of which both the boundary components are in $\partial N$,
(j) a once punctured annulus of which one of the boundary component
is in $\partial N$, and the other component is in $\partial S$,
(k) a twice punctured annulus.
In cases (e)(g)(h)(j), let $\mathcal{B}_{\lambda}$ be as in Figure 2.
In cases (f)(i)(k), let $\mathcal{B}_{\lambda}$ be empty.

\begin{figure}
\begin{center}
\input{eghj.tex}

\caption{$\mathcal{B}_{\lambda}$ for (e), (g), (h), and (j)}
\end{center}
\end{figure}
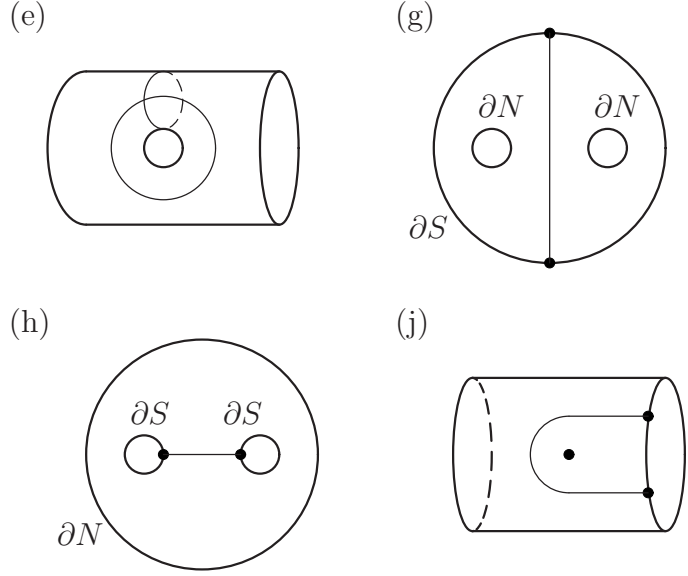

{\it Case 3.} $\chi(S_{\lambda})\le -2$. Let $r$ and $r^{\prime}$
be the cardinality of $\pi_0(\partial N \cap S_{\lambda})$ and
$\pi_0(\partial S \cap S_{\lambda})$, respectively,
and let $g$ be the genus of $S_{\lambda}$ and $n$ the number
of punctures of $S_{\lambda}$. We have $r \ge 1$ and
$2g+r+r^{\prime}+n \ge 4$. If $r^{\prime}>0$, let
$\mathcal{B}_{\lambda}$ be as in Figure 3. If $r^{\prime}=0$
and $g>0$, let $\mathcal{B}_{\lambda}$ be as in Figure 4.
If $r^{\prime}=g=0$, then $r+n \ge 4$. We let $\mathcal{B}_{\lambda}$
be as in Figure 5.

Finally we set $\mathcal{B}=\bigcup_{\lambda}\mathcal{B}_{\lambda}$.
Then $\mathcal{B}$ has the following properties.

\begin{enumerate}
\item[(1)] Any member of $\mathcal{B}$ is disjoint from $\partial N$.
\item[(2)] Any simple closed curve in $\mathcal{B}$ is not parallel to
a component of $\partial N$ and the ends of any arcs in $\mathcal{B}$ are in $\partial S$,
\item[(3)] Members of $\mathcal{B}$ are pairwise non-isotopic and
pairwise in minimal position in $S$.
\item[(4)] The surface obtained from $S \setminus {\rm Int}(N)$
by cutting along $\mathcal{B}$ is a disjoint union of surfaces of the types
(a), (b), (c), (d), (f), (i), and (k).
\end{enumerate}

It is clear that members of $\mathcal{B}_{\lambda}$ are pairwise
non-isotopic and pairwise in minimal position in $S_{\lambda}$.
The property (3) for $\mathcal{B}$ also follows since
$N$ is not a disk or an annulus, as we remarked before.

\begin{figure}
\begin{center}
\input{chi-1.tex}

\caption{$\mathcal{B}_{\lambda}$ for $r^{\prime}>0$}
\end{center}
\end{figure}
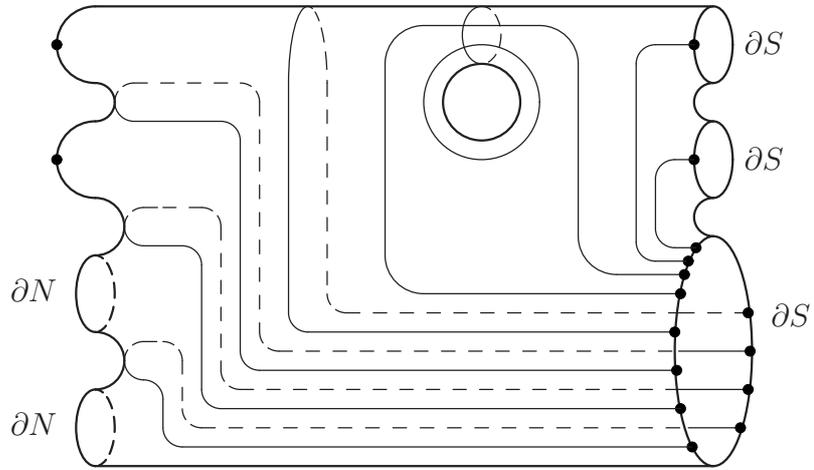

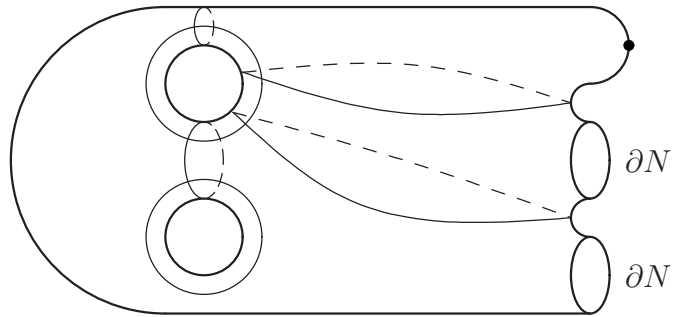
\begin{figure}
\begin{center}
\input{chi-2.tex}

\caption{$\mathcal{B}_{\lambda}$ for $r^{\prime}=0$, $g>0$}
\end{center}
\end{figure}

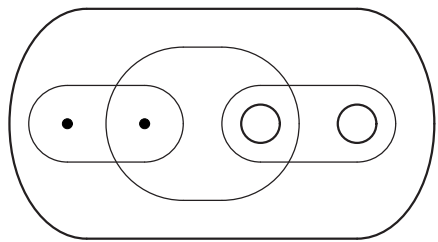
\begin{figure}
\begin{center}
\input{chi-3.tex}

\caption{$\mathcal{B}_{\lambda}$ for $r^{\prime}=g=0$}
\end{center}
\end{figure}

From the claim, each component of $\partial N$ and $\mathcal{B}$
is preserved by $\varphi$ up to isotopy. By \cite{FM} Proposition 2.8,
we may assume that $\varphi$ is the identity on $\partial N$ and $\mathcal{B}$.
Consider the restriction of $\varphi$ to each component
of the result of cutting $S$ along $\partial N$ and $\mathcal{B}$.
Each restriction is a self homeomorphism, and
by the property (4) for $\mathcal{B}$, is isotpic to the identity or to
a product of Dehn twists along the boundary components.
This implies that $\varphi|_{S \setminus {\rm Int}(N)}$
is isotopic to a product of Dehn twists along the boundary components of $N$.
This completes the proof.
\end{proof}

\subsection{The generalized Dehn twist along a figure eight}
\label{sec:5-4}

In this subsection we give a generalization of \cite{Ku2} Theorem 5.2.1.
We suppose $S$ is of finite type  with non-empty boundary,
$E\subset\partial S$, and any connected component of
$\partial S$ has an element of $E$.

Let $C$ be an unoriented immersed free loop in $S^* \setminus \partial S$.
We say $C$ is a {\it figure eight} if the self-intersections of $C$ consist of
a single double point and the inclusion homomorphism $\pi_1(C) \to \pi_1(S)$
is injective.

\begin{thm}\label{54eight}
Let $C$ be a figure eight on the surface $S$. 
Then $\exp(\sigma(zL(C))) \in \ASE$ is not in the image of
$\widehat{{\sf DN}}$ for any $z \in K\setminus\{0\}$. In particular
the generalized Dehn twist $t_C$ is not realizable as a diffeomorphism.
\end{thm}

\begin{proof}
Take a regular neighborhood $N$ of $C$ in $S^*\setminus\partial S$, which
satisfies the assumptions of Proposition \ref{32subs}. Assume
$\exp(\sigma(zL(C)))\in \ASE$ is realized by a diffeomorphism $\varphi$. Then, by Theorem
\ref{localize}, we may take $\varphi$ as a diffeomorphism whose
support is included in $N$. In fact, Theorem \ref{localize} only
treat the case $z=1$, but the proof works as well as for general $z$. Then
$U := \exp(-\sigma(zL(C))) \varphi \in A(N, \partial N)$ satisfies the
condition of Proposition \ref{32subs}. Hence, by (\ref{51subs}),  we have 
$$
\varphi = \exp(\sigma(zL(C)+F^U)) \in A(N, \partial N).
$$
Here we remark $[L(C), F^U] = 0$ since $C\cap \partial N = \emptyset$. 
The surface $N$ is diffeomorphic to a pair of pants. We take $\eta_i$,
$\gamma_i$, $1 \leq i \leq 3$, as in Figure 6.
\begin{figure}
\begin{center}
\input{figure-eight.tex}

\caption{$\eta_i$ and $\gamma_i$}
\end{center}
\end{figure}
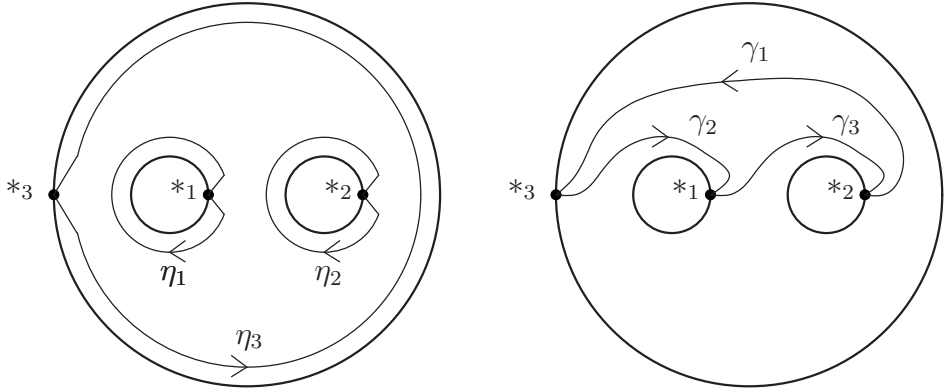

The mapping class group of the pair of pants is free abelian
of rank 3, generated by boundary-parallel Dehn twists (see, e.g.,
\cite{FM} \S 3.6). By Theorem \ref{42logDT}, we have 
$
\varphi = \exp(\sigma(\sum^3_{i=1}a_iL(\eta_i))) \in A(N,\partial N)
$
for some $a_i \in \mathbb{Z}$. By Proposition \ref{13exp} (3) we obtain
$$
\sigma((zL(C) + \textstyle\sum^3_{i=1}(a^U_i-a_i)L(\eta_i))
= 0 \in \Der\widehat{K\CC(N,\partial N)}.
$$
On the other hand, we have $C =
\vert({\gamma_3}^{-1}\eta_1\gamma_3{\eta_2}^{-1})^{\pm1}\vert$ and 
$$
\sigma(({\gamma_3}^{-1}\eta_1\gamma_3{\eta_2}^{-1})^{m})(\gamma_1) 
= -m({\gamma_3}^{-1}\eta_1\gamma_3{\eta_2}^{-1})^{m}\gamma_1
$$
for any $m \geq 0$. Hence 
$$
\sigma(L(C))(\gamma_1) =
\log(\eta_2{\gamma_3}^{-1}{\eta_1}^{-1}\gamma_3)\gamma_1.
$$
Further there exist some $b$ and $c \in K$ such that 
\begin{eqnarray*}
&&\sigma(\textstyle\sum^3_{i=1}(a^U_i-a_i)L(\eta_i))(\gamma_1) 
= b(\log\eta_2)\gamma_1 + c\gamma_1(\log{\eta_3}^{-1})\\
&=& (b\log\eta_2 +
c\log(\eta_2{\gamma_3}^{-1}\eta_1\gamma_3))\gamma_1. 
\end{eqnarray*}
Here note that $\gamma_1{\eta_3}^{-n}=(\eta_2{\gamma_3}^{-1}\eta_1\gamma_3)^n\gamma_1$
for $n\ge 0$. Hence we obtain 
\begin{equation}
z \log(\eta_2{\gamma_3}^{-1}{\eta_1}^{-1}\gamma_3) + 
b\log\eta_2 +
c\log(\eta_2{\gamma_3}^{-1}\eta_1\gamma_3)
=0 \in \widehat{K\pi_1(N, *_2)}.
\label{54rel1}
\end{equation}
The fundamental group $\pi_1(N, *_2)$ is a free group of rank $2$, so
that there exists an isomorphism of $K$-algebras
$\theta\colon \widehat{K\pi_1(N, *_2)}\overset\cong\to K\langle\langle X,
Y\rangle\rangle$ which satisfies $\theta(\log\eta_2) = X$ and 
$\theta(\log{\gamma_3}^{-1}\eta_1\gamma_3) = Y$. Here $K\langle\langle X,
Y\rangle\rangle = \prod^\infty_{m=0}(KX\oplus KY)^{\otimes m}$ is the
ring of non-commutative formal power series in indeterminates $X$ and
$Y$. In other words, $\theta$ is a group-like expansion of the free group 
$\pi_1(N, *_2)$. Then the equation (\ref{54rel1}) is equivalent to 
\begin{equation}
z X*(-Y) + bX + cX*Y = 0 \in K\langle\langle X,
Y\rangle\rangle.
\label{54rel2}
\end{equation}
Here $X*Y$ is the Hausdorff series in $X$ and $Y$. 
The degree $1$ part of (\ref{54rel2}) is 
$z(X - Y) + bX + c(X+Y) = 0$, so that we have $c = z$ and $b = -2z$. 
The degree $2$ part is 
$
-\frac{z}{2} [X,Y] + \frac12c[X,Y] = 0.
$
Thus the degree $3$ part is 
$$
\frac{z}{6} [Y,[Y,X]] = 0.
$$
This contradicts $\frac{z}{6} [Y,[Y,X]] \neq 0$, and
proves the theorem. 
\end{proof}

\section{Kontsevich's Lie algebras and Johnson homomorphisms}
\label{sec:6}

In this section we suppose $K$ is a field of characteristic $0$. 
Let $S$ be an oriented surface, and $E$ a non-empty closed
subset of $S$ with the property
$E\setminus\partial S$ is closed in $S$. In \S \ref{sec:4-1} we introduced the Lie
algebra homomorphism
$$
\sigma\colon \widehat{K\hat{\pi}}(S^*) \to \Der(\widehat{K\CSE}),
$$
where $S^* = S\setminus (E\setminus\partial S)$, while 
$\Der_\Delta\widehat{K\GG}$ is the Lie subalgebra of $\Der\widehat{K\GG}$ 
consisting of all the continuous derivations $D$ stabilizing the coproduct $\Delta$
for any groupoid $\GG$. See \S \ref{21filtration}.
We define Lie subalgebras $L(S,E)$ and $L^+(S,E)$
of $\widehat{K\hat{\pi}}(S^*)$ by 
\begin{eqnarray*}
L(S,E) &:=& \widehat{K\hat{\pi}}(S^*)(2)\cap
\sigma^{-1}(\Der_\Delta\widehat{K\CSE})
\subset \widehat{K\hat{\pi}}(S^*), \quad\mbox{and}\\
L^+(S,E) &:=& \widehat{K\hat{\pi}}(S^*)(3)\cap L(S,E). 
\end{eqnarray*}
$L^+(S,E)$ is an ideal of $L(S,E)$. 

\subsection{Geometric interpretation of Kontsevich's `associative' and `Lie'}
\label{aglg}

First of all, we study the case $S = \Sigma_{g,1}$ and $E =\{*\} \subset
\partial \Sigma_{g,1}$. Then we prove 
$\widehat{K{\hat\pi}}(\Sigma_{g,1})(2)$ is isomorphic to 
a completion of Kontsevich's `associative' $a_g$, and 
$L(S,E) = L(\Sigma_{g,1}, \{*\})$ a completion of Kontsevich's `Lie' $l_g$
\cite{Kon}. 
These results are essentially due to our previous work \cite{KK1}. 
Anyway this means Kontsevich's `associative' and `Lie' 
are contructed in a geometric context, and
$\widehat{K{\hat\pi}}(S^*)(2)$ and $L(S,E)$ for a general $(S,E)$
can be regarded as generalizations of Kontsevich's `associative' and
`Lie', respectively.\par
To state our previous results \cite{KK1}, we need some notations. 
Adopting the notations in \S \ref{sec:3-3}
we denote $\pi:=\pi_1(\Sigma_{g,1}, \{*\})$, 
$H := H_1(\Sigma_{g,1}; K)$ and $\widehat{T} := \prod^\infty_{m=0}
H^{\otimes m}$. Let $\zeta \in \pi$ be a boundary loop in the opposite
direction, and $\{A_i, B_i\}^g_{i=1}\subset H$ a symplectic basis. 
The symplectic form $\omega:= \sum^g_{i=1}A_iB_i -B_iA_i \in H^{\otimes
2}$ is independent of the choice of a symplectic basis.

\begin{dfn}[Massuyeau \cite{Mas}]
\label{def:symp-exp}
A {\it symplectic expansion} $\theta\colon \pi \to \widehat{T}$ is a
group-like expansion which satisfies the equation $\theta(\zeta) =
\exp\omega$.
\end{dfn}

As was stated in \S \ref{sec:3-3}, the algebra $\widehat{T}$ has a
filtration defined by  the ideals $\widehat{T}_p :=
\prod^\infty_{m=p}H^{\otimes m}$, $p \geq 1$. By the Poincar\'e duality,
we identify $H$ with $H^*= \Hom(H, K)$ via the isomorphism
$H \stackrel{\cong}{\rightarrow} H^*, X\mapsto (Y\mapsto (Y\cdot X))$.
Then $\widehat{T}_1$ is identified 
with $H\otimes\widehat{T} = H^*\otimes \widehat{T} = \Der(\widehat{T})$, 
the (continuous) derivation Lie algebra of the (filtered) $K$-algebra $\widehat{T}$. 
Recall the linear map $N\colon \widehat{T} \to \widehat{T}$ introduced in \S \ref{sec:4-3}. 
Then the image $N(\widehat{T}_1) = N(\widehat{T})$ equals the Lie
subalgebra of $\widehat{T}_1 = \Der(\widehat{T})$ consisting
of symplectic derivations, namely, derivations annihilating the
symplectic form $\omega$ (see \cite{KK1} \S 2.7), which we denote $\mathfrak{a}_g^- =
\Der_\omega(\widehat{T})$. The Lie subalgebra
$\mathfrak{a}_g := N(\widehat{T}_2)$ is a completion of Kontsevich's
`associative' $a_g$\cite{Kon}. 
Our previous results are
\begin{thm}[\cite{KK1} Theorem 1.2.1]\label{121}
Let $\theta\colon \pi \to \widehat{T}$
be a symplectic expansion. Then the map 
$$
-\lambda_\theta\colon K{\hat\pi}(\Sigma_{g,1}) \to \mathfrak{a}_g^-, \quad
\lambda_\theta(\vert x\vert) := N\theta(x), \quad x \in \pi,
$$
is a Lie algebra homomorphism. The kernel is the subspace $K1$ 
spanned by the constant loop $1$, and the image is dense in
$N(\widehat{T}_1) =\mathfrak{a}_g^-$ with respect to the $\widehat{T}_1$-adic topology. 
\end{thm}

\begin{thm}[\cite{KK1} Theorem 1.2.2]\label{122}
Let $\theta$ be a symplectic expansion. Then, 
for $u \in K\hat{\pi}(\Sigma_{g,1})$ and $v \in K\pi$, 
we have the equality
$$
\theta(\sigma(u)v) = -\lambda_{\theta}(u)\theta(v).
$$
Here the right hand side means minus the action of
$\lambda_{\theta}(u) \in \mathfrak{a}_g^-$ 
on the tensor $\theta(v) \in \widehat{T}$
as a derivation.
In other words, the diagram
$$
\begin{CD}
K\hat{\pi}(\Sigma_{g,1}) \times K\pi @>{\sigma}>>
K\pi \\
@V{-\lambda_{\theta}\times \theta}VV @VV{\theta}V \\
\mathfrak{a}_g^{-} \times \widehat{T} @>>> \widehat{T},
\end{CD}
$$
where the bottom horizontal arrow means the derivation, commutes.
\end{thm}

Note that for any $p \geq 2$, 
$N(\widehat{T}_p)$ is
a Lie subalgebra of $\mathfrak{a}_g^-$ and an ideal of $\mathfrak{a}_g$. 
Immediately from Corollary \ref{43Ntheta} and Theorem
\ref{121} we obtain

\begin{thm}\label{61ag} 
Let $\theta$ be a symplectic expansion. The map
$\lambda_\theta\colon K{\hat\pi}(\Sigma_{g,1}) \to \mathfrak{a}_g^-$
induces isomorphisms of Lie algebras
$$
-\lambda_\theta\colon \widehat{K{\hat\pi}}(\Sigma_{g,1}) \overset\cong\to
\mathfrak{a}_g^-\quad\mbox{and}\quad
-\lambda_\theta\colon \widehat{K{\hat\pi}}(\Sigma_{g,1})(2) \overset\cong\to
\mathfrak{a}_g.
$$
\end{thm}

Under the identification $\widehat{T}_1 = H\otimes\widehat{T} =
\Der(\widehat{T})$, the Lie subalgebra consisting of all (continuous) derivations 
stabilizing the coproduct $\Delta$ on $\widehat{T}$ coincides with 
$H\otimes\widehat{\mathcal{L}}$. Here 
$\widehat{\mathcal{L}}$ is the space of all Lie-like elements in the
completed tensor algebra $\widehat{T}$. 
The Lie algebra $H\otimes\widehat{\mathcal{L}}$ can be regarded as 
the (continuous) derivation Lie algebra 
of the Lie algebra $\widehat{\mathcal{L}}$. 
We define the Lie algebra  
$\mathfrak{l}_g$ by the intersection 
$$
\mathfrak{l}_g := \mathfrak{a}_g^- \cap (H\otimes\widehat{\mathcal{L}}),
$$
which is a completion of Kontsevich's `Lie' $l_g$ \cite{Kon}. 
Here it should be remarked the Lie algebra $l_g$ was introduced
earlier by Morita \cite{MoPJA} \cite{MoICM} as a target of the Johnson
homomorphisms of the (higher) Torelli groups. Anyway it is clear from the definition that $u \in \mathfrak{a}_g^-$ is in $\mathfrak{l}_g$ if and only if 
$(u\widehat{\otimes}u) \circ \Delta = \Delta \circ u$, where $\Delta$ is the coproduct of the complete Hopf algebra $\widehat{T}$. 
Moreover we define
$\mathfrak{l}_g^+$ by
$$
\mathfrak{l}_g^+ := \mathfrak{l}_g\cap \widehat{T}_3 \quad
(\subset \mathfrak{l}_g \subset \widehat{T}_1 = \Der\widehat{T}), 
$$
which is an ideal of $\mathfrak{l}_g$. 
\par

Taking the projective limits of the diagram in Theorem \ref{122}, 
we obtain the commutative diagram
\begin{equation}
\begin{CD}
\widehat{K\hat{\pi}}(\Sigma_{g,1}) \times \widehat{K\pi} @>{\sigma}>>
\widehat{K\pi} \\
@V{\cong}VV @V{\cong}VV \\
\mathfrak{a}_g^{-} \times \widehat{T} @>>> \widehat{T}.
\end{CD}
\label{61com}
\end{equation}
The isomorphism $\theta\colon \widehat{K\pi} \overset\cong\to \widehat{T}$
preserves the coproducts. From the definition
of $\mathfrak{l}_g$ and $\mathfrak{l}_g^+$, and the fact
$-\lambda_\theta\colon \widehat{K\hat{\pi}}(\Sigma_{g,1}) \overset\cong\to \mathfrak{a}_g^-$
preserves the filtration, we obtain the following theorem. 
\begin{thm}\label{61lg} Let $\theta$ be a symplectic expansion. The map
$\lambda_\theta\colon K{\hat\pi}(\Sigma_{g,1}) \to \mathfrak{a}_g^-$ induces 
isomorphisms of Lie algebras
$$
-\lambda_\theta\colon L(\Sigma_{g,1}, \{*\}) \overset\cong\to \mathfrak{l}_g, 
\quad\mbox{and}\quad 
-\lambda_\theta\colon L^+(\Sigma_{g,1}, \{*\}) \overset\cong\to \mathfrak{l}_g^+.
$$
\end{thm}

\subsection{Infinitesimal Dehn-Nielsen theorem}
\label{N62}

In \S\ref{sec:3-1} we discussed the injectivity of the Dehn-Nielsen
homomorphism ${\sf DN}\colon \mathcal{M}(S,E) \to \Aut(K\CSE)$. 
In view of the results in \S \ref{sec:5}, we may regard the Lie algebra homomorphism
$\sigma\colon \widehat{K\hat\pi}(S) \to \Der(\widehat{K\CSE})$ as an
infinitesimal analogue of the Dehn-Nielsen homomorphism. This subsection
is devoted to the proof of the following, an infinitesimal Dehn-Nielsen
theorem.

\begin{thm}\label{N62DN}
Let $S$ be a compact connected oriented surface with non-empty boundary, 
$E \subset\partial S$, and suppose any connected component of $\partial
S$ has an element of $E$. Then the homomorphism $\sigma\colon
\widehat{K\hat\pi}(S) \to \Der(\widehat{K\CSE})$ is injective. 
\end{thm}

The completion maps $K\hat\pi(S)/K1 \to \widehat{K\hat\pi}(S)$ and 
$K\CSE \to \widehat{K\CSE}$ are injective from Corollary \ref{43inj} and
\cite{Bou}. Hence, as a corollary, we have 
\begin{cor}\label{N62DNa}
Under the assumption of {\rm Theorem \ref{N62DN}}, the homomorphism $\sigma\colon
K\hat\pi(S)/K1 \to \Der(K\CSE)$ is injective. 
\end{cor}
 
If $S=\Sigma_{g,1}$, Theorem \ref{N62DN} follows immediately from
Theorems \ref{122} and \ref{61ag}. Otherwise, we have $S =
\Sigma_{g,n+1}$ for $n \geq 2$. We number the boundary components $\partial
S = \coprod^n_{j=0}\partial_jS$. Capping each $\partial_jS$, $1 \leq j \leq n$, 
by a compact surface diffeomorphic to $\Sigma_{1,1}$, we obtain $S' =
\Sigma_{g+n,1}$ as in Figure 7. We denote by $\imath: S \hookrightarrow S'$ the inclusion map, and $S'' := \overline{S'\setminus S}$. 

\begin{figure}
\begin{center}

\input{capping.tex}

\caption{capping}

\end{center}
\end{figure}
We denote $\pi:= \pi_1(S,*)$ and $\pi' := \pi_1(S,*)$ for any fixed $p \in S$. 
The groups
$\pi$ and $\pi'$ are finitely generated free groups. We can choose free
generator systems $\{\alpha_i, \beta_i\}^g_{i=1}\cup \{\gamma_j\}^{n
}_{j=1} \subset \pi$ and $\{\alpha'_i, \beta'_i\}^{g+n}_{i=1} \subset
\pi'$, such that each $\gamma_j$ is conjugate to the boundary loop of $\partial_jS$, and the inclusion homomorphism maps 
$$
\begin{cases}
\alpha_i \mapsto \alpha'_i,&(1 \leq i \leq g),\\
\beta_i \mapsto \beta'_i, &(1 \leq i \leq g),\\
\gamma_j \mapsto
\alpha'_{g+j}\beta'_{g+j}{\alpha'_{g+j}}^{-1}{\beta'_{g+j}}^{-1},
&(1 \leq j \leq n).
\end{cases}
$$
Then we have the following, which we need for the proof of Theorem \ref{N62DN}. 
\begin{lem}\label{N62inj}
\begin{enumerate}
\item The inclusion homomorphism
$
\widehat{K\pi_1(S,p)} \to \widehat{K\pi_1(S',p)}
$
is injective for any $p \in S$. 
\item The kernel of the inclusion homomorphism
$
\widehat{K\hat\pi}(S) \to \widehat{K\hat\pi}(S')
$
is $n$-dimensional, and spanned by $\{\vert\log\gamma_j\vert\}_{j=1}^n$.
\end{enumerate}
\end{lem}

\begin{proof}
(1) 
We write simply $H:= H_1(S; K)$, $H' := H_1(S';K)$, $A_i := [\alpha_i],
B_i := [\beta_i], C_j := [\gamma_j] \in H$, and $A'_i := [\alpha'_i],
B'_i := [\beta'_i] \in H'$. Then the sets $\{A_i, B_i\}^g_{i=1}\cup
\{C_j\}^{n}_{j=1} \subset H$ and $\{A'_i, B'_i\}^{g+n}_{i=1} \subset
H'$ are $K$-free bases of $H$ and $H'$, respectively. Let $\theta\colon \pi
\to \widehat{T}$ and $\theta'\colon \pi' \to \widehat{T}'$ be group-like
expansions, where $\widehat{T} = \prod^\infty_{m=0}H^{\otimes m}$ and
$\widehat{T}' = \prod^\infty_{m=0}(H')^{\otimes m}$. By the isomorphisms
(\ref{expansion}) $\theta\colon \widehat{K\pi} \cong \widehat{T}$ and
$\theta'\colon \widehat{K\pi'} \cong \widehat{T}'$, the inclusion homomorphism
induces an algebra homomorphism $\imath\colon \widehat{T} \to \widehat{T}'$
such that 
\begin{equation}
\begin{cases}
\imath(A_i) = A'_i + \mbox{higher terms}, &(1 \leq i \leq
g),\\
\imath(B_i) = B'_i + \mbox{higher terms}, &(1 \leq i \leq
g),\\ 
\imath(C_j) = [A'_{g+j}, B'_{g+j}] + \mbox{higher terms}, &(1 \leq j
\leq n).
\end{cases}
\label{N62iota}
\end{equation}
Hence it suffices to show any (continuous) algebra homomorphism $\imath\colon
\widehat{T} \to \widehat{T}'$ satisfying the condition (\ref{N62iota}) is
injective. \par
To prove this, we introduce some filtrations on the algebras
$\widehat{T}$ and $\widehat{T}'$. We have $H = H_{(1)}\oplus H_{(2)}$
where $H_{(1)}$ and $H_{(2)}$ are the linear spans of the sets $\{A_i,
B_i\}^g_{i=1}$ and $\{C_j\}^{n}_{j=1}$, respectively. We consider the
decreasing filtration on $H$ defined by $F_1H := H$ and $F_2H :=
H_{(2)}$. 
It induces a decreasing filtraion on the algebra $\widehat{T}$ such that 
$F_0\widehat{T}/F_1\widehat{T} = K$ and 
$$
F_n\widehat{T}/F_{n+1}\widehat{T} =
\bigoplus^n_{q=1}\bigoplus_{\delta_1+\cdots \delta_q=n}
H_{(\delta_1)}\otimes \cdots\otimes H_{(\delta_q)}
$$
for any $n\geq 1$. On $\widehat{T}'$ we introduce the filtration we
consider usually, $\widehat{T}'_n := \prod^\infty_{m=n}(H')^{\otimes m}$,
$n \geq 0$. We have 
$$
\widehat{T}'_n/\widehat{T}'_{n+1} =
\bigoplus_{\delta_1,\dots,\delta_n}
H'_{(\delta_1)}\otimes \cdots\otimes H'_{(\delta_n)},
$$
where $H'_{(1)}$ and $H'_{(2)}$ are the linear spans of the sets 
$\{A'_i, B'_i\}^g_{i=1}$ and $\{A'_{g+j}, B'_ {g+j}\}^{n}_{j=1}$, 
respectively. \par
The condition (\ref{N62iota}) implies $\imath(F_n\widehat{T})
\subset \widehat{T}'_n$ for any $n \geq 0$. The map $\imath$ induces 
an isomorphism $H_{(1)} \overset\cong\to H'_{(1)}$ and an injective map 
$H_{(2)} \to H'_{(2)}\otimes H'_{(2)}$ whose image is a direct summand of
the target. Hence the induced map 
$$
\imath_n\colon F_n\widehat{T}/F_{n+1}\widehat{T} \to
\widehat{T}'_n/\widehat{T}'_{n+1}
$$
is injective for any $n \geq 0$.\par
Now assume there exists a non-zero element $u$ of the map $\imath\colon
\widehat{T} \to \widehat{T}'$. Since $\bigcap^\infty_{n=1}F_n\widehat{T}
= 0$, we have $u \in F_n\widehat{T}\setminus F_{n+1}\widehat{T}$ for some
$n \geq 0$. On the other hand, we have $\imath_n(u\bmod
F_{n+1}\widehat{T}) = 0 \in \widehat{T}'_n/\widehat{T}'_{n+1}$. This
implies $u \in F_{n+1}\widehat{T}$, which contradicts $u \in
F_n\widehat{T}\setminus F_{n+1}\widehat{T}$. Hence the kernel of the map
$\imath\colon \widehat{T} \to \widehat{T}'$ is zeto. This proves the part
(1).\qed\par
(2) In view of Theorem \ref{43isom}, the inclusion homomorphism
$
\widehat{K\hat\pi}(S) \to \widehat{K\hat\pi}(S'),
$
is equivalent to the composite of  the
inclusion homomorphisms
\begin{eqnarray*}
&& H_1(S; \widehat{\mathcal{S}}^c(S)) \to H_1(S;
\widehat{\mathcal{S}}^c(S')), \quad \mbox{and}\\
&& H_1(S; \widehat{\mathcal{S}}^c(S')) \to 
H_1(S'; \widehat{\mathcal{S}}^c(S')).
\end{eqnarray*}
The former is injective from (1) and 
$H_2(S; \widehat{\mathcal{S}}^c(S')/\widehat{\mathcal{S}}^c(S)) = 0$. \par
In order to compute the dimension of the kernel of the latter, it suffices to show 
\begin{equation}
H_2(S',S; \widehat{\mathcal{S}}^c(S')) \cong K^{\oplus n}. 
\label{N62H2}
\end{equation}
By the excision and the Poincar\'e duality theorem, the second homology
group is isomorphic to $H^0(S''; \widehat{\mathcal{S}}^c(S'))$. Each
connected component ${S''}_j$, $1 \leq j \leq n$, of $S''$ is of genus
$1$, and so has two distinct non-separating simple closed curves. Hence,
by Proposition \ref{32simple}, we have $H^0({S''}_j;
\widehat{\mathcal{S}}^c(S')) = K[[A'_{g+j}]] \cap K[[B'_{g+j}]] = K$, 
$1 \leq j \leq n$. This means (\ref{N62H2}).\par
Choose a group-like expansion $\theta': \pi_1(S', *) \to \widehat{T}'$. 
Then $N\theta'(\imath\log\gamma_j) = 0$ for any $1 \leq j \leq n$, since $\theta'(\imath\log\gamma_j) = \theta'(\log\imath\gamma_j)$ is 
a Lie-like element of $\widehat{T}'$, and $\imath\gamma_j = \alpha'_{g+j}\beta'_{g+j}{\alpha'_{g+j}}^{-1}{\beta'_{g+j}}^{-1}$ is null-homologous in the surface $S'$. From Corollary \ref{43Ntheta} this means that the set $\{\vert\log\gamma_j\vert\}_{j=1}^n$ is included in the kernel of the inclusion homomorphism. \par
Now choose points  $*_j \in \partial_jS$ for $0 \leq j \leq n$, and simple paths $\delta_j$ on $S$ from $*_0$ to $*_j$ for $1 \leq j \leq n$. By a similar computation to that in the proof of Theorem \ref{42logDT}, we have 
\begin{equation}
\sigma(\vert\log\gamma_j\vert)(\delta_k) =
\begin{cases} 
\pm \delta_j, &\mbox{if $j=k$}\\
0, &\mbox{otherwise.}
\end{cases}
\label{N62delta}
\end{equation}
In particular, the set $\{\vert\log\gamma_j\vert\}_{j=1}^n$ is $K$-linear independent. This proves the part (2).
\end{proof}

\begin{proof}[Proof of Theorem \ref{N62DN}] Recall $S'' =
\overline{S'\setminus S}$. We denote $E'' := E\setminus\partial_rS\subset S''$. 
By the assumption on $E$, each connected component of $S\cap S''$ has some
point in $E''$. Hence, from Proposition \ref{32vKT}, $K\CC(S',E)$ is
generated by $K\CSE$ and $K\CC(S'', E'')$. \par
Let $u \in \widehat{K\hat\pi}(S)$ satisfy $\sigma(u) = 0 \in 
\Der(\widehat{K\CSE})$. Clearly $\sigma(u) = 0 \in
\Der(\widehat{K\CC(S'', E'')})$. Hence we have $\sigma(u) = 0 \in
\Der(\widehat{K\CC(S',E)})$. By the injectivity of $\sigma$ for $S' =
\Sigma_{g+n,1}$, we have $u=0 \in \widehat{K\hat\pi}(S')$. Using Lemma
\ref{N62inj}(2), we find $u \in \widehat{K\hat\pi}(S)$ is a linear combination of $\{\vert\log\gamma_j\vert\}_{j=1}^n$. From (\ref{N62delta}) and the assumption $\sigma(u) = 0$, we have $u = 0 \in \widehat{K\hat\pi}(S)$. 
This proves completes the proof of Theorem \ref{N62DN}. 
\end{proof}

\subsection{The geometric Johnson homomorphism}
\label{tau}

Let $S$ be an oriented surface, and $E$ a non-empty closed
subset of $S$ with the property
$E\setminus\partial S$ is closed in $S$. 
The group $\mathcal{M}(S,E)$ acts on the $K$-SAC 
$K\CSE/F_{2}K\CSE$ in an obvious way.
We define {\it the Torelli group} $\ISE$ of the pair $(S,E)$ to be the
kernel of this action 
$$
\ISE := \Ker(\mathcal{M}(S,E)\to \Aut K\CC/F_{2}K\CC),
$$
which is independent of the choice of $K$, a field of characteristic $0$. 
When $S = \Sigma_{g,1}$ and $E \subset \partial\Sigma_{g,1}$,
$\ISE = \mathcal{I}_{g,1}$, the Torelli group of genus $g$ with $1$
boundary component. But, in general, as will be shown later, 
$\ISE$ is the ``smallest'' Torelli group in the sense of Putman \cite{P}.
It was studied also in \cite{JII} and \cite{vB}. 
In order to obtain Putman's Torelli group \cite{P} of other kinds, it
seems to be needed to change the filtration on $K\CSE$ to that induced
from a capping by surfaces of positive genus. 
\par
From Lemma \ref{13conv}, the
exponential $\exp(D)$ converges as an element in $\ASE$ for any 
derivation $D \in  \sigma(L^+(S,E))$.
Clearly the image $\sigma(L^+(S,E))$ is a Lie subalgebra of 
$\Der\widehat{K\CSE}$. 
\begin{lem}\label{62exp} 
The exponential 
$$
\exp\colon \sigma(L^+(S,E)) \to \ASE
$$
is injective, and its image is a subgroup of $\ASE$. 
\end{lem}

\begin{proof} The injectivity follows from Proposition \ref{13exp} (3). 
As was shown in Proposition \ref{13exp}, $(\exp D)^{-1} = \exp(-D) \in
\ASE$ for $D \in \sigma(L^+(S,E))$. From the definition,
$\sigma(L^+(S,E)) \subset F_1\Der\widehat{K\CC}$. Moreover
$\widehat{K{\hat\pi}}(S^*)$ is complete 
with respect to the filtration
$\{\widehat{K{\hat\pi}}(S^*)(n)\}_{n\geq1}$. 
Hence the Hausdorff series $u*u'$ of $u$ and $u'
\in L^+(S,E)$ converges as an element of $L^+(S,E)$. 
This implies
$\exp\sigma(L^+(S,E))$ is a subgroup of $\ASE$. 
\end{proof}
Thus the exponential $\exp$ induces a group structure on the set
$\sigma(L^+(S,E))$, which is a pro-nilpotent group. \par

For the rest of this subsection, 
we suppose $S$ is compact with non-empty bundary, 
$E \subset \partial S$, and each component of $\partial S$ has an 
element of $E$. From Theorems \ref{31DN} and \ref{N62DN}, 
both of the Dehn-Nielsen homomorphism $\widehat{\sf DN}\colon 
\mathcal{M}(S,E)\to \ASE$ and the homomorphism $\sigma\colon L^+(S,E) 
\to \Der(\widehat{K\CSE})$ are injective. In particular, the Lie algebra 
$L^+(S,E)$ has the structure of a pro-nilpotent group. \par
Assume the inclusion
\begin{equation}
\widehat{\sf DN}(\ISE) \subset \exp\sigma(L^+(S,E))
\label{62incl}
\end{equation}
holds. 
Then we have a unique injective map
$$
\tau\colon \ISE \to L^+(S,E)
$$
such that $\widehat{\sf DN}\vert_\ISE = \exp\circ\sigma\circ\tau\colon
\ISE\to\ASE$, 
which is a group homomorphism with respect to the pro-nilpotent group
structure on $L^+(S,E)$. As we will see later, the homomorphism $\tau$ is a 
generalization of the higher Johnson homomorphism of $\mathcal{I}_{g,1}$ 
introduced by Johnson \cite{JS} and improved by Morita \cite{MoAQ}. 
So we call it {\it the geometric Johnson homomorphism}
of the Torelli group $\mathcal{I}(S, E)$.
At present we have no suitable presentation of the Lie algebra
$L^+(S,E)$. In order to obtain it, we need to generalize 
Magnus expansions \cite{Ka} \cite{Mas} of free groups to those for
groupoids. \par

The Zariski closure of the image of $\tau$ does not equal the whole 
$L^+(S,E)$. This fact was discovered by Morita \cite{MoAQ}. For 
recent progress on this problem, see \cite{ES} and its reference. From
Theorem \ref{54eight}, 
the generalized Dehn twist along a null-homologous
figure eight is in the complement of the closure. 
In our forthcoming paper \cite{KK4} we will give a candidate for the
defining equation of the Zariski closure. 
\par

Now suppose $S = \Sigma_{g,r}$, $r\geq 1$, and each component of
$\partial S$ has a unique point of $E$. 
Then Putman \cite{P} gave an explicit
generator system of $\ISE$ for this $(S, E)$.  
Using this remarkable theorem together with our formula for Dehn twists
(Theorem \ref{42logDT}), we prove the inclusion (\ref{62incl}) in this
case. \par

To state Putman's theorem, we number the boundary components
as $\partial S = \coprod^r_{i=1}\partial_iS$, and define 
a partition $P$ of $\pi_0(\partial S)$ by $P =
\{\{1,\dots, r\}\}$. Our $\ISE$ is identified 
with Putman's $\mathcal{I}(S,P)$ for this $P$, as follows. 
For this partition $P$, Putman's $\partial H^P_1(S)$
vanishes,  because $\sum^r_{i=1}[\partial_iS] = 0 \in H_1(S;
\mathbb{Z})$. Hence we have $H^P_1(S) = H_1(S, E; \mathbb{Z})$. By
\cite{P} Theorem 3.3, the group $\mathcal{I}(S,P)$ is exactly the
subgroup of $\mathcal{M}(S,E)$ which acts trivially on $H_1(S,E;
\mathbb{Z})$. On the other hand, if
$*_0\neq *_1 \in E$, we have $\CSE^\abel(*_0,*_1) = \mathbb{Z} \oplus
H_1(S,\{*_0,*_1\}; \mathbb{Z})$ as $\mathcal{M}(S,E)$-modules.
Thus a
mapping class $\varphi \in \mathcal{M}(S,E)$ acts trivially on
$K\CC/F_2K\CC$ if and only if it acts trivially on $H_1(S, E;
\mathbb{Z})$. In other words, our $\ISE$ equals Putman's
$\mathcal{I}(S,P)$ with $P = \{\{1,\dots,r\}\}$. Putman proved the following theorems.
\begin{thm}[Putman \cite{P} Theorem 1.3]
\label{putman}
For any partitioned surface $(\Sigma_{g,r}, P)$ with $g \geq 1$, the
group $\mathcal{I}(\Sigma_{g,r}, P)$ is generated by twists about
$P$-separating curves and twists about $P$-bounding pairs. 
\end{thm}

\begin{thm}[Putman \cite{P} Theorem 1.5]
\label{putman0}
For any genus $0$ partitioned surface $(\Sigma_{0,r}, P)$,
the group $\mathcal{I}(\Sigma_{0,r}, P)$ is generated
by twists about $P$-separating curves, twists about $P$-bounding pairs,
and commutators of simply intersecting pairs.
\end{thm}

For our $P$, a $P$-separating curve is a simple closed curve $C$ with
$[C] = 0 \in H_1(S; \mathbb{Z})$, and a $P$-bounding pair is a pair of
disjoint, non-isotopic simple closed curves $C_1$ and $C_2$ with
$\pm[C_1] = \pm[C_2] \in H_1(S; \mathbb{Z})$. 
The main theorem in this subsection is 

\begin{thm}
The inclusion {\rm (\ref{62incl})} holds for any $g \geq 0$ and $r \geq 1$.
\end{thm}

\begin{proof} In view of Putman's theorems, it suffices to show the
$\widehat{\sf DN}$-images of twists about $P$-separating curves, 
twists about $P$-bounding pairs 
and commutators of simply intersecting pairs 
are in $\exp(\sigma(L^+(S,E)))$. \par
Fix an element $* \in E$. 
Let $C$ be a $P$-separating curve represented by some $x \in \pi_1(S,
*)$. Since $\pm[x] = \pm[C] = 0 \in H_1(S; \mathbb{Z})$, $x-1 \in
I\pi_1(S,*)^2$. Hence we have $L(x) \in \widehat{I\pi_1(S,*)^4}$, and
so $L(C) \in \widehat{K{\hat\pi}}(S)(4)$. It follows from Theorem
\ref{42logDT} that $\widehat{\sf DN}(t_C) = \exp(\sigma(L(C))) \in
\exp(\sigma(L^+(S,E)))$. \par

Let $C_1$ and $C_2$ form a $P$-bounding pair. Choose $x_1$ and $x_2 \in
\pi_1(S,*)$ suth that $C_1 = \vert {x_1}^{\pm1}\vert$, $C_2 = \vert
{x_2}^{\pm1}\vert$ and $[x_1] = [x_2] \in H_1(S; \mathbb{Z})$. We have 
$x_2 = x_1z$ for some $z \in [\pi_1(S,*), \pi_1(S,*)]$. Since $z-1
\in I\pi_1(S,*)^2$, we have $(x_2-1)^2 - (x_1-1)^2 = x_1(z-1)x_1(z-1) +
x_1(z-1)(x_1-1) + (x_1-1)x_1(z-1) \in I\pi_1(S,*)^3$, while $L(x_1) -
\frac12(x_1-1)^2$ and $L(x_2) -\frac12(x_2-1)^2$ are in
$\widehat{I\pi_1(S,*)^3}$. Hence $L(x_2) - L(x_1) \in
\widehat{I\pi_1(S,*)^3}$, 
and so $L(C_2) - L(C_1) \in \widehat{K{\hat\pi}}(S)(3)$.
Since $C_1\cap C_2= \emptyset$, we have $[L(C_1), L(C_2)] = 0$. 
It follows from Theorem
\ref{42logDT} $\widehat{\sf DN}({t_{C_1}}^{-1}t_{C_2}) =
\exp(\sigma(L(C_2)-L(C_1))) \in\exp(\sigma(L^+(S,E)))$. \par

Finally let $\{C_1, C_2\}$ be a pair of simple closed curves 
in $S$ whose geometric intersection number is $2$ and 
whose algebraic intersection number is $0$. 
Then

\begin{lem}\label{63pair} The Lie bracket of $L(C_1)$ and 
$L(C_2)$ is in $\widehat{K\hat\pi}(3)$,
$$
[L(C_1), L(C_2)] \in \widehat{K\hat\pi}(3).
$$
\end{lem}
\begin{proof}
Choose oriented simple free loops $\gamma$ and $\delta$ 
with minimal intersections such that
$C_1 = \gamma^{\pm1}$ and $C_2 = \delta^{\pm1}$. 
Since $L(C_1) - \frac12 (\gamma-1)^2$ and 
$L(C_2) - \frac12(\delta-1)^2 \in \widehat{K\hat\pi}(3)$, 
it suffices to show 
$$
[(\gamma-1)^2, (\delta-1)^2] \in K\hat\pi(3).
$$
Let $p$ and $q$ be the positive and the negative 
intersection points of $\gamma$ and $\delta$, respectively.
We remark $\gamma_q = {\gamma_{pq}}^{-1}\gamma_p\gamma_{pq}$ 
and $\delta_q = {\delta_{pq}}^{-1}\delta_p\delta_{pq}$, and denote 
$z_p := \gamma_{pq}{\delta_{pq}}^{-1} \in \pi_1(S, p)$. Then we have 
\begin{eqnarray*}
&& [(\gamma-1)^2, (\delta-1)^2] = [\gamma^2 - 2\gamma, \delta^2 - 2\delta]\\
&=& 4 \vert{\gamma_p}^2{\delta_p}^2 - {\gamma_p}{\delta_p}^2 - {\gamma_p}^2{\delta_p} + {\gamma_p}{\delta_p}\vert 
- 4 \vert{\gamma_q}^2{\delta_q}^2 - {\gamma_q}{\delta_q}^2 - {\gamma_q}^2{\delta_q} + {\gamma_q}{\delta_q}\vert \\
&=& 4 \vert{\gamma_p}({\gamma_p}-1){\delta_p}({\delta_p}-1)\vert 
- 4 \vert{\gamma_q}({\gamma_q}-1){\delta_q}({\delta_q}-1)\vert  \\
&=& 4 \vert{\gamma_p}({\gamma_p}-1){\delta_p}({\delta_p}-1)\vert 
- 4 \vert{\gamma_{pq}}^{-1}{\gamma_p}({\gamma_p}-1)\gamma_{pq}{\delta_{pq}}^{-1}{\delta_p}({\delta_p}-1)\delta_{pq}\vert  \\
&=& 4 \vert{\gamma_p}({\gamma_p}-1)\{{\delta_p}({\delta_p}-1) 
- z_p\delta_p(\delta_p-1){z_p}^{-1}\}\vert  \\
&=& 4 \vert{\gamma_p}({\gamma_p}-1)(1-z_p){\delta_p}({\delta_p}-1)\vert 
+ 4 \vert{\gamma_p}({\gamma_p}-1)z_p{\delta_p}({\delta_p}-1)(1-{z_p}^{-1}) 
\vert ,
\end{eqnarray*} 
which is clearly an element of $K\hat\pi(3)$. This proves the lemma. 
\end{proof}
Now we denote by $\ast$ the Hausdorff series. 
We will prove the power series 
$$
P:= L(C_1)\ast L(C_2)\ast(-L(C_1))\ast(-L(C_2))
$$
converges to an element of $\widehat{K\hat\pi}(3)$. 
From the Baker-Campbell-Hausdorff formula, each term of the series $P$ 
contains the bracket $[L(C_1), L(C_2)]$, which is an element of 
$\widehat{K\hat\pi}(3)$, as was shown in Lemma \ref{63pair}. 
This also implies 
\begin{equation}
[{\rm ad}(L(C_1)), {\rm ad}(L(C_2))](\widehat{K\hat\pi}(n)) 
={\rm ad}([L(C_1), L(C_2)])(\widehat{K\hat\pi}(n)) 
\subset \widehat{K\hat\pi}(n+1)
\label{63increase1}
\end{equation}
for any $n \geq 2$.
On the other hand, as was shown in Lemma \ref{51ft} (3) and the proof of 
Lemma \ref{13conv}, we have $\sigma(L(C_j))^{n+1}((I\pi)^n) \subset (I\pi)^{n+1}$ 
for any $j = 1,2$ and $n \geq 1$, so that 
\begin{equation}
{\rm ad}(L(C_j))^{n+1}(\widehat{K\hat\pi}(n)) \subset \widehat{K\hat\pi}(n+1).
\label{63increase2}
\end{equation}
Hence, from (\ref{63increase1}) and (\ref{63increase2}), we have 
$$
{\rm ad}(Z_1)\cdots{\rm ad}(Z_{2n+1})(\widehat{K\hat\pi}(n)) \subset \widehat{K\hat\pi}(n+1)
$$
for any $Z_i \in \{L(C_1), L(C_2)\}$, $1 \leq i \leq 2n+1$.
This means that any bracket of $2+ \sum^{n-1}_{k=3}(2k+1) = n^2-7$ 
copies of $L(C_1)$ and $L(C_2)$ is in $\widehat{K\hat\pi}(n)$
for any $n \geq 2$, so that the series $P$ converges to an element of 
$\widehat{K\hat\pi}(3)$. Since $L(C_1)$ and $L(C_2)$ stabilize 
the coproduct $\Delta$, the element $P$ is also in $L^+(S, E)$. 
Thus $\widehat{\sf DN}({t_{C_1}}{t_{C_2}}{t_{C_1}}^{-1}{t_{C_2}}^{-1}) = \exp(\sigma(P))$ 
is in $\exp(\sigma(L^+(S,E)))$.    
\par
This completes the proof of the theorem.
\end{proof}

Consequently we obtain the geometric Johnson homomorphism
$$
\tau\colon \mathcal{I}(\Sigma_{g,r}, E) \to L^+(\Sigma_{g,r}, E),
$$
which is an embedding of the
Torelli group $\mathcal{I}(\Sigma_{g,r}, E)$ into the pro-nilpotent group 
$L^+(\Sigma_{g,r}, E)$. 
This homomorphism is natural in the following sense. \par 

Now we consider the case $S = \Sigma_{g,1}$ and $E = \{*\}$. 
Recall the notation in the previous subsections. 
In our situation the Torelli group $\mathcal{I}(\Sigma_{g,1}, \{*\})$ is classically denoted by $\mathcal{I}_{g,1}$. By definition, 
the group $A(\Sigma_{g,1}, \{*\})$ is exactly the topological automorphism group of the complete Hopf algebra $\widehat{K\pi}$, $A(\Sigma_{g,1}, \{*\}) = {\rm Aut}(\widehat{K\pi}, \Delta)$. A symplectic expansion $\theta$ induces an isomorphism of complete Hopf algebras $\theta: \widehat{K\pi} \cong \widehat{T}$, and  
an isomorphism of the automorphism groups $\theta_*: A(\Sigma_{g,1}, \{*\}) = {\rm Aut}(\widehat{K\pi}, \Delta) \cong {\rm Aut}(\widehat{T}, \Delta)$. 
Massuyeau's improvement $\rho^\theta$ of the total Johnson map is defined so that the diagram 
$$
\begin{CD}
\mathcal{I}_{g,1} @>{\rho^\theta}>> \mathfrak{l}_g^+\\
@| @V{\exp}VV\\
\mathcal{I}_{g,1} @>{\theta_*\circ\widehat{\sf DN}}>> {\rm Aut}(\widehat{T}, \Delta)
\end{CD}
$$
commutes. Hence we obtain the following. 

\begin{cor} For any symplectic expansion $\theta$, the diagram
$$
\begin{CD}
\mathcal{I}(\Sigma_{g,1}, \{*\}) @>{\tau}>> L(\Sigma_{g,1}, \{*\})\\
@| @V{-\lambda_\theta}VV\\
\mathcal{I}_{g,1} @>{\rho^\theta}>> \mathfrak{l}_g^+
\end{CD}
$$
commutes. In other words, the geometric Johnson homomorphism $\tau$ for $(\Sigma_{g,1}, \{*\})$ is essentially the same as Massuyeau's map $\rho^\theta$, 
since $-\lambda_\theta$ is an isomorphism of Lie algebras. 
\end{cor}
The graded quotients of the total Johnson map with respect to the Johnson filtration on $\mathcal{I}_{g,1}$ and the filtration on ${\rm Aut}(\widehat{T}, \Delta)$ induced by that on $\widehat{T}$ coincides with the classical Johnson homomorphisms of all degrees \cite{Ka}. See also \cite{Mas}. Hence, so are the graded quotients of the geometric Johnson homomorphism $\tau$. 
\par
Finally we discuss the naturality of the geometric Johnson homomorphism with respect to an embedding. Let $S$ and
$S'$ be compact oriented surfaces with non-empty boundary, and $i\colon S
\hookrightarrow S'$ an embedding. 
Choose a single point on each boundary component of these
surfaces, and let $E\subset \partial S$ and $E' \subset \partial S'$
be the sets of these points. Then we have the natural homomorphism
$$
\imath\colon \mathcal{M}(S,E) \to \mathcal{M}(S',E')
$$
extending diffeomorphisms by the identity on the complement $S\setminus
S'$. Then 
\begin{prop}
\label{tau-compat}
There exists a Lie algebra homomorphism
$i\colon L^+(S,E) \to L^+(S^{\prime},E^{\prime})$, which
depends on the embedding $i\colon S \hookrightarrow S^{\prime}$,
such that the diagram 
$$
\begin{CD}
\ISE @>{\imath}>> \mathcal{I}(S',E')\\
@V{\tau}VV @V{\tau}VV \\
L^+(S,E) @>{i}>> L^+(S',E')
\end{CD}
$$
commutes.
\end{prop}
\begin{proof}
First we check the top horizontal arrow is well-defined.
Denote $S_1 := S$, $E_1 := E$, $S_2 := \overline{S'\setminus S}$, and 
$E_2 := (E\cup E')\cap S_2$. Then the inclusion homomorphism 
$H_1(S,E;\mathbb{Z}) \oplus H_1(S_2, E_2;\mathbb{Z}) \to H_1(S',E\cup
E';\mathbb{Z})$ is surjective, while the inclusion homomorphism $H_1(S',
E'; \mathbb{Z}) \to H_1(S',E\cup E';\mathbb{Z})$ is injective. 
Hence the homomorphism $\imath$ maps $\ISE$ into $\mathcal{I}(S',E')$. \par
Moreover the pairs $(S_1,E_1)$ and $(S_2,E_2)$
satisfy the assumption of the van Kampen theorem (Proposition
\ref{32vKT}). Hence $K\CC_3 = K\CC(S', E \cup E')$ is generated by 
$K\CSE$ and $K\CC(S_2,E_2)$. Hence, if $u \in \widehat{K{\hat\pi}}(S^*)$
satisfies $\sigma(u) = 0 \in \Der \widehat{K\CSE}$, then we have
$\sigma(i(u)) = 0 \in \Der K\CC(S', E \cup E')$. Here $i\colon
\widehat{K{\hat\pi}}(S^*) \to \widehat{K{\hat\pi}}({S'}^*)$ is the
inclusion homomorphism. This implies 
the homomorphism 
$$
i\colon  \sigma(\widehat{K{\hat\pi}}(S^*)) \to 
\sigma(\widehat{K{\hat\pi}}({S'}^*)) \subset 
\Der K\CC(S', E \cup E')
$$
extending derivations by $0$ on $S_2$ is well-defined.
Post-composing the forgetful map $\phi \colon \Der K\CC(S^{\prime},E \cup E^{\prime})
\to \Der K\CC(S^{\prime},E^{\prime})$ and restricting $\phi \circ i$ to
$L^+(S,E)$, we obtain a Lie algebra homomorphism 
$i\colon L^+(S,E) \to L^+(S',E')$. \par
From Theorems \ref{putman} and \ref{putman0},
the group $\ISE$ is generated by twists about
$P$-separating curves, twists about $P$-bounding pairs
and commutators of simply intersecting pairs.
It is clear that the diagram commutes on these mapping classes.
This proves the proposition.
\end{proof}

Recently Church \cite{Chu} introduced the first Johnson homomorphism for
all kinds of Putman's partitioned Torelli groups. It would be very
interesting to describe an explicit relation between Church's
homomorphisms and ours.

\noindent \textsc{Nariya Kawazumi\\
Department of Mathematical Sciences,\\
University of Tokyo,\\
3-8-1 Komaba Meguro-ku Tokyo 153-8914 JAPAN}\\
\noindent \texttt{E-mail address: kawazumi@ms.u-tokyo.ac.jp}

\vspace{0.5cm}

\noindent \textsc{Yusuke Kuno\\
Department of Mathematics,\\
Tsuda College,\\
2-1-1 Tsuda-Machi, Kodaira-shi, Tokyo 187-8577, JAPAN}\\
\noindent \texttt{E-mail address: kunotti@tsuda.ac.jp}

\end{document}

%% file: DNtheorem.tex
\unitlength 0.1in
\begin{picture}( 34.0000, 20.0000)(  4.0000,-24.0000)
%
{\color[named]{Black}{%
\special{pn 13}%
\special{ar 3600 600 100 200  0.0000000 6.2831853}%
}}%
%
{\color[named]{Black}{%
\special{pn 13}%
\special{ar 3600 2000 200 400  0.0000000 6.2831853}%
}}%
%
{\color[named]{Black}{%
\special{pn 13}%
\special{ar 1400 1400 1000 1000  1.5707963 4.7123890}%
}}%
%
{\color[named]{Black}{%
\special{pn 13}%
\special{pa 1400 2400}%
\special{pa 3600 2400}%
\special{fp}%
}}%
%
{\color[named]{Black}{%
\special{pn 13}%
\special{pa 3600 400}%
\special{pa 1400 400}%
\special{fp}%
}}%
%
{\color[named]{Black}{%
\special{pn 13}%
\special{ar 1400 1400 400 400  0.0000000 6.2831853}%
}}%
%
{\color[named]{Black}{%
\special{pn 13}%
\special{ar 3600 1500 200 100  1.5707963 4.7123890}%
}}%
%
{\color[named]{Black}{%
\special{pn 13}%
\special{ar 3600 900 200 100  1.5707963 4.7123890}%
}}%
%
{\color[named]{Black}{%
\special{pn 4}%
\special{sh 1}%
\special{ar 3500 600 26 26 0  6.28318530717959E+0000}%
}}%
%
{\color[named]{Black}{%
\special{pn 4}%
\special{sh 1}%
\special{ar 3500 1200 26 26 0  6.28318530717959E+0000}%
}}%
%
{\color[named]{Black}{%
\special{pn 8}%
\special{ar 1400 1400 800 800  1.5707963 6.2831853}%
}}%
%
{\color[named]{Black}{%
\special{pn 8}%
\special{pa 1400 2200}%
\special{pa 3440 2200}%
\special{fp}%
}}%
%
{\color[named]{Black}{%
\special{pn 4}%
\special{sh 1}%
\special{ar 3430 2200 26 26 0  6.28318530717959E+0000}%
}}%
%
{\color[named]{Black}{%
\special{pn 8}%
\special{ar 2800 1400 600 600  1.5707963 3.1415927}%
}}%
%
{\color[named]{Black}{%
\special{pn 8}%
\special{pa 2800 2000}%
\special{pa 3400 2000}%
\special{fp}%
}}%
%
{\color[named]{Black}{%
\special{pn 4}%
\special{sh 1}%
\special{ar 3400 2000 26 26 0  6.28318530717959E+0000}%
}}%
%
{\color[named]{Black}{%
\special{pn 8}%
\special{ar 2500 1400 700 700  1.5707963 3.1415927}%
}}%
%
{\color[named]{Black}{%
\special{pn 8}%
\special{pa 2500 2100}%
\special{pa 3420 2100}%
\special{fp}%
}}%
%
{\color[named]{Black}{%
\special{pn 4}%
\special{sh 1}%
\special{ar 3410 2100 26 26 0  6.28318530717959E+0000}%
}}%
%
{\color[named]{Black}{%
\special{pn 8}%
\special{pa 1800 1400}%
\special{pa 1800 800}%
\special{da 0.070}%
}}%
%
{\color[named]{Black}{%
\special{pn 8}%
\special{ar 2200 800 400 400  3.1415927 3.2915927}%
\special{ar 2200 800 400 400  3.3815927 3.5315927}%
\special{ar 2200 800 400 400  3.6215927 3.7715927}%
\special{ar 2200 800 400 400  3.8615927 4.0115927}%
\special{ar 2200 800 400 400  4.1015927 4.2515927}%
\special{ar 2200 800 400 400  4.3415927 4.4915927}%
\special{ar 2200 800 400 400  4.5815927 4.7123890}%
}}%
%
{\color[named]{Black}{%
\special{pn 8}%
\special{ar 2200 800 400 400  4.7123890 6.2831853}%
}}%
%
{\color[named]{Black}{%
\special{pn 8}%
\special{pa 2600 800}%
\special{pa 2600 1600}%
\special{fp}%
}}%
%
{\color[named]{Black}{%
\special{pn 8}%
\special{ar 2900 1600 300 300  1.5707963 3.1415927}%
}}%
%
{\color[named]{Black}{%
\special{pn 8}%
\special{pa 2900 1900}%
\special{pa 3410 1900}%
\special{fp}%
}}%
%
{\color[named]{Black}{%
\special{pn 4}%
\special{sh 1}%
\special{ar 3410 1900 26 26 0  6.28318530717959E+0000}%
}}%
%
{\color[named]{Black}{%
\special{pn 4}%
\special{sh 1}%
\special{ar 3430 1800 26 26 0  6.28318530717959E+0000}%
}}%
%
{\color[named]{Black}{%
\special{pn 8}%
\special{ar 3200 1600 200 200  1.5707963 3.1415927}%
}}%
%
{\color[named]{Black}{%
\special{pn 8}%
\special{pa 3000 1600}%
\special{pa 3000 800}%
\special{fp}%
}}%
%
{\color[named]{Black}{%
\special{pn 8}%
\special{ar 3200 800 200 200  3.1415927 4.7123890}%
}}%
%
{\color[named]{Black}{%
\special{pn 8}%
\special{pa 3200 600}%
\special{pa 3500 600}%
\special{fp}%
}}%
%
{\color[named]{Black}{%
\special{pn 8}%
\special{pa 3200 1800}%
\special{pa 3430 1800}%
\special{fp}%
}}%
\put(14.6000,-21.5000){\makebox(0,0)[lb]{$\alpha_1$}}%
\put(23.9000,-7.5000){\makebox(0,0)[lb]{$\beta_1$}}%
\put(28.3000,-11.5000){\makebox(0,0)[lb]{$\gamma_1$}}%
\put(31.9000,-10.7000){\makebox(0,0)[lb]{$\delta_1$}}%
%
{\color[named]{Black}{%
\special{pn 13}%
\special{ar 3600 1200 100 200  4.7123890 6.2831853}%
\special{ar 3600 1200 100 200  0.0000000 1.5707963}%
}}%
%
{\color[named]{Black}{%
\special{pn 4}%
\special{sh 1}%
\special{ar 3450 1730 26 26 0  6.28318530717959E+0000}%
}}%
%
{\color[named]{Black}{%
\special{pn 4}%
\special{sh 1}%
\special{ar 3500 1660 26 26 0  6.28318530717959E+0000}%
}}%
%
{\color[named]{Black}{%
\special{pn 8}%
\special{pa 3450 1730}%
\special{pa 3300 1730}%
\special{fp}%
}}%
%
{\color[named]{Black}{%
\special{pn 8}%
\special{ar 3300 1630 100 100  1.5707963 3.1415927}%
}}%
%
{\color[named]{Black}{%
\special{pn 8}%
\special{pa 3490 1660}%
\special{pa 3400 1660}%
\special{fp}%
}}%
%
{\color[named]{Black}{%
\special{pn 8}%
\special{ar 3400 1560 100 100  1.5707963 3.1415927}%
}}%
%
{\color[named]{Black}{%
\special{pn 8}%
\special{pa 3500 1300}%
\special{pa 3400 1300}%
\special{fp}%
}}%
%
{\color[named]{Black}{%
\special{pn 8}%
\special{ar 3400 1400 100 100  3.1415927 4.7123890}%
}}%
%
{\color[named]{Black}{%
\special{pn 8}%
\special{pa 3300 1400}%
\special{pa 3300 1560}%
\special{fp}%
}}%
%
{\color[named]{Black}{%
\special{pn 8}%
\special{ar 3500 1200 100 100  4.7123890 6.2831853}%
\special{ar 3500 1200 100 100  0.0000000 1.5707963}%
}}%
%
{\color[named]{Black}{%
\special{pn 8}%
\special{pa 3500 1100}%
\special{pa 3300 1100}%
\special{fp}%
}}%
%
{\color[named]{Black}{%
\special{pn 8}%
\special{ar 3300 1200 100 100  3.1415927 4.7123890}%
}}%
%
{\color[named]{Black}{%
\special{pn 8}%
\special{pa 3200 1200}%
\special{pa 3200 1630}%
\special{fp}%
}}%
\end{picture}%

%% file: eghj.tex
\unitlength 0.1in
\begin{picture}( 35.0000, 29.3000)(  4.0000,-32.0000)
%
{\color[named]{Black}{%
\special{pn 13}%
\special{ar 1800 1000 100 400  0.0000000 6.2831853}%
}}%
%
{\color[named]{Black}{%
\special{pn 13}%
\special{pa 1800 600}%
\special{pa 800 600}%
\special{fp}%
}}%
%
{\color[named]{Black}{%
\special{pn 13}%
\special{pa 1800 1400}%
\special{pa 800 1400}%
\special{fp}%
}}%
%
{\color[named]{Black}{%
\special{pn 13}%
\special{ar 1200 1000 100 100  0.0000000 6.2831853}%
}}%
%
{\color[named]{Black}{%
\special{pn 8}%
\special{ar 1200 1000 270 270  0.0000000 6.2831853}%
}}%
%
{\color[named]{Black}{%
\special{pn 13}%
\special{ar 800 1000 200 400  1.5707963 4.7123890}%
}}%
%
{\color[named]{Black}{%
\special{pn 8}%
\special{ar 1200 750 100 150  1.5707963 4.7123890}%
}}%
%
{\color[named]{Black}{%
\special{pn 8}%
\special{ar 1200 750 100 150  4.7123890 5.1923890}%
\special{ar 1200 750 100 150  5.4803890 5.9603890}%
\special{ar 1200 750 100 150  6.2483890 6.7283890}%
\special{ar 1200 750 100 150  7.0163890 7.4963890}%
\special{ar 1200 750 100 150  7.7843890 7.8539816}%
}}%
\put(4.0000,-4.0000){\makebox(0,0)[lb]{(e)}}%
%
{\color[named]{Black}{%
\special{pn 13}%
\special{ar 3200 1000 600 600  0.0000000 6.2831853}%
}}%
%
{\color[named]{Black}{%
\special{pn 13}%
\special{ar 2900 1000 100 100  0.0000000 6.2831853}%
}}%
%
{\color[named]{Black}{%
\special{pn 13}%
\special{ar 3500 1000 100 100  0.0000000 6.2831853}%
}}%
%
{\color[named]{Black}{%
\special{pn 4}%
\special{sh 1}%
\special{ar 3200 1600 26 26 0  6.28318530717959E+0000}%
}}%
%
{\color[named]{Black}{%
\special{pn 4}%
\special{sh 1}%
\special{ar 3200 400 26 26 0  6.28318530717959E+0000}%
}}%
%
{\color[named]{Black}{%
\special{pn 8}%
\special{pa 3200 400}%
\special{pa 3200 1600}%
\special{fp}%
}}%
\put(24.0000,-4.0000){\makebox(0,0)[lb]{(g)}}%
\put(24.7000,-14.7000){\makebox(0,0)[lb]{$\partial S$}}%
\put(28.3000,-8.4000){\makebox(0,0)[lb]{$\partial N$}}%
\put(34.3000,-8.4000){\makebox(0,0)[lb]{$\partial N$}}%
%
{\color[named]{Black}{%
\special{pn 13}%
\special{ar 1400 2600 600 600  0.0000000 6.2831853}%
}}%
%
{\color[named]{Black}{%
\special{pn 13}%
\special{ar 1100 2600 100 100  0.0000000 6.2831853}%
}}%
%
{\color[named]{Black}{%
\special{pn 13}%
\special{ar 1700 2600 100 100  0.0000000 6.2831853}%
}}%
%
{\color[named]{Black}{%
\special{pn 4}%
\special{sh 1}%
\special{ar 1200 2600 26 26 0  6.28318530717959E+0000}%
}}%
%
{\color[named]{Black}{%
\special{pn 4}%
\special{sh 1}%
\special{ar 1600 2600 26 26 0  6.28318530717959E+0000}%
}}%
%
{\color[named]{Black}{%
\special{pn 8}%
\special{pa 1600 2600}%
\special{pa 1200 2600}%
\special{fp}%
}}%
\put(10.3000,-24.4000){\makebox(0,0)[lb]{$\partial S$}}%
\put(15.1000,-24.4000){\makebox(0,0)[lb]{$\partial S$}}%
\put(6.5000,-30.7000){\makebox(0,0)[lb]{$\partial N$}}%
\put(4.0000,-20.0000){\makebox(0,0)[lb]{(h)}}%
\put(24.0000,-20.0000){\makebox(0,0)[lb]{(j)}}%
%
{\color[named]{Black}{%
\special{pn 13}%
\special{ar 3800 2600 100 400  0.0000000 6.2831853}%
}}%
%
{\color[named]{Black}{%
\special{pn 13}%
\special{pa 3800 2200}%
\special{pa 2800 2200}%
\special{fp}%
}}%
%
{\color[named]{Black}{%
\special{pn 13}%
\special{pa 3800 3000}%
\special{pa 2800 3000}%
\special{fp}%
}}%
%
{\color[named]{Black}{%
\special{pn 13}%
\special{ar 2800 2600 100 400  1.5707963 4.7123890}%
}}%
%
{\color[named]{Black}{%
\special{pn 13}%
\special{ar 2800 2600 100 400  4.7123890 4.9523890}%
\special{ar 2800 2600 100 400  5.0963890 5.3363890}%
\special{ar 2800 2600 100 400  5.4803890 5.7203890}%
\special{ar 2800 2600 100 400  5.8643890 6.1043890}%
\special{ar 2800 2600 100 400  6.2483890 6.4883890}%
\special{ar 2800 2600 100 400  6.6323890 6.8723890}%
\special{ar 2800 2600 100 400  7.0163890 7.2563890}%
\special{ar 2800 2600 100 400  7.4003890 7.6403890}%
\special{ar 2800 2600 100 400  7.7843890 7.8539816}%
}}%
%
{\color[named]{Black}{%
\special{pn 4}%
\special{sh 1}%
\special{ar 3300 2600 26 26 0  6.28318530717959E+0000}%
}}%
%
{\color[named]{Black}{%
\special{pn 4}%
\special{sh 1}%
\special{ar 3710 2800 26 26 0  6.28318530717959E+0000}%
}}%
%
{\color[named]{Black}{%
\special{pn 4}%
\special{sh 1}%
\special{ar 3710 2400 26 26 0  6.28318530717959E+0000}%
}}%
%
{\color[named]{Black}{%
\special{pn 8}%
\special{pa 3710 2400}%
\special{pa 3300 2400}%
\special{fp}%
}}%
%
{\color[named]{Black}{%
\special{pn 8}%
\special{pa 3710 2800}%
\special{pa 3300 2800}%
\special{fp}%
}}%
%
{\color[named]{Black}{%
\special{pn 8}%
\special{ar 3300 2600 200 200  1.5707963 4.7123890}%
}}%
\end{picture}%

%% file: chi-1.tex
\unitlength 0.1in
\begin{picture}( 39.4000, 24.0000)(  1.6000,-28.0000)
%
{\color[named]{Black}{%
\special{pn 13}%
\special{ar 3800 2200 200 600  0.0000000 6.2831853}%
}}%
%
{\color[named]{Black}{%
\special{pn 13}%
\special{ar 3800 1200 100 200  0.0000000 6.2831853}%
}}%
%
{\color[named]{Black}{%
\special{pn 13}%
\special{ar 3800 600 100 200  0.0000000 6.2831853}%
}}%
%
{\color[named]{Black}{%
\special{pn 13}%
\special{ar 600 2600 100 200  1.5707963 4.7123890}%
}}%
%
{\color[named]{Black}{%
\special{pn 13}%
\special{ar 600 2600 100 200  4.7123890 5.1123890}%
\special{ar 600 2600 100 200  5.3523890 5.7523890}%
\special{ar 600 2600 100 200  5.9923890 6.3923890}%
\special{ar 600 2600 100 200  6.6323890 7.0323890}%
\special{ar 600 2600 100 200  7.2723890 7.6723890}%
}}%
%
{\color[named]{Black}{%
\special{pn 13}%
\special{ar 600 1900 100 200  1.5707963 4.7123890}%
}}%
%
{\color[named]{Black}{%
\special{pn 13}%
\special{ar 600 1900 100 200  4.7123890 5.1123890}%
\special{ar 600 1900 100 200  5.3523890 5.7523890}%
\special{ar 600 1900 100 200  5.9923890 6.3923890}%
\special{ar 600 1900 100 200  6.6323890 7.0323890}%
\special{ar 600 1900 100 200  7.2723890 7.6723890}%
}}%
%
{\color[named]{Black}{%
\special{pn 13}%
\special{ar 600 1200 200 200  1.5707963 4.7123890}%
}}%
%
{\color[named]{Black}{%
\special{pn 13}%
\special{ar 600 600 200 200  1.5707963 4.7123890}%
}}%
%
{\color[named]{Black}{%
\special{pn 13}%
\special{ar 600 900 100 100  4.7123890 6.2831853}%
\special{ar 600 900 100 100  0.0000000 1.5707963}%
}}%
%
{\color[named]{Black}{%
\special{pn 13}%
\special{ar 600 1550 150 150  4.7123890 6.2831853}%
\special{ar 600 1550 150 150  0.0000000 1.5707963}%
}}%
%
{\color[named]{Black}{%
\special{pn 13}%
\special{ar 600 2250 150 150  4.7123890 6.2831853}%
\special{ar 600 2250 150 150  0.0000000 1.5707963}%
}}%
%
{\color[named]{Black}{%
\special{pn 4}%
\special{sh 1}%
\special{ar 400 1200 26 26 0  6.28318530717959E+0000}%
}}%
%
{\color[named]{Black}{%
\special{pn 4}%
\special{sh 1}%
\special{ar 400 600 26 26 0  6.28318530717959E+0000}%
}}%
%
{\color[named]{Black}{%
\special{pn 13}%
\special{pa 3800 2800}%
\special{pa 600 2800}%
\special{fp}%
}}%
%
{\color[named]{Black}{%
\special{pn 13}%
\special{pa 3800 400}%
\special{pa 600 400}%
\special{fp}%
}}%
%
{\color[named]{Black}{%
\special{pn 13}%
\special{ar 3800 900 100 100  1.5707963 4.7123890}%
}}%
%
{\color[named]{Black}{%
\special{pn 13}%
\special{ar 3800 1500 100 100  1.5707963 4.7123890}%
}}%
%
{\color[named]{Black}{%
\special{pn 4}%
\special{sh 1}%
\special{ar 3700 600 26 26 0  6.28318530717959E+0000}%
}}%
%
{\color[named]{Black}{%
\special{pn 4}%
\special{sh 1}%
\special{ar 3700 1200 26 26 0  6.28318530717959E+0000}%
}}%
%
{\color[named]{Black}{%
\special{pn 13}%
\special{ar 2600 900 200 200  0.0000000 6.2831853}%
}}%
%
{\color[named]{Black}{%
\special{pn 8}%
\special{ar 2600 550 100 150  1.5707963 4.7123890}%
}}%
%
{\color[named]{Black}{%
\special{pn 8}%
\special{ar 2600 550 100 150  4.7123890 5.1923890}%
\special{ar 2600 550 100 150  5.4803890 5.9603890}%
\special{ar 2600 550 100 150  6.2483890 6.7283890}%
\special{ar 2600 550 100 150  7.0163890 7.4963890}%
\special{ar 2600 550 100 150  7.7843890 7.8539816}%
}}%
%
{\color[named]{Black}{%
\special{pn 8}%
\special{ar 2600 900 300 300  0.0000000 6.2831853}%
}}%
%
{\color[named]{Black}{%
\special{pn 4}%
\special{sh 1}%
\special{ar 3690 2700 26 26 0  6.28318530717959E+0000}%
}}%
%
{\color[named]{Black}{%
\special{pn 8}%
\special{pa 3670 2600}%
\special{pa 3930 2600}%
\special{fp}%
}}%
%
{\color[named]{Black}{%
\special{pn 4}%
\special{sh 1}%
\special{ar 3940 2600 26 26 0  6.28318530717959E+0000}%
}}%
%
{\color[named]{Black}{%
\special{pn 4}%
\special{sh 1}%
\special{ar 3630 2500 26 26 0  6.28318530717959E+0000}%
}}%
%
{\color[named]{Black}{%
\special{pn 8}%
\special{pa 3620 2400}%
\special{pa 3970 2400}%
\special{fp}%
}}%
%
{\color[named]{Black}{%
\special{pn 4}%
\special{sh 1}%
\special{ar 3980 2400 26 26 0  6.28318530717959E+0000}%
}}%
%
{\color[named]{Black}{%
\special{pn 8}%
\special{ar 850 2250 100 100  3.1415927 3.7415927}%
\special{ar 850 2250 100 100  4.1015927 4.7015927}%
}}%
%
{\color[named]{Black}{%
\special{pn 8}%
\special{ar 850 2250 100 100  1.5707963 3.1415927}%
}}%
%
{\color[named]{Black}{%
\special{pn 8}%
\special{ar 850 2450 100 100  4.7123890 6.2831853}%
}}%
%
{\color[named]{Black}{%
\special{pn 8}%
\special{pa 950 2450}%
\special{pa 950 2600}%
\special{fp}%
}}%
%
{\color[named]{Black}{%
\special{pn 8}%
\special{ar 1050 2600 100 100  1.5707963 3.1415927}%
}}%
%
{\color[named]{Black}{%
\special{pn 8}%
\special{pa 1050 2700}%
\special{pa 3690 2700}%
\special{fp}%
}}%
%
{\color[named]{Black}{%
\special{pn 8}%
\special{pa 850 2150}%
\special{pa 950 2150}%
\special{da 0.070}%
}}%
%
{\color[named]{Black}{%
\special{pn 8}%
\special{ar 950 2250 100 100  4.7123890 5.3123890}%
\special{ar 950 2250 100 100  5.6723890 6.2723890}%
}}%
%
{\color[named]{Black}{%
\special{pn 8}%
\special{pa 1050 2250}%
\special{pa 1050 2500}%
\special{da 0.070}%
}}%
%
{\color[named]{Black}{%
\special{pn 8}%
\special{ar 1150 2500 100 100  1.5707963 2.1707963}%
\special{ar 1150 2500 100 100  2.5307963 3.1307963}%
}}%
%
{\color[named]{Black}{%
\special{pn 8}%
\special{pa 1150 2600}%
\special{pa 3670 2600}%
\special{da 0.070}%
}}%
%
{\color[named]{Black}{%
\special{pn 8}%
\special{ar 850 1550 100 100  3.1415927 3.7415927}%
\special{ar 850 1550 100 100  4.1015927 4.7015927}%
}}%
%
{\color[named]{Black}{%
\special{pn 8}%
\special{ar 850 1550 100 100  1.5707963 3.1415927}%
}}%
%
{\color[named]{Black}{%
\special{pn 8}%
\special{pa 850 1450}%
\special{pa 950 1450}%
\special{da 0.070}%
}}%
%
{\color[named]{Black}{%
\special{pn 8}%
\special{pa 850 1650}%
\special{pa 1050 1650}%
\special{fp}%
}}%
%
{\color[named]{Black}{%
\special{pn 8}%
\special{ar 1050 1750 100 100  4.7123890 6.2831853}%
}}%
%
{\color[named]{Black}{%
\special{pn 8}%
\special{pa 1150 1750}%
\special{pa 1150 2400}%
\special{fp}%
}}%
%
{\color[named]{Black}{%
\special{pn 8}%
\special{ar 1250 2400 100 100  1.5707963 3.1415927}%
}}%
%
{\color[named]{Black}{%
\special{pn 8}%
\special{pa 1250 2500}%
\special{pa 3630 2500}%
\special{fp}%
}}%
%
{\color[named]{Black}{%
\special{pn 8}%
\special{pa 950 1450}%
\special{pa 1150 1450}%
\special{da 0.070}%
}}%
%
{\color[named]{Black}{%
\special{pn 8}%
\special{ar 1150 1550 100 100  4.7123890 5.3123890}%
\special{ar 1150 1550 100 100  5.6723890 6.2723890}%
}}%
%
{\color[named]{Black}{%
\special{pn 8}%
\special{pa 1250 1550}%
\special{pa 1250 2300}%
\special{da 0.070}%
}}%
%
{\color[named]{Black}{%
\special{pn 8}%
\special{ar 1350 2300 100 100  1.5707963 2.1707963}%
\special{ar 1350 2300 100 100  2.5307963 3.1307963}%
}}%
%
{\color[named]{Black}{%
\special{pn 8}%
\special{pa 1350 2400}%
\special{pa 3620 2400}%
\special{da 0.070}%
}}%
%
{\color[named]{Black}{%
\special{pn 8}%
\special{ar 800 900 100 100  3.1415927 3.7415927}%
\special{ar 800 900 100 100  4.1015927 4.7015927}%
}}%
%
{\color[named]{Black}{%
\special{pn 8}%
\special{ar 800 900 100 100  1.5707963 3.1415927}%
}}%
%
{\color[named]{Black}{%
\special{pn 8}%
\special{pa 800 1000}%
\special{pa 1250 1000}%
\special{fp}%
}}%
%
{\color[named]{Black}{%
\special{pn 8}%
\special{ar 1250 1100 100 100  4.7123890 6.2831853}%
}}%
%
{\color[named]{Black}{%
\special{pn 8}%
\special{pa 1350 1100}%
\special{pa 1350 2200}%
\special{fp}%
}}%
%
{\color[named]{Black}{%
\special{pn 8}%
\special{ar 1450 2200 100 100  1.5707963 3.1415927}%
}}%
%
{\color[named]{Black}{%
\special{pn 8}%
\special{pa 1450 2300}%
\special{pa 3610 2300}%
\special{fp}%
}}%
%
{\color[named]{Black}{%
\special{pn 4}%
\special{sh 1}%
\special{ar 3610 2300 26 26 0  6.28318530717959E+0000}%
}}%
%
{\color[named]{Black}{%
\special{pn 8}%
\special{pa 800 800}%
\special{pa 1350 800}%
\special{da 0.070}%
}}%
%
{\color[named]{Black}{%
\special{pn 8}%
\special{ar 1350 900 100 100  4.7123890 5.3123890}%
\special{ar 1350 900 100 100  5.6723890 6.2723890}%
}}%
%
{\color[named]{Black}{%
\special{pn 8}%
\special{pa 1450 900}%
\special{pa 1450 2100}%
\special{da 0.070}%
}}%
%
{\color[named]{Black}{%
\special{pn 8}%
\special{ar 1550 2100 100 100  1.5707963 2.1707963}%
\special{ar 1550 2100 100 100  2.5307963 3.1307963}%
}}%
%
{\color[named]{Black}{%
\special{pn 8}%
\special{pa 1550 2200}%
\special{pa 3620 2200}%
\special{da 0.070}%
}}%
%
{\color[named]{Black}{%
\special{pn 8}%
\special{pa 3620 2200}%
\special{pa 3970 2200}%
\special{fp}%
}}%
%
{\color[named]{Black}{%
\special{pn 4}%
\special{sh 1}%
\special{ar 3990 2200 26 26 0  6.28318530717959E+0000}%
}}%
%
{\color[named]{Black}{%
\special{pn 8}%
\special{ar 1700 800 100 400  3.1415927 4.7123890}%
}}%
%
{\color[named]{Black}{%
\special{pn 8}%
\special{ar 1700 800 100 400  4.7123890 4.9523890}%
\special{ar 1700 800 100 400  5.0963890 5.3363890}%
\special{ar 1700 800 100 400  5.4803890 5.7203890}%
\special{ar 1700 800 100 400  5.8643890 6.1043890}%
\special{ar 1700 800 100 400  6.2483890 6.2831853}%
}}%
%
{\color[named]{Black}{%
\special{pn 8}%
\special{pa 1600 800}%
\special{pa 1600 2000}%
\special{fp}%
}}%
%
{\color[named]{Black}{%
\special{pn 8}%
\special{ar 1700 2000 100 100  1.5707963 3.1415927}%
}}%
%
{\color[named]{Black}{%
\special{pn 8}%
\special{pa 1700 2100}%
\special{pa 3610 2100}%
\special{fp}%
}}%
%
{\color[named]{Black}{%
\special{pn 4}%
\special{sh 1}%
\special{ar 3600 2100 26 26 0  6.28318530717959E+0000}%
}}%
%
{\color[named]{Black}{%
\special{pn 8}%
\special{pa 1800 800}%
\special{pa 1800 1900}%
\special{da 0.070}%
}}%
%
{\color[named]{Black}{%
\special{pn 8}%
\special{ar 1900 1900 100 100  1.5707963 2.1707963}%
\special{ar 1900 1900 100 100  2.5307963 3.1307963}%
}}%
%
{\color[named]{Black}{%
\special{pn 8}%
\special{pa 1900 2000}%
\special{pa 3630 2000}%
\special{da 0.070}%
}}%
%
{\color[named]{Black}{%
\special{pn 8}%
\special{pa 3630 2000}%
\special{pa 3970 2000}%
\special{da 0.070}%
}}%
%
{\color[named]{Black}{%
\special{pn 4}%
\special{sh 1}%
\special{ar 3980 2000 26 26 0  6.28318530717959E+0000}%
}}%
%
{\color[named]{Black}{%
\special{pn 8}%
\special{ar 2300 700 200 200  3.1415927 4.7123890}%
}}%
%
{\color[named]{Black}{%
\special{pn 8}%
\special{pa 2100 700}%
\special{pa 2100 1700}%
\special{fp}%
}}%
%
{\color[named]{Black}{%
\special{pn 8}%
\special{ar 2300 1700 200 200  1.5707963 3.1415927}%
}}%
%
{\color[named]{Black}{%
\special{pn 8}%
\special{pa 2300 1900}%
\special{pa 3630 1900}%
\special{fp}%
}}%
%
{\color[named]{Black}{%
\special{pn 4}%
\special{sh 1}%
\special{ar 3630 1900 26 26 0  6.28318530717959E+0000}%
}}%
%
{\color[named]{Black}{%
\special{pn 8}%
\special{pa 2300 500}%
\special{pa 2900 500}%
\special{fp}%
}}%
%
{\color[named]{Black}{%
\special{pn 8}%
\special{ar 2900 700 200 200  4.7123890 6.2831853}%
}}%
%
{\color[named]{Black}{%
\special{pn 8}%
\special{pa 3100 700}%
\special{pa 3100 1600}%
\special{fp}%
}}%
%
{\color[named]{Black}{%
\special{pn 8}%
\special{ar 3300 1600 200 200  1.5707963 3.1415927}%
}}%
%
{\color[named]{Black}{%
\special{pn 8}%
\special{pa 3300 1800}%
\special{pa 3650 1800}%
\special{fp}%
}}%
%
{\color[named]{Black}{%
\special{pn 4}%
\special{sh 1}%
\special{ar 3650 1800 26 26 0  6.28318530717959E+0000}%
}}%
%
{\color[named]{Black}{%
\special{pn 4}%
\special{sh 1}%
\special{ar 3670 1730 26 26 0  6.28318530717959E+0000}%
}}%
%
{\color[named]{Black}{%
\special{pn 4}%
\special{sh 1}%
\special{ar 3710 1660 26 26 0  6.28318530717959E+0000}%
}}%
%
{\color[named]{Black}{%
\special{pn 8}%
\special{pa 3710 1660}%
\special{pa 3600 1660}%
\special{fp}%
}}%
%
{\color[named]{Black}{%
\special{pn 8}%
\special{ar 3600 1560 100 100  1.5707963 3.1415927}%
}}%
%
{\color[named]{Black}{%
\special{pn 8}%
\special{pa 3500 1560}%
\special{pa 3500 1300}%
\special{fp}%
}}%
%
{\color[named]{Black}{%
\special{pn 8}%
\special{ar 3600 1300 100 100  3.1415927 4.7123890}%
}}%
%
{\color[named]{Black}{%
\special{pn 8}%
\special{pa 3600 1200}%
\special{pa 3700 1200}%
\special{fp}%
}}%
%
{\color[named]{Black}{%
\special{pn 8}%
\special{pa 3660 1730}%
\special{pa 3500 1730}%
\special{fp}%
}}%
%
{\color[named]{Black}{%
\special{pn 8}%
\special{ar 3500 1630 100 100  1.5707963 3.1415927}%
}}%
%
{\color[named]{Black}{%
\special{pn 8}%
\special{pa 3400 1630}%
\special{pa 3400 700}%
\special{fp}%
}}%
%
{\color[named]{Black}{%
\special{pn 8}%
\special{ar 3500 700 100 100  3.1415927 4.7123890}%
}}%
%
{\color[named]{Black}{%
\special{pn 8}%
\special{pa 3500 600}%
\special{pa 3700 600}%
\special{fp}%
}}%
\put(1.6000,-19.4000){\makebox(0,0)[lb]{$\partial N$}}%
\put(1.6000,-26.4000){\makebox(0,0)[lb]{$\partial N$}}%
\put(41.0000,-20.7000){\makebox(0,0)[lb]{$\partial S$}}%
\put(39.6000,-12.5000){\makebox(0,0)[lb]{$\partial S$}}%
\put(39.6000,-6.5000){\makebox(0,0)[lb]{$\partial S$}}%
\end{picture}%

%% file: chi-2.tex
\unitlength 0.1in
\begin{picture}( 32.0000, 16.0000)(  2.0000,-22.0000)
%
{\color[named]{Black}{%
\special{pn 13}%
\special{ar 3200 2000 100 200  0.0000000 6.2831853}%
}}%
%
{\color[named]{Black}{%
\special{pn 13}%
\special{ar 3200 1400 100 200  0.0000000 6.2831853}%
}}%
%
{\color[named]{Black}{%
\special{pn 13}%
\special{ar 3200 800 200 200  4.7123890 6.2831853}%
\special{ar 3200 800 200 200  0.0000000 1.5707963}%
}}%
%
{\color[named]{Black}{%
\special{pn 13}%
\special{pa 3200 600}%
\special{pa 1000 600}%
\special{fp}%
}}%
%
{\color[named]{Black}{%
\special{pn 13}%
\special{ar 1000 1400 800 800  1.5707963 4.7123890}%
}}%
%
{\color[named]{Black}{%
\special{pn 13}%
\special{pa 1000 2200}%
\special{pa 3200 2200}%
\special{fp}%
}}%
%
{\color[named]{Black}{%
\special{pn 13}%
\special{ar 3200 1700 100 100  1.5707963 4.7123890}%
}}%
%
{\color[named]{Black}{%
\special{pn 13}%
\special{ar 3200 1100 100 100  1.5707963 4.7123890}%
}}%
%
{\color[named]{Black}{%
\special{pn 4}%
\special{sh 1}%
\special{ar 3400 800 26 26 0  6.28318530717959E+0000}%
}}%
%
{\color[named]{Black}{%
\special{pn 13}%
\special{ar 1200 1000 200 200  0.0000000 6.2831853}%
}}%
%
{\color[named]{Black}{%
\special{pn 13}%
\special{ar 1200 1800 200 200  0.0000000 6.2831853}%
}}%
%
{\color[named]{Black}{%
\special{pn 8}%
\special{ar 1200 700 50 100  1.5707963 4.7123890}%
}}%
%
{\color[named]{Black}{%
\special{pn 8}%
\special{ar 1200 700 50 100  4.7123890 5.5123890}%
\special{ar 1200 700 50 100  5.9923890 6.7923890}%
\special{ar 1200 700 50 100  7.2723890 7.8539816}%
}}%
%
{\color[named]{Black}{%
\special{pn 8}%
\special{ar 1200 1000 300 300  0.0000000 6.2831853}%
}}%
%
{\color[named]{Black}{%
\special{pn 8}%
\special{ar 1200 1400 100 200  1.5707963 4.7123890}%
}}%
%
{\color[named]{Black}{%
\special{pn 8}%
\special{ar 1200 1400 100 200  4.7123890 5.1123890}%
\special{ar 1200 1400 100 200  5.3523890 5.7523890}%
\special{ar 1200 1400 100 200  5.9923890 6.3923890}%
\special{ar 1200 1400 100 200  6.6323890 7.0323890}%
\special{ar 1200 1400 100 200  7.2723890 7.6723890}%
}}%
%
{\color[named]{Black}{%
\special{pn 8}%
\special{ar 1200 1800 300 300  0.0000000 6.2831853}%
}}%
%
{\color[named]{Black}{%
\special{pn 8}%
\special{pa 3100 1100}%
\special{pa 3068 1104}%
\special{pa 3036 1110}%
\special{pa 3004 1114}%
\special{pa 2844 1134}%
\special{pa 2810 1138}%
\special{pa 2650 1152}%
\special{pa 2554 1158}%
\special{pa 2490 1160}%
\special{pa 2460 1160}%
\special{pa 2428 1162}%
\special{pa 2396 1160}%
\special{pa 2364 1160}%
\special{pa 2332 1160}%
\special{pa 2268 1156}%
\special{pa 2238 1154}%
\special{pa 2206 1150}%
\special{pa 2144 1142}%
\special{pa 2112 1138}%
\special{pa 2080 1132}%
\special{pa 1988 1114}%
\special{pa 1956 1108}%
\special{pa 1924 1100}%
\special{pa 1832 1076}%
\special{pa 1802 1068}%
\special{pa 1708 1042}%
\special{pa 1678 1032}%
\special{pa 1646 1022}%
\special{pa 1616 1012}%
\special{pa 1586 1002}%
\special{pa 1554 992}%
\special{pa 1524 982}%
\special{pa 1492 972}%
\special{pa 1462 962}%
\special{pa 1432 950}%
\special{pa 1400 940}%
\special{sp}%
}}%
%
{\color[named]{Black}{%
\special{pn 8}%
\special{pa 1400 940}%
\special{pa 1432 938}%
\special{pa 1466 934}%
\special{pa 1498 930}%
\special{pa 1530 926}%
\special{pa 1594 920}%
\special{pa 1626 916}%
\special{pa 1658 914}%
\special{pa 1690 912}%
\special{pa 1754 906}%
\special{pa 1788 904}%
\special{pa 1916 896}%
\special{pa 1980 894}%
\special{pa 2012 894}%
\special{pa 2044 892}%
\special{pa 2076 892}%
\special{pa 2106 894}%
\special{pa 2138 894}%
\special{pa 2170 894}%
\special{pa 2202 896}%
\special{pa 2234 898}%
\special{pa 2266 900}%
\special{pa 2298 904}%
\special{pa 2328 906}%
\special{pa 2360 910}%
\special{pa 2392 914}%
\special{pa 2422 920}%
\special{pa 2454 924}%
\special{pa 2548 942}%
\special{pa 2640 964}%
\special{pa 2672 972}%
\special{pa 2702 980}%
\special{pa 2734 988}%
\special{pa 2764 996}%
\special{pa 2858 1024}%
\special{pa 2888 1034}%
\special{pa 2920 1042}%
\special{pa 2950 1052}%
\special{pa 2980 1062}%
\special{pa 3012 1072}%
\special{pa 3074 1092}%
\special{pa 3100 1100}%
\special{sp 0.070}%
}}%
%
{\color[named]{Black}{%
\special{pn 8}%
\special{pa 3100 1700}%
\special{pa 3066 1702}%
\special{pa 3032 1706}%
\special{pa 3000 1708}%
\special{pa 2966 1710}%
\special{pa 2898 1714}%
\special{pa 2864 1716}%
\special{pa 2830 1718}%
\special{pa 2798 1720}%
\special{pa 2764 1722}%
\special{pa 2664 1724}%
\special{pa 2566 1724}%
\special{pa 2534 1724}%
\special{pa 2500 1722}%
\special{pa 2468 1722}%
\special{pa 2436 1720}%
\special{pa 2372 1714}%
\special{pa 2342 1710}%
\special{pa 2310 1706}%
\special{pa 2248 1696}%
\special{pa 2218 1690}%
\special{pa 2186 1684}%
\special{pa 2156 1678}%
\special{pa 2126 1670}%
\special{pa 2096 1660}%
\special{pa 2068 1652}%
\special{pa 2010 1632}%
\special{pa 1954 1608}%
\special{pa 1898 1580}%
\special{pa 1870 1566}%
\special{pa 1816 1534}%
\special{pa 1764 1500}%
\special{pa 1738 1482}%
\special{pa 1712 1462}%
\special{pa 1686 1444}%
\special{pa 1662 1424}%
\special{pa 1636 1404}%
\special{pa 1610 1384}%
\special{pa 1560 1342}%
\special{pa 1536 1320}%
\special{pa 1512 1298}%
\special{pa 1488 1276}%
\special{pa 1462 1254}%
\special{pa 1438 1232}%
\special{pa 1414 1210}%
\special{pa 1390 1186}%
\special{pa 1366 1164}%
\special{pa 1350 1150}%
\special{sp}%
}}%
%
{\color[named]{Black}{%
\special{pn 8}%
\special{pa 1350 1150}%
\special{pa 1412 1166}%
\special{pa 1444 1176}%
\special{pa 1536 1200}%
\special{pa 1568 1208}%
\special{pa 1598 1216}%
\special{pa 1628 1224}%
\special{pa 1660 1234}%
\special{pa 1690 1242}%
\special{pa 1722 1250}%
\special{pa 1752 1258}%
\special{pa 1784 1268}%
\special{pa 1814 1276}%
\special{pa 1844 1284}%
\special{pa 1876 1292}%
\special{pa 1968 1320}%
\special{pa 1998 1328}%
\special{pa 2060 1346}%
\special{pa 2090 1356}%
\special{pa 2152 1374}%
\special{pa 2182 1384}%
\special{pa 2214 1392}%
\special{pa 2244 1402}%
\special{pa 2274 1412}%
\special{pa 2304 1422}%
\special{pa 2336 1432}%
\special{pa 2396 1452}%
\special{pa 2426 1462}%
\special{pa 2486 1482}%
\special{pa 2518 1492}%
\special{pa 2548 1502}%
\special{pa 2578 1512}%
\special{pa 2608 1524}%
\special{pa 2638 1534}%
\special{pa 2698 1556}%
\special{pa 2728 1566}%
\special{pa 2758 1576}%
\special{pa 2790 1588}%
\special{pa 2820 1598}%
\special{pa 2850 1608}%
\special{pa 3090 1696}%
\special{pa 3100 1700}%
\special{sp 0.070}%
}}%
\put(33.8000,-14.6000){\makebox(0,0)[lb]{$\partial N$}}%
\put(33.8000,-20.6000){\makebox(0,0)[lb]{$\partial N$}}%
\end{picture}%

%% file: chi-3.tex
\unitlength 0.1in
\begin{picture}( 22.0000, 12.0000)(  4.0000,-14.0000)
%
{\color[named]{Black}{%
\special{pn 13}%
\special{ar 800 800 400 600  1.5707963 4.7123890}%
}}%
%
{\color[named]{Black}{%
\special{pn 13}%
\special{pa 800 200}%
\special{pa 2200 200}%
\special{fp}%
}}%
%
{\color[named]{Black}{%
\special{pn 13}%
\special{ar 2200 800 400 600  4.7123890 6.2831853}%
\special{ar 2200 800 400 600  0.0000000 1.5707963}%
}}%
%
{\color[named]{Black}{%
\special{pn 13}%
\special{pa 2200 1400}%
\special{pa 800 1400}%
\special{fp}%
}}%
%
{\color[named]{Black}{%
\special{pn 4}%
\special{sh 1}%
\special{ar 700 800 26 26 0  6.28318530717959E+0000}%
}}%
%
{\color[named]{Black}{%
\special{pn 4}%
\special{sh 1}%
\special{ar 1100 800 26 26 0  6.28318530717959E+0000}%
}}%
%
{\color[named]{Black}{%
\special{pn 13}%
\special{ar 1700 800 100 100  0.0000000 6.2831853}%
}}%
%
{\color[named]{Black}{%
\special{pn 13}%
\special{ar 2200 800 100 100  0.0000000 6.2831853}%
}}%
%
{\color[named]{Black}{%
\special{pn 8}%
\special{ar 700 800 200 200  1.5707963 4.7123890}%
}}%
%
{\color[named]{Black}{%
\special{pn 8}%
\special{ar 2200 800 200 200  4.7123890 6.2831853}%
\special{ar 2200 800 200 200  0.0000000 1.5707963}%
}}%
%
{\color[named]{Black}{%
\special{pn 8}%
\special{pa 1100 600}%
\special{pa 700 600}%
\special{fp}%
}}%
%
{\color[named]{Black}{%
\special{pn 8}%
\special{pa 1100 1000}%
\special{pa 700 1000}%
\special{fp}%
}}%
%
{\color[named]{Black}{%
\special{pn 8}%
\special{ar 1100 800 200 200  4.7123890 6.2831853}%
\special{ar 1100 800 200 200  0.0000000 1.5707963}%
}}%
%
{\color[named]{Black}{%
\special{pn 8}%
\special{ar 1700 800 200 200  1.5707963 4.7123890}%
}}%
%
{\color[named]{Black}{%
\special{pn 8}%
\special{pa 1700 600}%
\special{pa 2200 600}%
\special{fp}%
}}%
%
{\color[named]{Black}{%
\special{pn 8}%
\special{pa 1700 1000}%
\special{pa 2200 1000}%
\special{fp}%
}}%
%
{\color[named]{Black}{%
\special{pn 8}%
\special{ar 1500 800 400 400  4.7123890 6.2831853}%
\special{ar 1500 800 400 400  0.0000000 1.5707963}%
}}%
%
{\color[named]{Black}{%
\special{pn 8}%
\special{ar 1300 800 400 400  1.5707963 4.7123890}%
}}%
%
{\color[named]{Black}{%
\special{pn 8}%
\special{pa 1300 400}%
\special{pa 1500 400}%
\special{fp}%
}}%
%
{\color[named]{Black}{%
\special{pn 8}%
\special{pa 1300 1200}%
\special{pa 1500 1200}%
\special{fp}%
}}%
\end{picture}%

%% file: figure-eight.tex
\unitlength 0.1in
\begin{picture}( 48.5000, 20.0000)(  1.5000,-24.0000)
%
{\color[named]{Black}{%
\special{pn 13}%
\special{ar 1400 1400 1000 1000  0.0000000 6.2831853}%
}}%
%
{\color[named]{Black}{%
\special{pn 13}%
\special{ar 1000 1400 200 200  0.0000000 6.2831853}%
}}%
%
{\color[named]{Black}{%
\special{pn 13}%
\special{ar 1800 1400 200 200  0.0000000 6.2831853}%
}}%
%
{\color[named]{Black}{%
\special{pn 4}%
\special{sh 1}%
\special{ar 400 1400 26 26 0  6.28318530717959E+0000}%
}}%
%
{\color[named]{Black}{%
\special{pn 4}%
\special{sh 1}%
\special{ar 1200 1400 26 26 0  6.28318530717959E+0000}%
}}%
%
{\color[named]{Black}{%
\special{pn 4}%
\special{sh 1}%
\special{ar 2000 1400 26 26 0  6.28318530717959E+0000}%
}}%
%
{\color[named]{Black}{%
\special{pn 8}%
\special{ar 1400 1400 900 900  3.3650693 6.2831853}%
\special{ar 1400 1400 900 900  0.0000000 2.9181161}%
}}%
%
{\color[named]{Black}{%
\special{pn 8}%
\special{pa 400 1400}%
\special{pa 520 1200}%
\special{fp}%
}}%
%
{\color[named]{Black}{%
\special{pn 8}%
\special{pa 400 1400}%
\special{pa 520 1600}%
\special{fp}%
}}%
%
{\color[named]{Black}{%
\special{pn 8}%
\special{ar 1000 1400 300 300  0.3430239 5.9401614}%
}}%
%
{\color[named]{Black}{%
\special{pn 8}%
\special{pa 1200 1400}%
\special{pa 1280 1300}%
\special{fp}%
}}%
%
{\color[named]{Black}{%
\special{pn 8}%
\special{pa 1200 1400}%
\special{pa 1280 1500}%
\special{fp}%
}}%
%
{\color[named]{Black}{%
\special{pn 8}%
\special{ar 1800 1400 300 300  0.3430239 5.9401614}%
}}%
%
{\color[named]{Black}{%
\special{pn 8}%
\special{pa 2000 1400}%
\special{pa 2080 1300}%
\special{fp}%
}}%
%
{\color[named]{Black}{%
\special{pn 8}%
\special{pa 2000 1400}%
\special{pa 2080 1500}%
\special{fp}%
}}%
%
{\color[named]{Black}{%
\special{pn 8}%
\special{pa 1000 1700}%
\special{pa 1080 1640}%
\special{fp}%
\special{pa 1000 1700}%
\special{pa 1080 1750}%
\special{fp}%
}}%
%
{\color[named]{Black}{%
\special{pn 8}%
\special{pa 1800 1700}%
\special{pa 1880 1640}%
\special{fp}%
\special{pa 1800 1700}%
\special{pa 1880 1750}%
\special{fp}%
}}%
%
{\color[named]{Black}{%
\special{pn 8}%
\special{pa 1400 2300}%
\special{pa 1320 2240}%
\special{fp}%
\special{pa 1400 2300}%
\special{pa 1320 2350}%
\special{fp}%
}}%
%
{\color[named]{Black}{%
\special{pn 13}%
\special{ar 4000 1400 1000 1000  0.0000000 6.2831853}%
}}%
%
{\color[named]{Black}{%
\special{pn 13}%
\special{ar 3600 1400 200 200  0.0000000 6.2831853}%
}}%
%
{\color[named]{Black}{%
\special{pn 13}%
\special{ar 4400 1400 200 200  0.0000000 6.2831853}%
}}%
%
{\color[named]{Black}{%
\special{pn 4}%
\special{sh 1}%
\special{ar 3000 1400 26 26 0  6.28318530717959E+0000}%
}}%
%
{\color[named]{Black}{%
\special{pn 4}%
\special{sh 1}%
\special{ar 3800 1400 26 26 0  6.28318530717959E+0000}%
}}%
%
{\color[named]{Black}{%
\special{pn 4}%
\special{sh 1}%
\special{ar 4600 1400 26 26 0  6.28318530717959E+0000}%
}}%
%
{\color[named]{Black}{%
\special{pn 8}%
\special{pa 3800 1400}%
\special{pa 3834 1404}%
\special{pa 3866 1404}%
\special{pa 3896 1402}%
\special{pa 3924 1392}%
\special{pa 3950 1376}%
\special{pa 3972 1358}%
\special{pa 3994 1334}%
\special{pa 4014 1308}%
\special{pa 4034 1280}%
\special{pa 4056 1252}%
\special{pa 4078 1226}%
\special{pa 4102 1200}%
\special{pa 4128 1176}%
\special{pa 4156 1156}%
\special{pa 4186 1138}%
\special{pa 4216 1124}%
\special{pa 4248 1112}%
\special{pa 4280 1104}%
\special{pa 4312 1098}%
\special{pa 4344 1096}%
\special{pa 4374 1098}%
\special{pa 4406 1102}%
\special{pa 4434 1108}%
\special{pa 4462 1118}%
\special{pa 4490 1132}%
\special{pa 4518 1146}%
\special{pa 4546 1164}%
\special{pa 4574 1184}%
\special{pa 4606 1204}%
\special{pa 4638 1226}%
\special{pa 4666 1248}%
\special{pa 4688 1272}%
\special{pa 4700 1294}%
\special{pa 4696 1318}%
\special{pa 4680 1340}%
\special{pa 4656 1362}%
\special{pa 4624 1386}%
\special{pa 4600 1400}%
\special{sp}%
}}%
%
{\color[named]{Black}{%
\special{pn 8}%
\special{pa 3000 1400}%
\special{pa 3034 1404}%
\special{pa 3066 1404}%
\special{pa 3096 1402}%
\special{pa 3124 1392}%
\special{pa 3150 1376}%
\special{pa 3172 1358}%
\special{pa 3194 1334}%
\special{pa 3214 1308}%
\special{pa 3234 1280}%
\special{pa 3256 1252}%
\special{pa 3278 1226}%
\special{pa 3302 1200}%
\special{pa 3328 1176}%
\special{pa 3356 1156}%
\special{pa 3386 1138}%
\special{pa 3416 1124}%
\special{pa 3448 1112}%
\special{pa 3480 1104}%
\special{pa 3512 1098}%
\special{pa 3544 1096}%
\special{pa 3574 1098}%
\special{pa 3606 1102}%
\special{pa 3634 1108}%
\special{pa 3662 1118}%
\special{pa 3690 1132}%
\special{pa 3718 1146}%
\special{pa 3746 1164}%
\special{pa 3774 1184}%
\special{pa 3806 1204}%
\special{pa 3838 1226}%
\special{pa 3866 1248}%
\special{pa 3888 1272}%
\special{pa 3900 1294}%
\special{pa 3896 1318}%
\special{pa 3880 1340}%
\special{pa 3856 1362}%
\special{pa 3824 1386}%
\special{pa 3800 1400}%
\special{sp}%
}}%
%
{\color[named]{Black}{%
\special{pn 8}%
\special{pa 3000 1400}%
\special{pa 3026 1380}%
\special{pa 3072 1336}%
\special{pa 3092 1312}%
\special{pa 3110 1286}%
\special{pa 3126 1258}%
\special{pa 3142 1228}%
\special{pa 3154 1200}%
\special{pa 3168 1170}%
\special{pa 3180 1140}%
\special{pa 3194 1112}%
\special{pa 3210 1084}%
\special{pa 3226 1060}%
\special{pa 3246 1036}%
\special{pa 3266 1014}%
\special{pa 3288 994}%
\special{pa 3310 974}%
\special{pa 3336 956}%
\special{pa 3362 940}%
\special{pa 3388 924}%
\special{pa 3416 910}%
\special{pa 3446 898}%
\special{pa 3476 886}%
\special{pa 3506 876}%
\special{pa 3538 866}%
\special{pa 3572 856}%
\special{pa 3604 848}%
\special{pa 3638 842}%
\special{pa 3708 830}%
\special{pa 3744 824}%
\special{pa 3780 820}%
\special{pa 3816 816}%
\special{pa 3852 812}%
\special{pa 3888 808}%
\special{pa 3924 806}%
\special{pa 3960 804}%
\special{pa 3996 800}%
\special{pa 4068 796}%
\special{pa 4104 794}%
\special{pa 4174 792}%
\special{pa 4208 792}%
\special{pa 4242 794}%
\special{pa 4276 794}%
\special{pa 4308 798}%
\special{pa 4340 800}%
\special{pa 4372 806}%
\special{pa 4404 812}%
\special{pa 4464 828}%
\special{pa 4492 838}%
\special{pa 4520 850}%
\special{pa 4546 864}%
\special{pa 4572 880}%
\special{pa 4598 898}%
\special{pa 4622 918}%
\special{pa 4644 940}%
\special{pa 4668 964}%
\special{pa 4690 988}%
\special{pa 4710 1012}%
\special{pa 4730 1038}%
\special{pa 4750 1064}%
\special{pa 4768 1092}%
\special{pa 4782 1120}%
\special{pa 4792 1150}%
\special{pa 4798 1182}%
\special{pa 4800 1218}%
\special{pa 4798 1252}%
\special{pa 4790 1288}%
\special{pa 4778 1322}%
\special{pa 4760 1352}%
\special{pa 4740 1376}%
\special{pa 4714 1394}%
\special{pa 4686 1404}%
\special{pa 4654 1406}%
\special{pa 4622 1404}%
\special{pa 4600 1400}%
\special{sp}%
}}%
%
{\color[named]{Black}{%
\special{pn 8}%
\special{pa 3580 1100}%
\special{pa 3500 1040}%
\special{fp}%
\special{pa 3580 1100}%
\special{pa 3500 1150}%
\special{fp}%
}}%
%
{\color[named]{Black}{%
\special{pn 8}%
\special{pa 4380 1100}%
\special{pa 4300 1040}%
\special{fp}%
\special{pa 4380 1100}%
\special{pa 4300 1150}%
\special{fp}%
}}%
%
{\color[named]{Black}{%
\special{pn 8}%
\special{pa 3860 810}%
\special{pa 3940 750}%
\special{fp}%
\special{pa 3860 810}%
\special{pa 3940 860}%
\special{fp}%
}}%
\put(1.5000,-14.2000){\makebox(0,0)[lb]{$*_3$}}%
\put(27.5000,-14.2000){\makebox(0,0)[lb]{$*_3$}}%
\put(10.0000,-14.2000){\makebox(0,0)[lb]{$*_1$}}%
\put(36.0000,-14.2000){\makebox(0,0)[lb]{$*_1$}}%
\put(18.0000,-14.2000){\makebox(0,0)[lb]{$*_2$}}%
\put(44.0000,-14.2000){\makebox(0,0)[lb]{$*_2$}}%
\put(9.5000,-18.8000){\makebox(0,0)[lb]{$\eta_1$}}%
\put(9.5000,-18.8000){\makebox(0,0)[lb]{$\eta_1$}}%
\put(17.5000,-18.8000){\makebox(0,0)[lb]{$\eta_2$}}%
\put(13.4000,-22.1000){\makebox(0,0)[lb]{$\eta_3$}}%
\put(39.6000,-7.1000){\makebox(0,0)[lb]{$\gamma_1$}}%
\put(36.9000,-10.9000){\makebox(0,0)[lb]{$\gamma_2$}}%
\put(44.3000,-10.9000){\makebox(0,0)[lb]{$\gamma_3$}}%
\end{picture}%

%% file: capping.tex
\unitlength 0.1in
\begin{picture}( 40.5000, 19.1000)(  2.0000,-22.0000)
%
{\color[named]{Black}{%
\special{pn 13}%
\special{ar 4000 1600 100 600  0.0000000 6.2831853}%
}}%
%
{\color[named]{Black}{%
\special{pn 13}%
\special{pa 4000 2200}%
\special{pa 800 2200}%
\special{fp}%
}}%
%
{\color[named]{Black}{%
\special{pn 13}%
\special{ar 800 1600 600 600  1.5707963 4.7123890}%
}}%
%
{\color[named]{Black}{%
\special{pn 13}%
\special{ar 600 1600 150 150  0.0000000 6.2831853}%
}}%
%
{\color[named]{Black}{%
\special{pn 13}%
\special{ar 1600 1600 150 150  0.0000000 6.2831853}%
}}%
%
{\color[named]{Black}{%
\special{pn 13}%
\special{ar 1800 900 100 100  6.2831853 6.2831853}%
\special{ar 1800 900 100 100  0.0000000 1.5707963}%
}}%
%
{\color[named]{Black}{%
\special{pn 13}%
\special{ar 2500 900 100 100  1.5707963 3.1415927}%
}}%
%
{\color[named]{Black}{%
\special{pn 13}%
\special{ar 2800 900 100 100  6.2831853 6.2831853}%
\special{ar 2800 900 100 100  0.0000000 1.5707963}%
}}%
%
{\color[named]{Black}{%
\special{pn 13}%
\special{ar 3500 900 100 100  1.5707963 3.1415927}%
}}%
%
{\color[named]{Black}{%
\special{pn 13}%
\special{pa 1900 900}%
\special{pa 1900 600}%
\special{fp}%
}}%
%
{\color[named]{Black}{%
\special{pn 13}%
\special{pa 2400 900}%
\special{pa 2400 600}%
\special{fp}%
}}%
%
{\color[named]{Black}{%
\special{pn 13}%
\special{pa 2900 900}%
\special{pa 2900 600}%
\special{fp}%
}}%
%
{\color[named]{Black}{%
\special{pn 13}%
\special{pa 3400 900}%
\special{pa 3400 600}%
\special{fp}%
}}%
%
{\color[named]{Black}{%
\special{pn 13}%
\special{ar 2150 600 250 250  3.1415927 6.2831853}%
}}%
%
{\color[named]{Black}{%
\special{pn 13}%
\special{ar 3150 600 250 250  3.1415927 6.2831853}%
}}%
%
{\color[named]{Black}{%
\special{pn 13}%
\special{pa 3500 1000}%
\special{pa 4000 1000}%
\special{fp}%
}}%
%
{\color[named]{Black}{%
\special{pn 13}%
\special{pa 2800 1000}%
\special{pa 2500 1000}%
\special{fp}%
}}%
%
{\color[named]{Black}{%
\special{pn 13}%
\special{pa 800 1000}%
\special{pa 1800 1000}%
\special{fp}%
}}%
%
{\color[named]{Black}{%
\special{pn 13}%
\special{ar 2160 700 110 110  0.0000000 6.2831853}%
}}%
%
{\color[named]{Black}{%
\special{pn 13}%
\special{ar 3160 700 110 110  0.0000000 6.2831853}%
}}%
\put(9.3000,-17.0000){\makebox(0,0)[lb]{$\cdots$}}%
%
{\color[named]{Black}{%
\special{pn 8}%
\special{ar 2150 920 250 50  6.2831853 6.2831853}%
\special{ar 2150 920 250 50  0.0000000 3.1415927}%
}}%
%
{\color[named]{Black}{%
\special{pn 8}%
\special{ar 2150 920 250 50  3.1415927 3.5415927}%
\special{ar 2150 920 250 50  3.7815927 4.1815927}%
\special{ar 2150 920 250 50  4.4215927 4.8215927}%
\special{ar 2150 920 250 50  5.0615927 5.4615927}%
\special{ar 2150 920 250 50  5.7015927 6.1015927}%
}}%
%
{\color[named]{Black}{%
\special{pn 8}%
\special{ar 3150 920 250 50  6.2831853 6.2831853}%
\special{ar 3150 920 250 50  0.0000000 3.1415927}%
}}%
%
{\color[named]{Black}{%
\special{pn 8}%
\special{ar 3150 920 250 50  3.1415927 3.5415927}%
\special{ar 3150 920 250 50  3.7815927 4.1815927}%
\special{ar 3150 920 250 50  4.4215927 4.8215927}%
\special{ar 3150 920 250 50  5.0615927 5.4615927}%
\special{ar 3150 920 250 50  5.7015927 6.1015927}%
}}%
\put(20.5000,-12.2000){\makebox(0,0)[lb]{$\partial_1S$}}%
\put(30.5000,-12.2000){\makebox(0,0)[lb]{$\partial_2S$}}%
\put(42.5000,-16.2000){\makebox(0,0)[lb]{$\partial_rS$}}%
\put(38.5000,-4.2000){\makebox(0,0)[lb]{$S^{\prime}$}}%
\put(4.1000,-9.0000){\makebox(0,0)[lb]{$\overbrace{\ \hspace{3cm}\ }$}}%
\put(10.0000,-7.0000){\makebox(0,0)[lb]{$g$}}%
\end{picture}%